\documentclass[11pt]{article}
\usepackage{graphicx} 
\usepackage{bbm}
\usepackage{bm}
\usepackage[margin=1in]{geometry}
\usepackage{color}
\usepackage{amsfonts}
\usepackage{algorithm}
\usepackage{algorithmic}
\usepackage{latexsym, amssymb, amsmath, amscd, amsthm, amsxtra}
\usepackage{mathtools}
\usepackage{enumerate}
\usepackage[all]{xy}
\usepackage{mathrsfs}
\usepackage{fancyhdr}
\usepackage{listings}
\usepackage{hyperref}
\usepackage{enumitem}
\usepackage{multirow}
\usepackage{tikz}

\def\11{\mathbbm{1}}

\def\ER{Erd\H{o}s-R\'enyi\ }

\def\Fop{\operatorname{F}}

\newcommand{\Pb}{\mathbb P}
\newcommand{\Qb}{\mathbb Q}

\newcommand{\doublesetminus}
{\mathbin{\setminus\mkern-5mu\setminus}}

\newtheorem{thm}{Theorem}[section]

\newtheorem{proposition}[thm]{Proposition}

\newtheorem{conjecture}[thm]{Heuristic}
\newtheorem{lemma}[thm]{Lemma}

\newtheorem{defn}[thm]{Definition}

\newtheorem{remark}[thm]{Remark}
\newtheorem{DEF}[thm]{Definition}

\numberwithin{equation}{section}



\hypersetup{colorlinks=true,linkcolor=blue}

\title{Algorithmic Contiguity from Low-Degree Heuristic II: Predicting Detection-Recovery Gaps}

\usepackage{authblk}
\author[1]{Zhangsong Li\thanks{Email: \textit{ramblerlzs@pku.edu.cn}. Partially supported by National Key R$\&$D Program of China  (Project No. 2023YFA1010102) and NSFC Key Program (Project No.12231002)}}
\affil[1]{School of Mathematical Sciences, Peking University}

\date{\today}
\begin{document}
\maketitle

\begin{abstract}
    The low-degree polynomial framework has emerged as a powerful tool for providing evidence of statistical-computational gaps in high-dimensional inference. For detection problems, the standard approach bounds the low-degree advantage through an explicit orthonormal basis. However, this method does not extend naturally to estimation tasks, and thus fails to capture the \emph{detection-recovery gap phenomenon} that arises in many high-dimensional problems. Although several important advances have been made to overcome this limitation \cite{SW22, SW25, CGGV25+}, the existing approaches often rely on delicate, model-specific combinatorial arguments.

    In this work, we develop a general approach for obtaining \emph{conditional computational lower bounds} for recovery problems from mild bounds on low-degree testing advantage. Our method combines the notion of algorithmic contiguity in \cite{Li25} with a cross-validation reduction in \cite{DHSS25} that converts successful recovery into a hypothesis test with lopsided success probabilities. This yields a general principle: if the planted and null models have a mild bound on low-degree advantage, then efficient recovery would imply an efficient test for a related problem ruled out by the low-degree heuristic. In contrast to prior unconditional lower bounds, our argument is conceptually simple, flexible, and largely model-independent.

    We apply this framework to several canonical inference problems, including planted submatrix, planted dense subgraph, stochastic block model, multi-frequency angular synchronization, orthogonal group synchronization, and multi-layer stochastic block model. In the first three settings, our method recovers existing low-degree lower bounds for recovery in \cite{SW22, SW25} via a substantially simpler argument. In the latter three, it gives new evidence for conjectured computational thresholds including the persistence of detection-recovery gaps. Together, these results suggest that mild control of low-degree advantage is often sufficient to explain computational barriers for recovery in high-dimensional statistical models.
\end{abstract}

\tableofcontents

\section{Introduction}{\label{sec:intro}}

The task of recovering a hidden signal from large, noisy data is central to modern statistics and data science. Beyond the statistical question of identifying the weakest signal that can be reliably recovered, these high-dimensional problems also raise a computational question: whether such signals can be recovered by algorithms with practical running time. As a testbed for understanding the fundamental limits of statistical and computational efficiency, we consider the following canonical models of planted signals in random matrices and random graphs.
\begin{itemize}
    \item {\bf Planted submatrix:} For a sparsity parameter $\rho \in [0,1]$ and a signal-to-noise parameter $\lambda \geq 0$, we observe an $n*n$ matrix $\bm Y = \lambda\theta\theta^{\top} + \bm Z$, where $\theta \in \mathbb R^n$ has i.i.d.\ $\mathsf{Ber}(\rho)$ entries and $\bm Z\in\mathbb R^{n*n}$ is an $n*n$ symmetric random matrix whose entries $\{ \bm Z_{i,j}: i\leq j \}$ are independent normal variables with variance $1$. The goal is to estimate $\theta$ from the observation $\bm Y$.
    \item {\bf Planted dense subgraph:} For a sparsity parameter $\rho \in [0,1]$ and edge density parameters $0 \le p_0 \le p_1 \le 1$, we observe a random graph on $n$ vertices with adjacency matrix $\bm Y=(\bm Y_{i,j})_{1 \leq i<j \leq n}$ generated as follows. First draw $\theta\in\mathbb R^n$ with i.i.d.\ $\mathsf{Ber}(\rho)$ entries to indicate membership in the planted subgraph. Conditioned on $\theta$, draw $\bm Y_{i,j} \sim \mathsf{Ber}(p_0+(p_1-p_0)\theta_i\theta_j)$ independently for each $i<j$. Thus, edges within the planted subgraph have probability $p_1$ while others have probability $p_0$. The goal is to estimate $\theta$ from the observation $\bm Y$. 
    \item {\bf Stochastic block model:} For an integer community number parameter $q\geq 2$ and edge density parameters $0\leq p_0\leq p_1 \leq 1$, we observe a random graph on $n$ vertices with adjacency matrix $\bm Y=(\bm Y_{i,j})_{1 \leq i<j \leq n}$ generated as follows. First, each vertex $i \in [n]:=\{ 1,\ldots,n \}$ is independently assigned a community label $\sigma_i$, drawn uniformly from $[q]=\{ 1,\ldots,q \}$. Conditioned on the labeling $\sigma$, draw $\bm Y_{i,j} \sim \mathsf{Ber}(p_0+(p_1-p_0)\mathbf 1_{\sigma_i=\sigma_j})$ independently for each $i<j$. Equivalently, within-community edges are present with probability $p_1$, and cross-community edges with probability $p_0$. The goal is to estimate whether two given vertices are in the same community from the observation $\bm Y$.
    \item {\bf Multi-frequency angular synchronization:} For an integer frequency parameter $L \in \mathbb N$ and a signal-to-noise ratio $\lambda \ge 0$, we observe $L$ matrices $(\bm Y_1,\ldots,\bm Y_L)$ generated as follows. Let $\bm x \in \mathbb C^n$ have i.i.d.\ entries drawn uniformly from the complex circle group $\{ e^{i\varphi}: \varphi\in[0,2\pi) \}$. For each $1 \le \ell \le L$, the observation matrix is given by $\bm Y_{\ell} = \lambda \bm x^{(\ell)} (\bm x^{(\ell)})^{*} + \bm W_{\ell}$ for $1 \leq \ell \leq L$. Here $\bm x^{(k)}$ denotes the entrywise $k$-th power, $\bm x^{*}=\overline{\bm x}^{\top}$ denotes the complex conjugate transpose, and $\bm W_1,\ldots,\bm W_L$ are independent Hermitian matrices such that $\{ (\bm W_{\ell})_{i,j}:i<j \}$ are i.i.d.\ complex standard Gaussian random variables (that is, a variable $x+iy$ where $x,y \sim \mathcal N(0,\frac{1}{2})$ are independent). The goal is to estimate the relative phase between two given coordinates of $\bm x$ from the observation $(\bm Y_1,\ldots,\bm Y_L)$.
    \item {\bf Orthogonal group synchronization:} For a dimension parameter $d \in \mathbb N$ and a signal-to-noise ratio parameter $\lambda \ge 0$, we observe an $nd*nd$ matrix $\bm Y$ generated as follows. Let $\bm U=(\bm O_1,\ldots,\bm O_n) \in \mathbb R^{d*nd}$ where $\bm O_1,\ldots,\bm O_n$ are drawn independently from the uniform distribution on the orthogonal group $\mathbb O(d)$. We then observe $\bm Y= \lambda \bm U^{\top}\bm U +\bm W$ where $\bm W \in \mathbb R^{nd*nd}$ is a noise matrix with i.i.d.\ standard normal entries. Equivalently, viewing $\bm Y=(\bm Y_{i,j})_{1 \le i,j \le n}$ as a block matrix with $d*d$ blocks, we have $\bm Y_{i,j}=\lambda \bm O_i^{\top} \bm O_j + \bm W_{i,j}$. The goal is to estimate $\bm O_i^{\top} \bm O_j$ from the observation $\bm Y$.
    \item {\bf Multi-layer stochastic block model:} For a correlation parameter $\rho \in [0,1]$ and number of layers $L \in \mathbb N$, we observe $L$ related stochastic block models on the same set of $n$ vertices, generated as follows. First, draw a latent community labeling $\sigma \in [q]^n$ uniformly at random, where $q$ is the number of communities. Conditioned on $\sigma$, draw independent layer-specific labelings $\sigma_1,\ldots,\sigma_L \in [q]^n$ such that, for each vertex $i \in [n]$ and layer $\ell \in [L]$, we have $\sigma_\ell(i)=\sigma(i)$ with probability $\frac{1+(q-1)\rho}{q}$, while each of the other $q-1$ labels is assigned with probability $\frac{1-\rho}{q}$. For edge probabilities $0 \le p_0 \le p_1 \le 1$, we then generate $L$ random graphs on $n$ vertices with adjacency matrices $\bm Y_1,\ldots,\bm Y_L$, independently across layers conditional on $\sigma_1,\ldots,\sigma_L$. Specifically, for each layer $\ell \in [L]$ and each pair $1 \le i < j \le n$, we draw $\bm Y_{i,j} \sim \mathsf{Ber}(p_0+(p_1-p_0)\mathbf 1_{\sigma_\ell(i)=\sigma_\ell(j)})$. The goal is to estimate whether two given vertices are in the same community from the observation $\bm Y_1,\ldots,\bm Y_L$.
\end{itemize}
In each of these models, we consider the asymptotic regime $n \to \infty$, while the remaining parameters may either vary with $n$ in a prescribed way or remain fixed constants independent of $n$. Throughout, asymptotic notation such as $O(\cdot)$, $o(\cdot)$, $\Omega(\cdot)$, $\omega(\cdot)$, and $\Theta(\cdot)$ is understood with respect to this limit. The goal is to estimate or recover the planted signal (for example, $\theta$) to a prescribed level of accuracy, measured either by mean squared error or by the attainment of a suitable success criterion with high probability, namely with probability $1-o(1)$ as $n \to \infty$. We assume that the model parameters (such as $\lambda$ and $\rho$) are known to the statistician, whereas the latent variables (such as $\theta$) are unknown.

The models defined above have each been extensively studied, and we defer a comprehensive review of the literature to Section~\ref{sec:main-results}. A significant body of prior work has focused on determining the statistical limits of these problems, namely, characterizing the parameter regimes in which successful estimation is possible or impossible when no computational constraints are imposed. Concurrently, another line of research seeks estimators that can be implemented \emph{efficiently}, typically in polynomial time. Notably, all of the models described above are believed to exhibit \emph{statistical–computational gaps} in the sense that there exists a regime of parameters where estimation is statistically possible (for instance via exhaustive search) yet no polynomial-time algorithm is known to succeed. In this setting, a central question is whether this apparent computational hardness is intrinsic. Moreover, the best known efficient algorithms for these models often exhibit sharp phase transitions: the problem becomes abruptly tractable once an appropriate signal-to-noise parameter exceeds a critical threshold. The goal of this work is to establish matching lower bounds, showing that below this threshold no polynomial-time algorithm can succeed.  

A key challenge in this endeavor is that classical notions of complexity such as NP-hardness are not applicable here, since we are dealing with \emph{average-case} problems where the goal is to succeed for (not all but) ``most'' random inputs from a particular distribution. In light of this, currently the leading approaches for {\em demonstrating} hardness are based on either average-case reductions which formally relate different average-case problems to each other (see, e.g., \cite{BBH18, BB20} and references therein) or based on lower bounds against restricted classes of algorithms (see e.g., \cite{BPW18, Gamarnik21}). For the latter purpose, the framework of \emph{low-degree polynomials} emerges as a natural and compelling choice. The basic idea is to study a restricted class of algorithms consisting of multivariate polynomials $f \in \mathbb R[\bm Y]$ in the input variables $\bm Y_{i,j}$, with degree at most some parameter $D$ which may itself scale with $n$. Indeed, it has been proved that the class of low-degree polynomial algorithms is a useful proxy for computationally efficient algorithms, in the sense that the best-known polynomial-time algorithms for a wide variety of high-dimensional inference problems are captured by the low-degree class; see, for example, \cite{Hopkins18, KWB22, Wein25+}. Furthermore, it is conjectured in \cite{Hopkins18} that the failure of degree-$D$ polynomial algorithms implies the failure of all ``robust'' algorithms with running time $n^{\widetilde{O}(D)}$ (here $\widetilde{O}$ means having at most this order up to a $\operatorname{poly}\log(n)$ factor). While recent work \cite{BHJK25} has uncovered counterexamples to this conjecture, the low-degree framework remains a highly valuable tool for predicting average-case computational lower bounds. The original studies in this area focused on hypothesis testing (detection) problems \cite{HKP+17, HS17}, which reduce to bounding the degree-$D$ advantage between two probability measures,
\begin{align*}
    \mathsf{Adv}_{\leq D}(\Pb;\Qb):= \sup_{ \substack{ f \in \mathbb R[\bm Y] \\ \operatorname{deg}(f) \leq D } } \left\{ \frac{ \mathbb E_{\Pb}[f] }{ \sqrt{\mathbb E_{\Qb}[f^2]} } \right\} \,,
\end{align*}
where $\Pb$ is the planted measure (for instance, the law of $\bm Y$ sampled from a planted submatrix) and $\Qb$ is a suitably chosen ``null'' measure (for instance, the law of a symmetric Gaussian matrix). It is fair to say that we now have a relatively complete toolkit for bounding $\mathsf{Adv}_{\leq D}(\Pb;\Qb)$ in most settings, at least when $\Pb$ and $\Qb$ are of a ``planted-versus-null'' nature.

However, our work differs from most prior research on low-degree complexity in that we study the estimation (recovery) problem rather than the testing (detection) problem. The estimation setting has also received considerable attention \cite{SW22, LZ22, Wein23, MW25, LG24, MWZ23, KMW24, HM25, SW25, CMSW25, CGGV25+}, but it is inherently more difficult to analyze, and the available mathematical toolbox remains less developed. Although hardness results for estimation can sometimes be obtained by studying the corresponding testing problem (see, e.g., \cite{DHSS25, Li25}), this approach does not always provide a complete understanding of the difficulty of estimation --- particularly for the problems considered here --- because the testing and estimation thresholds may differ. It is therefore essential to develop tools that address the estimation task directly. Recent works \cite{SW22, Wein23, SW25} proposed attacking estimation through the auxiliary detection problem
\begin{align*}
    \mathcal H_1: \bm Y \sim \Pb(\cdot\mid x) \mbox{ v.s. } \mathcal H_0: \bm Y \sim \Pb(\cdot) \,,
\end{align*}
where $x$ is a scalar quantity we want to estimate (for instance, in the planted submatrix problem, a natural choice is $x=\theta_1$, the first entry of the signal vector). This boils down to controlling the degree-$D$ minimum mean squared error,
\begin{align*}
    \mathsf{MMSE}_{\leq D}:= \inf_{ \substack{ f \in \mathbb R[\bm Y] \\ \operatorname{deg}(f) \leq D } } \left\{ \mathbb E\left[ (f(\bm Y)-x)^2 \right] \right\} \,.
\end{align*}
In a recent breakthrough \cite{SW25} (building on a series of earlier advances \cite{SW22, Wein23}), the authors proposed a general strategy for proving low-degree lower bounds for recovery problems by bounding $\mathsf{MMSE}_{\leq D}$. This framework has been successfully applied to derive sharp computational thresholds in the planted submatrix, planted dense subgraph, stochastic block model, and spiked matrix model. However, its applicability in more complex settings can be limited, as it requires constructing special solutions of a large overdetermined system, a task that appears to inevitably involve sophisticated and case-dependent combinatorial arguments. An alternative and simpler framework for proving lower bounds on $\mathsf{MMSE}_{\leq D}$ was proposed in \cite{CGGV25+}. This approach relies on directly constructing an ``almost orthonormal basis'' under the planted measure. While conceptually elegant, this method appears unable to capture sharp computational thresholds such as the Kesten–Stigum threshold in the stochastic block model or the Baik-Ben Arous-P\'ech\'e (BBP) threshold in synchronization and spiked matrix models.

In this work, we present a new approach for providing evidence of computational hardness of estimation problems within the low-degree polynomial framework. Our proof is motivated by a series of recent advances on the information-theoretic side \cite{Du25+, GHL26+}, which exploit the fact that recovery is in fact \emph{considerably harder} than detection statistically, in the sense that recovery requires the chi-square divergence $\chi^2(\Pb\|\Qb)$ to be at least $e^{\Theta(n)}$, whereas detection only requires the chi-square divergence to be $\omega(1)$. Intuitively, this stems from the fact that a randomly guessed $\widehat{\theta}$ has a non-vanishing correlation with the true signal $\theta$ with probability only $e^{-\Theta(n)}$. The main contribution of this work is to develop a (weaker) \emph{algorithmic analogue} of the arguments in \cite{Du25+, GHL26+} within the low-degree polynomial framework. Instead of proving an unconditional lower bound on $\mathsf{MMSE}_{\leq D}$, we will prove a \emph{conditional lower bound} by reducing the estimation problem to a detection problem with lopsided success probability. This approach enables us to rule out efficient estimation using only a mild bound on $\mathsf{Adv}_{\leq D}(\Pb;\Qb)$. In addition, our proof suggests that ``an exponential growth of low-degree advantage'' is crucial for the success of estimation algorithms, and the absence of such growth might be the main reason for the emergence of detection-recovery gaps. We refer to Section~\ref{subsec:additional-discussion} for further discussion.

\subsection{Our contributions}{\label{subsec:contribution}}

\subsubsection{Example: planted submatrix}{\label{subsubsec:contribution-example}}
We begin with the planted submatrix model, which provides a simple illustration of our approach and already captures the main phenomenon studied in this paper. In this model, one observes a symmetric Gaussian matrix containing a planted principal submatrix of approximate size $\rho n*\rho n$ with elevated mean $\lambda$. The goal is to estimate the support of the planted submatrix. A convenient way to formulate the recovery task is through correlation with the rank-one matrix $\theta\theta^\top$, where $\theta\in {0,1}^n$ denotes the planted support vector.
\begin{DEF}{\label{def-recovery-planted-submatrix}}
    We say an estimator $\mathcal X:=\mathcal X(\bm Y) \in \mathbb R^{n*n}$ achieves \emph{weak recovery}, if 
    \begin{align}{\label{eq-def-weak-recovery-planted-submatrix}}
        \mathbb E_{\Pb}\Bigg[ \frac{ \langle \mathcal X, \theta \theta^{\top} \rangle }{ \| \mathcal X \|_{\Fop} \| \theta \theta^{\top} \|_{\Fop} } \Bigg] \geq c \mbox{ for some constant } c>0 \,.
    \end{align}
    Similarly, we say that $\mathcal X:=\mathcal X(\bm Y) \in \mathbb R^{n*n}$ achieves \emph{strong recovery}, if 
    \begin{align}{\label{eq-def-strong-recovery-planted-submatrix}}
        \mathbb E_{\Pb}\Bigg[ \frac{ \langle \mathcal X, \theta \theta^{\top} \rangle }{ \| \mathcal X \|_{\Fop} \| \theta \theta^{\top} \|_{\Fop} } \Bigg] \to 1 \mbox{ as } n \to \infty \,.
    \end{align}
\end{DEF}
We now state a simplified form of our result for this model.
\begin{thm}[Special case of Theorem~\ref{Main-thm-planted-submatrix}]{\label{Main-thm-planted-submatrix-special-case}}
    Consider the planted submatrix problem in the asymptotic regime $n\to\infty$, with $\lambda=n^{-a}$, $\rho=n^{-b}$ for constants $a>0$ and $0<b<\frac{1}{2}$, and suppose that $2a+2b>1$. Then, assuming that the low-degree heuristic from Section~\ref{subsubsec:low-degree-conjecture} holds for a related testing problem (formally defined in Definition~\ref{def-hidden-sample}), any algorithm that achieves weak recovery in the sense of Definition~\ref{def-recovery-planted-submatrix} must have running time at least $\exp(n^{2a-o(1)})$. 
\end{thm}
This result should be interpreted as evidence for a computational barrier throughout the conjectured hard regime for recovery. Indeed, when $2a+2b<1$, there are polynomial-time algorithms based on low-degree polynomial statistics that achieve strong recovery. By contrast, in the regime $2a+2b>1$, no polynomial-time algorithm is currently known, even though recovery remains statistically possible in a larger part of this region. Thus, the theorem supports the existence of a genuine detection-recovery gap in this model. It is also worth noting that the lower bound above is consistent with the complexity of the best currently known algorithms in the statistically feasible regime. In particular, when recovery is information-theoretically possible, the best known algorithms run in time of order $\exp(n^{2a+o(1)})$. Our result therefore suggests that this runtime is essentially optimal, at least from the perspective of the low-degree framework.

More broadly, the significance of Theorem~\ref{Main-thm-planted-submatrix-special-case} is not only the lower bound itself, but also the mechanism behind it. Existing low-degree lower bounds for recovery in planted submatrix are obtained by directly controlling the low-degree MMSE, which requires a rather delicate model-specific analysis. In contrast, our approach proceeds indirectly: we show that successful recovery would imply the existence of a hypothesis test with highly unbalanced error probabilities, and then rule out such a test using a mild bound on the low-degree advantage together with algorithmic contiguity (see Section~\ref{subsec:proof-overview} for details). This viewpoint is simpler, and it extends more readily to other models.

Finally, we emphasize that the hardness of estimation established above cannot be deduced from the statistical limits of estimation \cite{KBRS11, BI13, BIS15} nor from the computational limits of hypothesis testing \cite{MW15, BBH18}, as these involve different thresholds; see \cite{SW22} for further discussion.

\subsubsection{Other models}{\label{subsubsec:other-models}}
We now briefly summarize the consequences of our approach for the other models introduced above, postponing the precise statements to Section~\ref{sec:main-results}. In each case, the main point is similar to that of the planted submatrix example: once one has a suitable reduction from recovery to a testing problem with imbalanced success probabilities, it becomes enough to establish a relatively mild upper bound on the corresponding low-degree testing advantage.

For the planted dense subgraph model, we obtain the same lower bound as in the planted submatrix problem under the substitution $\lambda^2 = (p_1-p_0)^2/(p_0(1-p_0))$; see Theorem~\ref{Main-thm-planted-dense-subgraph}. In particular, this recovers the recovery lower bound established in earlier work \cite{SW22}, but through a simpler argument that avoids directly analyzing the low-degree MMSE. Since efficient algorithms are known to succeed above the corresponding threshold \cite{CX16, Mon15, SW25}, this again supports the conjectured computational barrier for recovery in the sparse regime.

For the stochastic block model, there are well-established conjectures regarding the computational threshold for the onset of weak recovery (i.e., nontrivial estimation). This threshold is known as the Kesten-Stigum (KS) threshold \cite{KS66, DKMZ11}. In particular, it was conjectured in \cite[Conjecture~2]{AS18} (based on \cite{DKMZ11}) that below the KS threshold, nontrivial estimation is impossible in polynomial time. In \cite{SW25}, the authors provide evidence for this computational transition at the KS threshold by establishing a tight lower bound on $\mathsf{MMSE}_{\leq D}$ via a delicate combinatorial argument, under the additional condition that the number of communities $q \ll n^{1/8}$. This result was later generalized to $q \ll n^{1/2}$ in \cite{CMSW25}. Moreover, it was shown in \cite{CMSW25, CGV25+} that when $q \gg n^{1/2}$, the KS threshold no longer characterizes the computational threshold for estimation. In the present work, we provide evidence for a computational transition at the KS threshold throughout the full regime $q \ll n^{1/8}$, thereby covering the results in \cite{SW25} with a substantially simpler argument. 

The multi-frequency angular synchronization model and the orthogonal group synchronization model are natural generalizations of the spiked matrix model \cite{BBP05, FP07}, for which a sharp computational transition is conjectured to occur at the Baik–Ben Arous–P\'ech\'e (BBP) threshold (see, e.g., \cite{KWB22, SW25}). For these generalized models, a natural and intriguing question is whether efficient algorithms can achieve nontrivial estimation at a lower signal-to-noise ratio than in the rank-one spiked matrix model. Based on non-rigorous methods from statistical physics together with numerical simulations, \cite{PWBM18b} predicted that the presence of additional frequencies does not reduce the computational threshold. Existing low-degree lower bounds, however, are limited to the regime of very low frequency or dimension (assuming $L$ or $d$ are constants; see, e.g., \cite{KBK24+}), where no detection-recovery gap occurs. In this work, we provide further supporting evidence by establishing lower bounds for estimation below the sharp BBP threshold when the frequency or dimension satisfies $L,d=n^{o(1)}$.

The multi-layer stochastic block model is a natural generalization of the stochastic block model and has attracted growing attention in statistical network analysis. Most previous results are restricted to the two-community setting \cite{LZZ24, YLS25, GHL26+}, where no statistical-computational gap occurs and the phase transition for estimation lies at the generalized Kesten-Stigum (KS) threshold. However, for a larger number of communities, we expect additional statistical-computational gaps to emerge, in line with phenomena observed in the single-layer model \cite{AS18, MSS25a}. In this work, we provide evidence that the generalized KS threshold remains the computational barrier for estimation, as long as both the number of communities and the number of layers are $n^{o(1)}$.

\subsection{Proof techniques}{\label{subsec:proof-overview}}

Here we present the main ideas in the lower bound proofs.

\subsubsection{Overview of the proof strategy}{\label{subsubsec:overview-strategy}}
We first provide a high-level overview of the proof of Theorem~\ref{Main-thm-planted-submatrix-special-case}. As mentioned before, our proof is motivated by a series of recent progress \cite{Du25+, GHL26+}, which exploits the fact that recovery is in fact considerably harder than detection \emph{statistically}, in the sense that recovery requires the chi-square divergence $\chi^2(\Pb\|\Qb)$ to be at least $e^{\Theta(n)}$, whereas detection only requires the chi-square divergence to be $\omega(1)$. The main contribution of this work is to develop an \emph{algorithmic analogue} of the arguments in \cite{Du25+, GHL26+} within the low-degree polynomial framework. Informally speaking, we show that even a mild bound on the low-degree advantage (defined in \eqref{eq-def-low-deg-Adv}) is sufficient to rule out all efficient recovery algorithms. The proof is organized into three key steps:
\begin{enumerate}
    \item[(1)] \underline{Reducing to a detection problem with lopsided success probability.} We will modify the cross-validation argument in \cite{DHSS25} to show that if weak recovery can be achieved efficiently, then we can find an efficient test between $\Pb$ and $\Qb$ with non-vanishing type-I accuracy and exponentially decaying type-II error (see Lemma~\ref{lem-reduce-to-one-sided-test-planted-submatrix} for details). 
    \item[(2)] \underline{Providing a mild bound on low-degree advantage.} Using a coarse concentration bound, we will show that the low-degree advantage $\mathsf{Adv}_{\leq D}(\Pb;\Qb)$ is \emph{not too large} (see Lemma~\ref{lem-bound-low-deg-adv} for details).
    \item[(3)] \underline{Arguing by contradiction via algorithmic contiguity}. We will carry out a refined analysis of the algorithmic contiguity framework in \cite{Li25}, which shows that (assuming the low-degree heuristic holds for the testing problem in Definition~\ref{def-hidden-sample}) a mild bound on the low-degree advantage suffices to rule out all tests with lopsided success probability (see Proposition~\ref{thm-alg-contiguity} for details), thus contradicting Step~(1).
\end{enumerate}

\subsubsection{Low-degree advantage and low-degree heuristic}{\label{subsubsec:low-degree-conjecture}}

Emerging from the works of \cite{BHK+19, HS17, HKP+17, Hopkins18}, the low-degree polynomial framework provides a useful proxy for computationally efficient algorithms in a broad range of high-dimensional inference problems (e.g., \cite{HS17, HKP+17, Hopkins18, SW22, BEH+22, BH22, DKWB24, KMW24, DDL25, CDGL24+}). Here we will focus on applying the framework of low-degree polynomials in the context of high-dimensional hypothesis testing problems. To be more precise, consider the hypothesis testing problem between two probability measures $\Pb$ and $\Qb$ based on the sample $\mathsf Y \in \mathbb R^{N}$. We will be interested in asymptotic settings in which $N=N_n, \Qb=\Qb_n, \Pb=\Pb_n, \mathsf Y=\mathsf Y_n$ scale with $n$ as $n \to \infty$ in some prescribed way. We first recall the standard notions of strong and weak detection.
\begin{defn}[Strong/weak detection]{\label{def-strong-weak-detection}}
    We say an algorithm $\mathcal A$ that takes $\mathsf Y$ as input and outputs either $0$ or $1$ achieves
    \begin{itemize}
        \item {\bf strong detection}, if the sum of type-I and type-II errors $\Qb(\mathcal A(\mathsf Y)=1) + \Pb(\mathcal A(\mathsf Y)=0)$ tends to $0$ as $n\to \infty$. 
        \item {\bf weak detection}, if the sum of type-I and type-II errors is uniformly bounded above by $1-\epsilon$ for some fixed $\epsilon>0$.
    \end{itemize}
\end{defn}
Let $\mathcal P_D=\mathcal P_{n,D}$ denote the set of polynomials from $\mathbb R^{N_n}$ to $\mathbb R$ with degree no more than $D$. The corresponding low-degree advantage is defined by
\begin{equation}{\label{eq-def-low-deg-Adv}}
    \mathsf{Adv}_{\leq D}(\Pb;\Qb):= \sup_{f \in \mathcal P_{D}} \left\{ \frac{ \mathbb E_{\Pb}[f] }{ \sqrt{ \mathbb E_{\Qb}[f^2] } } \right\} \,.
\end{equation}
Informally, $\mathsf{Adv}_{\leq D}(\Pb;\Qb)$ measures how well degree-$D$ polynomial statistics can distinguish $\Pb$ from $\Qb$. The guiding principle of the low-degree framework is that if this quantity remains small, then no computationally efficient algorithm should be able to solve the testing problem. We will use the following standard low-degree heuristic.
\begin{conjecture}[Low-degree conjecture in \cite{Hopkins18}]{\label{conj-low-deg}}
    For ``natural'' high-dimensional hypothesis testing problems between $\Pb$ and $\Qb$, the following statements hold.
    \begin{enumerate}
        \item[(1)] If $\mathsf{Adv}_{\leq D}(\Pb;\Qb)=O(1)$ as $n\to\infty$, then there exists a constant $C$ such that no algorithm with running time $N^{D/C\log N}$ that achieves strong detection between $\Pb$ and $\Qb$. 
        \item[(2)] If $\mathsf{Adv}_{\leq D}(\Pb;\Qb)=1+o(1)$ as $n\to\infty$, then there exists a constant $C$ such that no algorithm with running time $N^{D/C\log N}$ that achieves weak detection between $\Pb$ and $\Qb$. 
    \end{enumerate}
\end{conjecture}
Motivated by \cite[Hypothesis~2.1.5 and Conjecture~2.2.4]{Hopkins18} as well as the fact that low-degree polynomials capture the best known algorithms for a wide variety of statistical inference tasks, this heuristic appears to hold for distributions of a specific form that frequently arises in high-dimensional statistics. For further discussion on what types of distributions are suitable for this framework, we refer readers to \cite{Hopkins18, KWB22, ZSWB22, Wein25+}. However, to prevent the readers from being overly optimistic for this conjecture, we point out recent works \cite{BHJK25, JV26+} that find counterexamples for this conjecture. We therefore clarify that the low-degree framework is expected to be optimal for a certain, yet imprecisely defined, class of ``high-dimensional'' problems. Despite these important caveats, we still believe that analyzing low-degree polynomials remains highly meaningful for our setting, as it provides a benchmark for robust algorithmic performance. We refer the reader to the survey \cite{Wein25+} for a more detailed discussion of these subtleties. 

\subsubsection{Algorithmic contiguity from low-degree heuristic}{\label{subsubsec:alg-contiguity}}
The framework in Section~\ref{subsubsec:low-degree-conjecture} provides a useful tool for probing the computational feasibility of both strong and weak detection. However, as discussed in \cite{Li25}, the failure of strong detection alone is insufficient for many purposes, particularly when we aim to construct reductions between inference tasks in regimes where weak detection is computationally possible. To address this, \cite{Li25} introduces a stronger framework that rules out all \emph{one-sided tests} (i.e., tests satisfying $\Pb(\mathcal A=1)\geq \Omega(1)$ and $\Qb(\mathcal A=1)=o(1)$). Under the low-degree heuristic, it was shown that a bounded low-degree advantage not only rules out strong detection but also precludes any efficient one-sided testing algorithm. The core of this subsection is to provide fine-grained estimates of the arguments in \cite{Li25}, thereby elucidating the precise interplay between the low-degree advantage, the testing error, and the algorithm's running time.
\begin{DEF}{\label{def-one-side-test}}
    We say a test $\mathcal A$ that takes $\mathsf Y$ as input and outputs either $0$ or $1$ is a $(T;c;\epsilon)$-test between $\Pb$ and $\Qb$, if the following conditions hold:
    \begin{enumerate}
        \item[(1)] The test $\mathcal A$ can be calculated in time $O(N^{T})$;
        \item[(2)] The type-I accuracy satisfies $\Pb(\mathcal A(\mathsf Y)=1) \geq c$;
        \item[(3)] The type-II error satisfies $\Qb(\mathcal A(\mathsf Y)=1) \leq \epsilon$.
    \end{enumerate}
\end{DEF}
\begin{proposition}{\label{thm-alg-contiguity}}
    Consider the hypothesis testing problem between $\Pb=\Pb_n$ and $\Qb=\Qb_n$. Assuming that Heuristic~\ref{conj-low-deg} holds for a related testing problem (formally defined in Definition~\ref{def-hidden-sample}). If we have $\mathsf{Adv}_{\leq D_n}(\Pb_n;\Qb_n)^2 \leq \Delta_n$ for some $\Delta_n>1$, then there exists a constant $C>0$ such that there is no $(T_n;c_n;\epsilon_n)$-test between $\Pb_n$ and $\Qb_n$, such that (recall that $N_n$ is the dimension of $\mathsf Y_n$)
    \begin{equation}{\label{eq-assumption-time-error}}
        \log\Delta_n+(T_n+1)\log N_n \leq D_n/C, \quad c_n = \Omega(1), \quad \epsilon_n \Delta_n =o(1) \,.
    \end{equation}   
\end{proposition}
\begin{remark}
    Note that if we take $\Delta_n=O(1)$ and $\epsilon_n=o(1)$, Proposition~\ref{thm-alg-contiguity} reduces to \cite[Theorem~11]{Li25}. Thus, Proposition~\ref{thm-alg-contiguity} can be viewed as a fine-grained extension of \cite[Theorem~11]{Li25}.
\end{remark}
The remainder of this section is devoted to the proof of Proposition~\ref{thm-alg-contiguity}. For the sake of brevity, we will only work with some fixed $n$ throughout the analysis and we simply denote $\Pb_n,\Qb_n,N_n, \Delta_n,c_n,\epsilon_n$ as $\Pb,\Qb,N,\Delta,c,\epsilon$, respectively. Our proof roughly follows the proof of \cite[Theorem~11]{Li25} but we need a more careful quantitative bound. Suppose on the contrary that there is an efficient algorithm $\mathcal A$ that takes $\mathsf Y$ as input and outputs either $0$ or $1$ with
\begin{equation}{\label{eq-contradict-assumption}}
    \Pb(\mathcal A(\mathsf Y)=1) \geq c=\Omega(1) \mbox{ and } \Qb(\mathcal A(\mathsf Y)=0) \geq 1-\epsilon \,.
\end{equation}
To this end, we choose $M=M_n$ such that
\begin{equation}{\label{eq-def-M-n}}
    \Delta \ll M \leq N \Delta, \quad \epsilon M = o(1) \,.
\end{equation}
The crux of our argument is to consider the following \emph{hidden informative sample} problem.
\begin{defn}{\label{def-hidden-sample}}
    Consider the following hypothesis testing problem: we need to determine whether a sample $(\mathsf Y_1,\mathsf Y_2,\ldots, \mathsf Y_M)$ where each $\mathsf Y_i \in \mathbb R^N$ is generated by
    \begin{itemize}
        \item $\overline{\mathcal H}_0$: we let $\mathsf Y_1,\ldots,\mathsf Y_M$ be independently sampled from $\Qb_n$.
        \item $\overline{\mathcal H}_1$: we first sample $\kappa \in \{ 1,\ldots,M \}$ uniformly at random, and (conditioned on the value of $\kappa$) we let $\mathsf Y_1,\ldots,\mathsf Y_M$ be independent samples with $\mathsf Y_\kappa$ generated from $\Pb_n$ and $\{ \mathsf Y_j: j \neq \kappa \}$ generated from $\Qb_n$.
    \end{itemize}
    In addition, let $\overline{\Pb}$ and $\overline{\Qb}$ denote the law of $(\mathsf Y_1,\ldots,\mathsf Y_M)$ under $\overline{\mathcal H}_1$ and $\overline{\mathcal H}_0$, respectively. 
\end{defn}
We can now finish the proof of Proposition~\ref{thm-alg-contiguity}.
\begin{proof}[Proof of Proposition~\ref{thm-alg-contiguity}]
    Suppose on the contrary that there is a $(T;c;\epsilon)$-test $\mathcal A$ satisfying Definition~\ref{def-one-side-test}. Consider the hypothesis testing problem stated in Definition~\ref{def-hidden-sample}. Note that the dimension of $\overline{\Pb},\overline{\Qb}$ equals $MN$. Assuming that \eqref{eq-contradict-assumption} holds, we see that
    \begin{align}
        & \overline{\Qb}\Big( (\mathcal A(\mathsf Y_1), \ldots, \mathcal A(\mathsf Y_M)) = (0,\ldots,0) \Big) \geq 1-M\epsilon \overset{\eqref{eq-def-M-n}}{=} 1-o(1) \,; \label{eq-behavior-Qb-hidden-sample} \\
        & \overline{\Pb}\Big( (\mathcal A(\mathsf Y_1), \ldots, \mathcal A(\mathsf Y_M)) \neq (0,\ldots,0) \Big) \geq \Omega(1) \,. \label{eq-behavior-Pb-hidden-sample}
    \end{align}
    Thus, there is an algorithm that achieves weak detection between $\overline{\Pb}$ and $\overline{\Qb}$. In addition, this test can be calculated in time 
    \begin{align*}
        M N^{T} \overset{\eqref{eq-def-M-n}}{\leq} \Delta N^{T+1} \overset{\eqref{eq-assumption-time-error}}{\leq} (MN)^{D/C\log(MN)} \,.
    \end{align*}
    Thus, we see that there is an efficient algorithm that runs in time $(MN)^{D/C\log(MN)}$ and achieves weak detection between $\overline{\Pb}$ and $\overline{\Qb}$. However, using \cite[Lemma~2.6]{Li25}, we see that $\mathsf{Adv}_{\leq D}( \overline{\Pb};\overline{\Qb} )=1+o(1)$, which contradicts Item~(2) in Heuristic~\ref{conj-low-deg}.
\end{proof}

\subsubsection{A simple observation}{\label{subsubsec:reduce-to-easier-problem}}
We now sketch the proof of Theorem~\ref{Main-thm-planted-submatrix-special-case}. Without loss of generality, we may assume that $2a+b\leq 1$ since otherwise estimation is information-theoretically impossible \cite{KBRS11, BIS15}. We first argue that the following weaker result suffices to imply Theorem~\ref{Main-thm-planted-submatrix-special-case}.
\begin{lemma}{\label{lem-easier-prob-planted-submatrix}}
    Suppose that $\lambda=n^{-a},\rho=n^{-b}$ for some $2a+2b=1+\eta>1$. Then, assuming the low-degree heuristic holds for the testing problem in Definition~\ref{def-hidden-sample}, any algorithm achieving weak recovery in the sense of Definition~\ref{def-recovery-planted-submatrix} has runtime at least $\exp(n^{2a-\eta-o(1)})$.
\end{lemma}
We first show how to deduce Theorem~\ref{Main-thm-planted-submatrix-special-case} from Lemma~\ref{lem-easier-prob-planted-submatrix}. The intuition is that the problem becomes harder when we decrease $\lambda$ or $\rho$; therefore, it suffices to prove the result in the regime where $a+2b$ is sufficiently close to $1$. The following lemma formalizes this heuristic. 
\begin{lemma}{\label{lem-monotonicity-parameters}}
    Let $\Pb_{\lambda,\rho,n}$ denote the law of $\bm Y=\lambda \theta\theta^{\top}+\bm W$ with $\theta\in\mathbb R^n$ having i.i.d.\ $\mathsf{Ber}(\rho)$ entries for some $\lambda=\lambda_n, \rho=\rho_n$. Suppose that an algorithm $\mathcal A$ with running time $O(n^{T})$ takes $\bm Y\sim \Pb_{\lambda,\rho,n}$ as input and achieves weak recovery in the sense of Definition~\ref{def-recovery-planted-submatrix}. Then, for any $K=K_n\in \mathbb N$, there exists an algorithm $\mathcal A'$ with running time $O(Kn^{T})$ that takes $\bm Y \sim \Pb_{\lambda,\rho,Kn}$ as input and achieves weak recovery in the sense of Definition~\ref{def-recovery-planted-submatrix}.
\end{lemma}
\begin{proof}
    It is straightforward to verify that for any $n*n$ principal submatrix $\bm Y'$ of $\bm Y \sim \Pb_{\lambda,\rho,Kn}$, we have $\bm Y' \sim \Pb_{\lambda,\rho,n}$. Thus, weak recovery is easily achieved by applying $\mathcal A$ to each of the $K$ disjoint principal submatrices of $\bm Y$.
\end{proof}
Based on Lemma~\ref{lem-monotonicity-parameters}, we can finish the proof of Theorem~\ref{Main-thm-planted-submatrix-special-case}.
\begin{proof}[Proof of Theorem~\ref{Main-thm-planted-submatrix-special-case} assuming Lemma~\ref{lem-easier-prob-planted-submatrix}]
    Recall that $\lambda=n^{-a},\rho=n^{-b}$ for some $2a+2b>1$. Suppose on the contrary that there exists an algorithm $\mathcal A$ with running time $\exp(n^{2a-\delta})$ for some $\delta=\Omega(1)$ that takes $\bm Y \sim \Pb_{\lambda,\rho,n}$ as input and achieves weak recovery in the sense of Definition~\ref{def-recovery-planted-submatrix}. For any sufficiently small constant $\eta>0$, define 
    \begin{align*}
        N=Kn=n^{ \frac{2a+2b}{1+\eta} } \,.
    \end{align*}
    Using Lemma~\ref{lem-monotonicity-parameters}, we know that there is an algorithm $\mathcal A'$ with running time
    \begin{align*}
        \exp\left( n^{2a-\delta} \right) \cdot \operatorname{poly}(n) = \exp\left( N^{ \frac{(1+\eta)(2a-\delta)}{2a+2b}+o(1) } \right)
    \end{align*}
    that takes
    \begin{align*}
        \bm Y' \sim \Pb_{\lambda,\rho,N} \mbox{ with } \lambda=n^{-a} = N^{ \frac{-(1+\eta)a}{2a+2b}}, \rho = n^{-b} = N^{ \frac{-(1+\eta)b}{2a+2b}} 
    \end{align*}
    as input and achieves weak recovery in the sense of Definition~\ref{def-recovery-planted-submatrix}. However, from Lemma~\ref{lem-easier-prob-planted-submatrix} we have that all weak recovery algorithms for $\bm Y' \sim \Pb_{\lambda,\rho,N}$ require runtime at least 
    \begin{align*}
        \exp\left( N^{ \frac{(1+\eta)a}{a+b}-\eta+o(1) } \right) \,.
    \end{align*}
    Choosing $\eta=\eta(a,b,\delta)$ sufficiently small leads to a contradiction.
\end{proof}
From now on we will instead focus on Lemma~\ref{lem-easier-prob-planted-submatrix}. 

\subsubsection{Reducing to an imbalanced detection problem}{\label{subsubsec:reduction-to-detection}}
As mentioned in Section~\ref{subsec:proof-overview}, we will reduce Lemma~\ref{lem-easier-prob-planted-submatrix} to a detection problem with lopsided probability. Let $\Qb=\Qb_n$ be the law of $\bm W$, a symmetric random matrix whose entries $\{ \bm W_{i,j}:i\leq j \}$ are independent standard normal variables.
\begin{lemma}{\label{lem-reduce-to-one-sided-test-planted-submatrix}}
    Suppose that $\lambda=n^{-a},\rho=n^{-b}$ for some $2a+2b=1+\eta>1$. Let $\lambda'=\sqrt{1+\kappa^2} \lambda$ for some small constant $\kappa>0$. Suppose that an algorithm $\mathcal A$ with running time $n^{T}$ takes $\bm Y\sim \Pb_{\lambda,\rho,n}$ as input and achieves weak recovery in the sense of Definition~\ref{def-recovery-planted-submatrix}. Then there exists a $(O(T);c;\iota)$-test between $\Pb_{\lambda',\rho,n}$ and $\Qb_n$ with $c=\Omega(1)$ and $\iota=e^{-\Theta(n^{1-\eta})}$.
\end{lemma}
The rest of this subsection is devoted to the proof of Lemma~\ref{lem-reduce-to-one-sided-test-planted-submatrix}. Our proof is inspired by the cross-validation technique in \cite{DHSS25}. Suppose that given $\bm Y \sim \Pb_{\lambda,\rho,n}$, there exists an estimator $\mathcal X(\bm Y)$ satisfying Definition~\ref{def-recovery-planted-submatrix}, i.e., 
\begin{align*}
    \mathbb E_{\Pb}\left[ \frac{ \langle \mathcal X, \theta \theta^{\top} \rangle }{ \| \mathcal X \|_{\Fop} \| \theta \theta^{\top} \|_{\Fop} } \right] \geq c \mbox{ for some constant } c>0 \,.
\end{align*}
Using Markov inequality, we then have
\begin{align*}
    \Pb\left( \frac{ \langle \mathcal X, \theta \theta^{\top} \rangle }{ \| \mathcal X \|_{\Fop} \| \theta \theta^{\top} \|_{\Fop} } \geq \frac{c}{2} \right) = 1-\Pb\left( 1 - \frac{ \langle \mathcal X, \theta \theta^{\top} \rangle }{ \| \mathcal X \|_{\Fop} \| \theta \theta^{\top} \|_{\Fop} } \ge 1-\frac{c}{2} \right) \geq 1-\frac{1-c}{1-\frac{c}{2}} \geq \frac{c}{2} \,.
\end{align*}
Recall that $\lambda'=\lambda\sqrt{1+\kappa^2}$ for a small constant $\kappa>0$. We will introduce external randomness $\bm Z \sim \Qb$ independent of $\bm Y$ and set
\begin{align}
    \bm A = \frac{ 1 }{ \sqrt{1+\kappa^2} } \left( \bm Y+ \kappa \bm Z \right), \quad \bm B = \frac{ 1 }{ \sqrt{1+\kappa^{-2}} } \left( \bm Y - \kappa^{-1} \bm Z \right) \,.  \label{eq-split-into-two-matrices}
\end{align}
The reasons for the definition in \eqref{eq-split-into-two-matrices} are as follows:
\begin{itemize}
    \item Under $\bm Y \sim \Qb_n$, we have 
    \begin{align}
        \bm A = \frac{ 1 }{ \sqrt{1+\kappa^2} } \left( \bm W + \kappa \bm Z \right), \quad \bm B = \frac{ 1 }{ \sqrt{1+\kappa^{-2}} } \left( \bm W - \kappa^{-1} \bm Z \right) \,.  \label{eq-behavior-Y-1,2-Qb}
    \end{align}
    In particular, $(\bm A,\bm B)$ are matrices with i.i.d.\ standard normal entries and $\bm B$ is \underline{independent of} $\bm A$.
    \item Under $\bm Y \sim\Pb_{\lambda',\rho,n}$, we have 
    \begin{align}
        \bm A = \lambda \theta\theta^{\top} + \overline{\bm A}, \quad \bm B = \frac{ \lambda \sqrt{ (1+\kappa^2) } }{ \sqrt{(1+\kappa^{-2})} } \theta\theta^{\top} + \overline{\bm B} \,,  \label{eq-behavior-Y-1,2-Pb}
    \end{align}
    where 
    \begin{align*}
        \overline{\bm A} = \frac{ 1 }{ \sqrt{1+\kappa^2} } \left( \bm W + \kappa \bm Z \right), \quad \overline{\bm B} = \frac{ 1 }{ \sqrt{1+\kappa^{-2}} } \left( \bm W - \kappa^{-1} \bm Z \right) \,.
    \end{align*}
    In particular, $(\overline{\bm A}, \overline{\bm B})$ are matrices with i.i.d.\ standard normal entries and $\overline{\bm B}$ is \underline{independent of} $\bm A$. Also we have $\bm A \sim \Pb_{\lambda,\rho,n}$.
\end{itemize}
Now, since $\bm A \sim \Pb_{\lambda,\rho,n}$ under $\bm Y \sim \Pb_{\lambda',\rho,n}$, we can find an estimator 
\begin{align}{\label{eq-generate-estimator}}
    \mathcal X=\mathcal X(\bm A) \mbox{ such that } \Pb\Bigg( \frac{ \langle \mathcal X, \theta\theta^{\top} \rangle }{ \| \mathcal X \|_{\Fop} \| \theta\theta^{\top} \|_{\Fop} } \geq \frac{c}{2} \Bigg) \geq \frac{c}{2} \,.
\end{align}
The next lemma shows that $\langle \mathcal X, \bm B \rangle$ is ``large'' under $\Pb_{\lambda',\rho,n}$ and ``small'' under $\Qb_n$.
\begin{lemma}{\label{lem-behavior-Pb-Qb}}
    Suppose we choose $\mathcal X(\bm A)$ as in \eqref{eq-generate-estimator}. Then 
    \begin{align}
        &\Pb\left( \big\langle \mathcal X, \bm B \big\rangle \geq \frac{ c\lambda \rho n \| \mathcal X \|_{\Fop} }{ 4 } \sqrt{ \frac{ (1+\kappa^2) }{ (1+\kappa^{-2}) } } \right) \geq \frac{c}{2}-e^{-\Theta(n)} \,;  \label{eq-behavior-Pb} \\
        &\Qb\left( \big\langle \mathcal X, \bm B \big\rangle \geq \frac{ c\lambda \rho n \| \mathcal X \|_{\Fop} }{ 4 } \sqrt{ \frac{ (1+\kappa^2) }{ (1+\kappa^{-2}) } } \right) \leq e^{-\Theta(\lambda^2 \rho^2 n^2)} = e^{-\Theta(n^{1-\eta})} \,.  \label{eq-behavior-Qb}
    \end{align}
\end{lemma}
\begin{proof}
    We first prove \eqref{eq-behavior-Qb}. Recall \eqref{eq-behavior-Y-1,2-Qb}, under $\Qb$ we have $\bm B$ is a standard Wigner matrix independent of $\bm A$ (and thus also independent of $\mathcal X$). Thus, conditioned on $\mathcal X$ we have
    \begin{align*}
        \big\langle \mathcal X, \bm B \big\rangle \sim \mathcal N\left( 0,\| \mathcal X \|_{\Fop}^2 \right) \,.
    \end{align*}
    Thus, from a simple Gaussian tail inequality we see that \eqref{eq-behavior-Qb} holds.
    
    We then prove \eqref{eq-behavior-Pb}. Recall \eqref{eq-behavior-Y-1,2-Pb}, we can decompose $\langle \mathcal X, \bm B \rangle$ into the following terms:
    \begin{align}
        \big\langle \mathcal X, \bm B \big\rangle &= \frac{ \lambda \sqrt{ (1+\kappa^2) } }{ \sqrt{(1+\kappa^{-2})} } \cdot \big\langle \mathcal X, \theta \theta^{\top} \big\rangle \label{eq-behavior-Pb-part-1}  \\
        &+ \big\langle \mathcal X, \overline{\bm B} \big\rangle \,.  \label{eq-behavior-Pb-part-2}
    \end{align}
    Using \eqref{eq-generate-estimator}, we see that (recall that $\theta$ has i.i.d\ $\mathsf{Ber}(\rho)$ entries)  
    \begin{align}
        & \Pb\left( \eqref{eq-behavior-Pb-part-1} \geq \frac{ c\lambda\rho n \| \mathcal X \|_{\Fop} }{ 3 } \sqrt{ \frac{ (1+\kappa^2) }{ (1+\kappa^{-2}) } } \right) = \Pb\left( \big\langle \mathcal X, \theta \theta^{\top} \big\rangle \geq \frac{c\rho n\| \mathcal X \|_{\Fop}}{3} \right) \nonumber \\
        \ge \ & \Pb\left( \big\langle \mathcal X, \theta \theta^{\top} \big\rangle \geq \frac{c}{2} \| \mathcal X \|_{\Fop} \| \theta \theta^{\top} \|_{\Fop} \right) - \Pb\left( \| \theta \theta^{\top} \|_{\Fop} \leq \frac{2\rho n}{3} \right) \geq \frac{c}{2}-o(1) \,.  \label{eq-behavior-Pb-part-1-bound}
    \end{align}
    where the last inequality follows from \eqref{eq-generate-estimator}.
    In addition, from the proof of \eqref{eq-behavior-Qb}, we see that 
    \begin{align}
        \Pb\left( |\eqref{eq-behavior-Pb-part-2}| \geq \frac{ c\lambda\rho n \| \mathcal X \|_{\Fop} }{ 12 } \sqrt{ \frac{ (1+\kappa^2) }{ (1+\kappa^{-2}) } } \right) \leq e^{-\Theta(n^{1-\eta})}=o(1) \,.   \label{eq-behavior-Pb-part-2-bound}
    \end{align}
    Combining \eqref{eq-behavior-Pb-part-1-bound} and \eqref{eq-behavior-Pb-part-2-bound}, we have that 
    \begin{align*}
    &\Pb\left( \big\langle \mathcal X, \bm B \big\rangle \geq \frac{ c\lambda\rho n \| \mathcal X \|_{\Fop} }{ 4 } \sqrt{ \frac{ (1+\kappa^2) }{ (1+\kappa^{-2}) } } \right) \\
    \ge\ & \Pb\left(\eqref{eq-behavior-Pb-part-1} \geq \frac{ c\lambda\rho n \| \mathcal X \|_{\Fop} }{ 3 } \sqrt{ \frac{ (1+\kappa^2) }{ (1+\kappa^{-2}) } } \right) \\ 
    +\ &\Pb\left( |\eqref{eq-behavior-Pb-part-2}| < \frac{ c\lambda\rho n \| \mathcal X \|_{\Fop} }{ 12 } \sqrt{ \frac{ (1+\kappa^2) }{ (1+\kappa^{-2}) } } \right)\ge \frac{c}{2}-o(1) \,. 
    \end{align*}
    We then see that \eqref{eq-behavior-Pb} holds.
\end{proof}
We point out that Lemma~\ref{lem-behavior-Pb-Qb} immediately implies the existence of an $(O(T);\Omega(1);e^{-\Theta(n^{1-\eta})})$-test between $\Pb_{\lambda',\rho,n}$ and $\Qb_n$.

\subsubsection{Bounding the low-degree advantage}{\label{subsubsec:bounding-low-deg-adv}}
We will proceed by providing a mild bound on the low-degree advantage $\mathsf{Adv}_{\leq D}(\Pb_{\lambda',\rho,n};\Qb_n)$, as incorporated in the following lemma.
\begin{lemma}{\label{lem-bound-low-deg-adv}}
    Suppose that $\lambda=n^{-a},\rho=n^{-b}$ for some $2a+2b=1+\eta>1$. Let $\lambda'=\sqrt{1+\kappa^2} \lambda$ for some small constant $\kappa>0$. We then have
    \begin{equation}{\label{eq-bound-low-deg-adv}}
        \mathsf{Adv}_{\leq D}(\Pb_{\lambda',\rho,n};\Qb_n)^2 \leq \exp\left( \Theta(1) \cdot n^{2a-2\eta+o(1)} \right) \mbox{ for all } D \leq n^{2a-\eta}/(\log n)^2 \,.
    \end{equation}
\end{lemma}
\begin{proof}
    Define 
    \begin{equation}{\label{eq-exp-leq-D}}
        \exp_{\leq D}(x) = \sum_{k=0}^{D} \frac{ x^k }{ k! } \,.
    \end{equation}
    Using \cite{KWB22}, we have
    \begin{align}{\label{eq-low-deg-adv-relax-1}}
        \mathsf{Adv}_{\leq D}(\Pb_{\lambda',\rho,n};\Qb_n) &\leq \mathbb E_{\theta,\theta'}\left[ \exp_{\leq D}\left( (\lambda')^2 \langle \theta,\theta' \rangle^2 \right) \right] \nonumber \\
        &= \mathbb E_{\theta,\theta'}\left[ \exp_{\leq D}\left( (1+\kappa^2)\lambda^2 \langle \theta,\theta' \rangle^2 \right) \right] \,,
    \end{align}
    where $\theta,\theta'$ are independent random vectors with i.i.d.\ $\mathsf{Ber}(\rho)$ entries. It suffices to bound the right hand side of \eqref{eq-low-deg-adv-relax-1}. Clearly, we have
    \begin{align}{\label{eq-low-deg-adv-relax-2}}
        \eqref{eq-low-deg-adv-relax-1} = \sum_{k=0}^{D} \frac{ (1+\kappa^2)^k \lambda^{2k} \mathbb E_{\theta,\theta'}[ \langle \theta,\theta' \rangle^{2k} ] }{ k! } \,.
    \end{align}
    Note that $\langle \theta,\theta' \rangle \sim \mathsf{Bin}(n,\rho^2)$, we have
    \begin{align*}
        \mathbb E_{\theta,\theta'}\left[ \langle \theta,\theta' \rangle^{2k} \right] &\leq (2n\rho^2)^{2k} + n^{2k} \cdot \Pb_{\theta,\theta'}\left( \langle \theta,\theta' \rangle>2n\rho^2  \right) \\
        &\leq (2n\rho^2)^{2k} + n^{2k} \cdot \exp\left( -\Theta(n\rho^2) \right) = (2n\rho^2)^{2k} + \exp\left( 2k\log n - \Theta(n\rho^2) \right) \\
        &\leq (2n\rho^2)^{2k} + \exp\left( 2n^{2a-\eta}/\log n - \Theta(n\rho^2) \right) = (2n\rho^2)^{2k}+o(1) \,,
    \end{align*}
    where the first inequality follows from $\langle \theta,\theta' \rangle \leq n$ and the standard Markov bound, the second inequality follows from the standard Chernoff's bound, the third inequality follows from $k \leq D \leq n^{2a-\eta}/(\log n)^2$, and the last equality follows from $\rho=n^{-b}$ with $2a+2b=1+\eta$. Thus, we have
    \begin{align*}
        \eqref{eq-low-deg-adv-relax-2} &\leq \sum_{k=0}^{D} \frac{ (1+\kappa^2)^k \lambda^{2k} }{ k! } \cdot (2n\rho^2)^{2k} \leq \exp\left( (1+\kappa^2)\lambda^2 (2n\rho^2)^2 \right) \\
        &= \exp\left( \Theta(1) \cdot n^2 \lambda^2 \rho^4 \right) = \exp\left( \Theta(1) \cdot n^{2-2a-4b} \right) = \exp\left( \Theta(1) \cdot n^{2a-2\eta} \right) \,,
    \end{align*}
    where the last equality follows from $2a+2b=1+\eta$.
\end{proof}  

\subsubsection{Putting it together}{\label{subsubsection:concluding}}
We can now conclude Lemma~\ref{lem-easier-prob-planted-submatrix}, thereby finishing the proof of Theorem~\ref{Main-thm-planted-submatrix-special-case}.
\begin{proof}[Proof of Lemma~\ref{lem-easier-prob-planted-submatrix}]
    Suppose on the contrary that there exists an algorithm $\mathcal A$ with running time $\exp(n^{2a-\eta-\Omega(1)})$ that takes $\bm Y\sim \Pb_{\lambda,\rho,n}$ as input and achieves weak recovery in the sense of Definition~\ref{def-recovery-planted-submatrix}. Using Lemma~\ref{lem-reduce-to-one-sided-test-planted-submatrix}, we see that there exists an $(T_n;c_n;\epsilon_n)$-test between $\Pb_{\lambda',\rho,n}$ and $\Qb_n$ such that
    \begin{align*}
        T_n=n^{2a-\eta-\Omega(1)}, \quad c_n=\Omega(1), \quad \epsilon_n = \exp(-\Theta(n^{1-\eta})) \,. 
    \end{align*}
    However, using Lemma~\ref{lem-bound-low-deg-adv} and Proposition~\ref{thm-alg-contiguity}, we see that (assuming the low-degree heuristic holds for the problem in Definition~\ref{def-hidden-sample}) there is no $(T_n';c_n';\epsilon_n')$-test between $\Pb_{\lambda',\rho,n}$ and $\Qb_n$ such that 
    \begin{equation*}
        T_n'=n^{2a-\eta-\Omega(1)}, \quad c_n'=\Omega(1), \quad \epsilon_n' = \exp(-n^{2a-\eta+\Omega(1)}) \,.
    \end{equation*}
    This forms a contradiction and thus yields Lemma~\ref{lem-easier-prob-planted-submatrix}.
\end{proof}

\subsubsection{Other models}{\label{subsubsec:proof-overview-other-models}}
We now summarize how the arguments extend to the other models considered in this work. The cross-validation argument from Section~\ref{subsubsec:reduction-to-detection} applies directly to any additive Gaussian model, as defined in Definition~\ref{def-add-Gaussian-model}, and therefore covers settings such as multi-frequency angular synchronization and orthogonal group synchronization. For models with binary-valued observations including the planted dense subgraph model, the stochastic block model, and the multi-layer stochastic block model, the cross-validation argument must be modified. A key ingredient in this case is the correlation-preserving technique introduced in \cite{HS17, DHSS25}. However, compared with the related cross-validation argument in \cite{DHSS25}, our approach requires a modified construction in order to obtain a sharp exponential tail bound (see Sections~\ref{subsec:cor-pre-proj} and \ref{subsec:fact-Ber-obser-model} for more details).

Given the cross-validation argument together with Proposition~\ref{thm-alg-contiguity}, it remains only to establish a mild upper bound on the low-degree advantage for each model. For the planted submatrix model and the planted dense subgraph model, such bounds are either straightforward or already available in previous work. For models such as multi-frequency angular synchronization and orthogonal group synchronization, we prove the required bounds using a careful Lindeberg's interpolation argument. Finally, for the stochastic block model and the multi-layer stochastic block model, we derive the corresponding bound through a careful combinatorial argument.

\subsection{Further discussions}{\label{subsec:additional-discussion}}

\subsubsection{Why the Low-degree framework}
Several general approaches have been developed for studying statistical-computational gaps in high-dimensional inference. In the present work, we focus on the low-degree framework because it appears particularly well suited to the recovery problems considered here.

One reason is that the low-degree framework applies directly to planted inference problems formulated in terms of a single high-dimensional observation, such as planted submatrix, planted dense subgraph, stochastic block model, and the synchronization models studied in this paper. Moreover, it has proved to be flexible enough to capture the behavior of many of the best known efficient algorithms in related settings, and it provides a useful benchmark for formulating conjectural computational thresholds.

By contrast, several alternative frameworks are less directly adapted to the present setting. The statistical query framework \cite{FGR+17, DKS17} is most naturally formulated for problems based on i.i.d.\ samples from an unknown distribution, whereas the models studied here are not naturally presented in that form. Sum-of-Squares lower bounds \cite{BHK+19, KMOW17, RSS18} provide very strong evidence in many average-case problems, but they are often most naturally connected to certification tasks, while our goal is to study recovery. Likewise, approaches based on the geometry of optimization landscapes \cite{BGJ20, GJS21, BWZ23} can give useful information about specific algorithmic formulations, but they do not necessarily address the inherent difficulty of the underlying recovery problem.

In summary, these other perspectives are important and complementary, and it would be interesting to understand whether they can also be brought to bear on the models considered here. Our choice to work with the low-degree framework is therefore not meant to exclude other approaches, but rather to adopt a framework that is both flexible and sufficiently close to the recovery questions of interest in this paper.


\subsubsection{Comparison with the unconditional low-degree lower bound}
We briefly compare the conditional lower bounds obtained in this paper with the unconditional low-degree lower bounds developed in \cite{SW22,SW25}\footnote{We note that, however, the lower bounds in \cite{SW22, SW25} are in fact equivalent to assuming that the low-degree heuristic holds for the testing problem between $\Pb(\cdot \mid \theta_1)$ and $\Pb$, as discussed in \cite{Wein25+}.}. The two approaches address recovery lower bounds from rather different perspectives.

The approach in \cite{SW22,SW25} proceeds by directly controlling the quantity $\mathsf{MMSE}_{\leq D}$, and in several important models this leads to very sharp lower bounds for recovery. In particular, in problems such as planted submatrix and planted dense subgraph, that method is able to match the sharp algorithmic threshold predicted by AMP-based methods. By contrast, the results in the present paper are conditional on the low-degree heuristic for a related testing problem, and therefore should be interpreted in a different way.

In comparison, the main advantage of our approach is that it is comparatively simple and largely model-independent. The general statement in Proposition~\ref{thm-alg-contiguity} applies quite generally and does not require model-specific analysis, while the cross-validation argument only requires the model to possess a suitable additive structure. Once these ingredients are in place, the remaining task is to establish a mild upper bound on the low-degree testing advantage between the planted and null distributions. This is often substantially easier than bounding $\mathsf{MMSE}_{\leq D}$, since one can work with an explicit orthogonal basis.

This difference is particularly relevant for the broader range of models considered in this paper. In more classical settings such as planted submatrix, planted dense subgraph, and stochastic block model, our method provides a simpler route to known or expected low-degree lower bounds for recovery. In less standard settings such as multi-frequency angular synchronization, orthogonal group synchronization, and multi-layer stochastic block model, the same viewpoint seems easier to extend, at least at the level of conditional evidence. By contrast, the approach of \cite{SW25} for spiked matrix models relies on a delicate cumulant bound from \cite{SW22}, which appears to depend more or less on the i.i.d.\ structure of the planted signal. For this reason, we view the present approach not as a replacement for the unconditional approach of \cite{SW22,SW25}, but rather as a complementary framework that trades unconditional lower bounds for simplicity and flexibility.  

\subsubsection{Comparison with the extended low-degree conjecture}
The work most closely related in spirit to ours is \cite{DHSS25}, which also studies recovery lower bounds based on a recovery-to-detection reduction. In particular, their cross-validation argument is an important source of inspiration for the present paper. The main conceptual difference between our work and \cite{DHSS25} is that the reduction argument in \cite{DHSS25} relies on a \emph{strengthened} version of the low-degree conjecture inspired by \cite{MW23+}, which roughly posits that low-degree polynomials perform at least as well as all algorithms of the corresponding runtime, when performance is measured by the testing advantage. This is a stronger statement than the heuristic used in the present paper. Here, by contrast, we only appeal to the usual low-degree heuristic for a related testing problem, together with the algorithmic contiguity framework developed in \cite{Li25}.

More precisely, consider the hypothesis testing problem between $\bm Y \sim \Pb$ and $\bm Y \sim \Qb$, where $\Pb$ and $\Qb$ are distributions on $\mathbb R^N$. For any test $f:\mathbb R^N \to \{0,1\}$, define
\begin{align*}
    \mathsf{Adv}(f):= \frac{ \mathbb E_{\Pb}[f] }{ \sqrt{\mathbb E_{\Qb}[f^2]} } \,.
\end{align*}
The work \cite{DHSS25} posits that for any $D\leq n^{1-\Omega(1)}$, if $\mathcal A_D$ denotes the set of functions $f:\mathbb R^N \to \{0,1\}$ computable in time $N^{O(D/\log N)}$, then
\begin{align}
    \max_{f \in \mathcal A_D} \left\{ \mathsf{Adv}(f) \right\} \leq O(1) \cdot \max_{ \operatorname{deg}(f)=O(D) } \left\{ \mathsf{Adv}(f) \right\} \,.  \label{eq-extend-low-deg-conj}
\end{align}
This extended low-degree hypothesis is closely related to the low-degree lower bound for random optimization problems (see \cite{GJW24}). However, several counterexamples limit its applicability. The most important counterexample is provided by the planted submatrix model
\begin{align*}
    \bm Y=\lambda vv^\top+\bm W \,,
\end{align*}
where $v\in\{0,1\}^n$ has i.i.d.\ $\mathsf{Ber}(\rho)$ entries with $\lambda=n^{-a},\rho=n^{-b}$. When $2(a+b)>1$ is below the spectral threshold and $2a+b<1$ is above the statistical threshold for recovery, a direct calculation shows that $\mathsf{Adv}(f)\le n^{O(D)}$ for every polynomial $f$ of degree at most $D$. On the other hand, since there is a weak recovery algorithm with runtime $\exp(n^{2a-o(1)})$, using the cross-validation trick, it is easy to construct a function $f$ computable in time $\exp(n^{2a-o(1)})$ such that 
\begin{align*}
    \mathsf{Adv}(f)=\exp(\Omega(n)) \gg \exp(n^{2a+o(1)}) \,.
\end{align*}
However, this also does not contradict the framework in Section~\ref{subsec:proof-overview}, since the conditions in Proposition~\ref{thm-alg-contiguity} are meaningful only when $\mathsf{Adv}_{\le D}(\Pb;\Qb)=\exp(o(D))$, and when $D\gg n^{2a}$ we actually have $\mathsf{Adv}_{\le D}(\Pb;\Qb)\geq \exp(\Omega(D))$. Indeed, when we restrict attention to polynomial-time algorithms, the condition in Proposition~\ref{thm-alg-contiguity} is precisely that $\mathsf{Adv}_{\le D}(\Pb;\Qb)=\exp(o(D))$, provided that $D\ge O(\log N)$. This suggests that an \emph{exponential growth} of the low-degree advantage may be crucial for the success of low-degree estimation algorithms. Indeed, in many examples including planted dense subgraph \cite{SW22}, graph alignment \cite{MWXY24, MWXY23}, and stochastic block model \cite{MNS18, HS17}, the key step in constructing a low-degree estimation algorithm in the ``easy'' regime is to find a degree-$D$ polynomial satisfying $\mathsf{Adv}(f)=\exp(\Theta(D))$. In view of this, we suggest revising \eqref{eq-extend-low-deg-conj} to
\begin{equation}{\label{eq-extend-low-deg-conj-fixed}}
    \min\left\{ \max_{f \in \mathcal A_D} \left\{ \mathsf{Adv}(f) \right\}, e^{O(D)} \right\} \leq O(1) \cdot \max_{ \operatorname{deg}(f)=O(D) } \left\{ \mathsf{Adv}(f) \right\} \mbox{ for all } \mathbb E_{\Pb}[f]=\Omega(1)
\end{equation}
as a tentative fix to the extended low-degree conjecture.

\subsubsection{On the testing problem in Definition~\ref{def-hidden-sample}}
A natural caveat in our framework is that the auxiliary testing problem introduced in Definition~\ref{def-hidden-sample} is not symmetric in the strongest sense usually assumed in the original low-degree framework \cite{Hopkins18}. Accordingly, one may ask whether the low-degree heuristic should still be expected to apply in this setting. We do not attempt to resolve this question here. Rather, we simply note that several recent works \cite{MW25b, KVWX23, DDL25} suggest that the low-degree methodology can remain informative under weaker symmetry assumptions than those appearing in the original formulation. At the same time, there are also examples in the literature \cite{Kunisky21, HW21} showing that the predictive power of the framework can fail in certain ad hoc asymmetric settings, so this issue should not be ignored.

For the problems considered in the present paper, we believe that the testing problem in Definition~\ref{def-hidden-sample} is still sufficiently structured that the low-degree heuristic provides meaningful guidance. Nevertheless, this should be viewed as a conditional input rather than as a theorem. In particular, the results of this paper should be interpreted as conditional evidence for computational barriers to recovery within the low-degree framework, rather than as unconditional hardness results. We also note that the previous work \cite{BDT24} also used the idea of studying carefully chosen related testing problems, although the specific technique they used is significantly different from ours.

The reader may also wonder whether one could work with a weaker conjecture than Heuristic~\ref{conj-low-deg}, one that only postulates the failure of algorithms with running time $n^{D/(\log n)^C}$ for some sufficiently large constant $C$, rather than ruling out all algorithms with running time $n^{O(D/\log n)}$. We emphasize, however, that Heuristic~\ref{conj-low-deg} still captures the intended tradeoff between polynomial degree and algorithmic running time, since it is widely believed that the failure of degree-$O(\log n)$ polynomials provides compelling evidence for the failure of all polynomial-time algorithms. Further support for this viewpoint comes from the planted clique problem: when the clique size satisfies $\log n \ll K \ll \sqrt{n}$, it is known that degree-$D$ polynomials fail if and only if $D=o((\log n)^2)$ \cite{BHK+19}. On the other hand, the best currently known algorithms in this regime indeed run in quasi-polynomial time $n^{\Theta(\log n)}$. We also remark that our results for planted submatrix and planted dense subgraph rely crucially on the stronger assumption ruling out algorithms with running time $n^{D/\log n}$. Indeed, the associated detection problem that we construct involves detecting a $\operatorname{poly}(\log N)=\operatorname{poly}(n)$-dimensional signal hidden in an ambient space of dimension $N=\exp(n^{\Omega(1)})$, which is analogous to the planted clique setting. By contrast, our results for multi-frequency angular synchronization, orthogonal group synchronization, and multi-layer stochastic block model remain unchanged under the weaker assumption. Finally, our result for stochastic block model also continues to hold under the weaker assumption, but in a weaker form: it then applies only when the number of communities satisfies $q=n^{o(1)}$.

\subsection{Notation}{\label{subsec:notation}}

We record in this subsection some notation conventions. We denote by $\mathsf{Ber}(p)$ the Bernoulli distribution with parameter $p$, denote by $\mathsf{Bin}(n,p)$ the binomial distribution with $n$ trials and success probability $p$, denote by $\mathcal N(\mu,\sigma^2)$ the normal distribution with mean $\mu$ and variance $\sigma^2$, and denote by $\mathcal N(\mu,\Sigma)$ the multivariate normal distribution with mean $\mu$ and covariance matrix $\Sigma$. For two probability measures $\mathbb P$ and $\mathbb Q$, we denote the total variation distance between them by $\mathsf{TV}(\mathbb P,\mathbb Q)$. The chi-squared divergence from $\Pb$ to $\Qb$ is defined as $\chi^2(\Pb \| \Qb)= \mathbb E_{\mathbf{X}\sim\Qb}[ (\frac{\mathrm{d}\Pb}{\mathrm{d}\Qb}(\mathbf X))^2 ]$. 

For a matrix or a vector $M$, we will use $M^{\top}$ to denote its transpose, and we denote $M^*=\overline{M}^{\top}$. For a $k*k$ matrix $M$, let $\operatorname{det}(M)$ and $\operatorname{tr}(M)$ be the determinant and trace of $M$, respectively. Denote $M \succ 0$ if $M$ is positive definite and $M \succeq 0$ if $M$ is semi-positive definite. In addition, if $M=M^*$ is Hermitian we let $\varsigma_1(M) \geq \varsigma_2(M) \geq \ldots \geq \varsigma_k(M)$ be the eigenvalues of $M$. Denote by $\mathrm{rank}(M)$ the rank of the matrix $M$. Denote $\mathsf{O}(m)$ to be the set of all $m*m$ orthogonal matrices, and denote $\mathsf{SO}(m)=\{ M\in \mathsf{O}(m):\mathrm{det}(M)=1 \}$. For two $k*k$ complex matrices $M_1$ and $M_2$, we define their inner product to be
\begin{align*}
    \big\langle M_1,M_2 \big\rangle:=\sum_{i,j=1}^k M_1(i,j) \overline{M}_2(i,j) \,.
\end{align*}
In particular, if $M_1,M_2$ are Hermitian matrices then $\langle M_1,M_2 \rangle \in \mathbb R$. We also define the Frobenius norm, operator norm respectively, and $\infty$-norm by
\begin{align*}
    \| M \|_{\operatorname{F}} = \mathrm{tr}(MM^{*})^{\frac{1}{2}} =  \langle M,M \rangle^{\frac{1}{2}}, \quad
    \| M \|_{\operatorname{op}} = \varsigma_1(M M^{*})^{\frac{1}{2}}, \quad \| M \|_{\infty} = \max_{ \substack{ 1 \leq i \leq l \\ 1 \leq j \leq m } } |M_{i,j}| \,.
\end{align*}
We will use $\mathbb{I}_{k}$ to denote the $k*k$ identity matrix (and we drop the subscript if the dimension is clear from the context). Similarly, we denote $\mathbb{O}_{k*l}$ the $k*l$ zero matrix and denote $\mathbb{J}_{k*l}$ the $k*l$ matrix with all entries being 1. We will abbreviate $\mathbb O_k = \mathbb O_{k*k}$ and $\mathbb J_k=\mathbb J_{k*k}$.  

Denote by $\mathsf K_n$ the complete graph with vertex set $[n]$. For any graph $H$, let $V(H)$ denote the vertex set of $H$ and let $E(H)$ denote the edge set of $H$. We say $H$ is a subgraph of $G$, denoted by $H\subset G$, if $V(H) \subset V(G)$ and $E(H) \subset E(G)$. For all $v \in V(H)$, define $\mathsf{deg}_H(v)=\#\{ e \in E(H): v \in e \}$ to be the degree of $v$ in $H$. We say $v$ is an isolated vertex of $H$, if $\mathsf{deg}_H(v)=0$. Let $\mathsf I(H)$ be the set of isolated vertices of $H$. Write $S \Subset \mathsf K_n$ if $S \subset \mathsf K_n$ and $\mathsf I(S)=\emptyset$. We say $v$ is a leaf of $H$, if $\mathsf{deg}_H(v)=1$. Denote by $\mathsf{L}(H)$ the set of leaves in $H$. For $H,S \subset \mathsf K_n$, denote by $H \cap S$ the graph with vertex set given by $V(H) \cap V(S)$ and edge set given by $E(H)\cap E(S)$, and denote by $S \cup H$ the graph with vertex set given by $V(H) \cup V(S)$ and edge set $E(H) \cup E(S)$. We say a subgraph $H \subset \mathsf K_n$ is a path with endpoints $u,v$ (possibly with $u=v$), if there exist distinct $w_1, \ldots, w_m \neq u,v$ such that $V(H)=\{ u,v,w_1,\ldots,w_m \}$ and $E(H)=\{ (u,w_1), (w_1,w_2) \ldots, (w_m,v) \}$. We say $H$ is a simple path if its endpoints $u \neq v$. Denote $\operatorname{EndP}(P)$ as the set of endpoints of a path $P$. We say a subgraph $H$ is an $m$-cycle if $V(H)=\{ v_1, \ldots, v_m \}$ and $E(H)=\{ (v_1,v_2), \ldots, (v_{m-1},v_m), (v_m,v_1) \}$. For a subgraph $K \subset H$, we say $K$ is an independent $m$-cycle of $H$, if $K$ is an $m$-cycle and no edge in $E(H)\setminus E(K)$ is incident to $V(K)$. Denote by $\mathtt{C}_m(H)$ the set of $m$-cycles of $H$ and denote by $\mathcal{C}_m(H)$ the set of independent $m$-cycles of $H$. For $H\subset S$, we define $\mathfrak{C}_{m}(S,H)$ to be the set of independent $m$-cycles in $S$ whose vertex set is disjoint from $V(H)$. Define $\mathfrak{C}(S,H)=\cup_{m\geq 3} \mathfrak{C}_m(S,H)$. Two graphs $H$ and $H'$ are isomorphic, denoted by $H\cong H'$, if there exists a bijection $\sigma:V(H) \to V(H')$ such that $(\sigma(u),\sigma(v)) \in E(H')$ if and only if $(u,v)\in E(H)$. Denote by $[H]$ the isomorphism class of $H$; it is customary to refer to these isomorphic classes as unlabeled hypergraphs. Let $\mathsf{Aut}(H)$ be the number of automorphisms of $H$ (graph isomorphisms to itself). 

For any two positive sequences $\{a_n\}$ and $\{b_n\}$, we write equivalently $a_n=O(b_n)$, $b_n=\Omega(a_n)$, $a_n\lesssim b_n$ and $b_n\gtrsim a_n$ if there exists a positive absolute constant $c$ such that $a_n/b_n\leq c$ holds for all $n$. We write $a_n=o(b_n)$, $b_n=\omega(a_n)$, $a_n\ll b_n$, and $b_n\gg a_n$ if $a_n/b_n\to 0$ as $n\to\infty$. We write $a_n =\Theta(b_n)$ if both $a_n=O(b_n)$ and $a_n=\Omega(b_n)$ hold.

\section{Main results}{\label{sec:main-results}}

We now state our main results. For simplicity, throughout the remainder of the paper, we will assume without further mention that Heuristic~\ref{conj-low-deg} holds for the corresponding testing problem defined in Definition~\ref{def-hidden-sample}, for each of the models considered in the following subsections. 

\subsection{Additive Gaussian model and binary observation model}{\label{subsec:Gaussian-Bernoulli-models}}

We first introduce the following general setting of \emph{Gaussian additive models}, which encompasses the planted submatrix problem as well as several other widely studied models, including the spiked Wigner \cite{FP07, CDF09} and positively spiked Wishart \cite{BBP05, BS06} models with an arbitrary prior on the planted vector (see, e.g., \cite{PWBM18a} and the references therein), as well as tensor PCA \cite{RM14}. A general low-degree analysis of the detection problem in additive Gaussian noise models was developed in \cite{KWB22}, while the aforementioned special cases were studied in \cite{HKP+17, Hopkins18, BKW20, DKWB24, LWB22}.
\begin{DEF}{\label{def-add-Gaussian-model}}
    In the general \emph{additive Gaussian model} we observe $\bm Y=\Theta+\bm W$ where $\Theta \in \mathbb R^N$ is drawn from an arbitrary (but known) prior, and $\bm W$ is i.i.d.\ $\mathcal N(0,1)$ independent from $\Theta$. Denote by $\mathbb P$ the law of $\bm Y=\Theta+\bm W$. For the detection problem, the goal is to distinguish $\mathbb P$ from $\mathbb Q$, the law of $\bm W$; for the recovery problem, the goal is to estimate $\Theta$ given $\bm Y \sim \mathbb P$.
\end{DEF}
\begin{DEF}{\label{def-recovery-add-Gaussian-model}}
    We say an estimator $\mathcal X:=\mathcal X(\bm Y) \in \mathbb R^{N}$ achieves \emph{weak recovery}, if 
    \begin{align}{\label{eq-def-weak-recovery-add-Gaussian}}
        \mathbb E_{\Pb}\Bigg[ \frac{ \langle \mathcal X, \Theta \rangle }{ \| \mathcal X \| \| \Theta \| } \Bigg] \geq c \mbox{ for some constant } c>0 \,.
    \end{align}
    Similarly, we say that $\mathcal X:=\mathcal X(\bm Y) \in \mathbb R^{N}$ achieves \emph{strong recovery}, if 
    \begin{align}{\label{eq-def-strong-recovery-add-Gaussian}}
        \mathbb E_{\Pb}\Bigg[ \frac{ \langle \mathcal X, \Theta \rangle }{ \| \mathcal X \| \| \Theta \| } \Bigg] \to 1 \mbox{ as } n \to \infty \,.
    \end{align}
\end{DEF}
We next consider a different setting, where the observed variables are binary-valued. This captures, for instance, various problems where the observation is a random graph.
\begin{DEF}{\label{def-Berboulli-obser-model}}
    The general \emph{binary observation model} is defined as follows. First, a signal $\Theta\in [0,1-p]^N$ is drawn from an arbitrary (but known) prior. We observe $\bm Y\in \{ 0,1 \}^N$ where $\mathbb E[\bm Y_i \mid \Theta]=p+\Theta_i$ and $\bm Y_i$ are conditionally independent given $\Theta$. Denote by $\mathbb P$ the law of $\bm Y$ generated as above. For the detection problem, the goal is to distinguish $\mathbb P$ with $\mathbb Q$, the law of $\bm Z \in \{ 0,1 \}^N$ with i.i.d.\ $\mathsf{Ber}(p)$ entries; for the recovery problem, the goal is to estimate $\Theta$ given $\bm Y \sim \mathbb P$. 
\end{DEF}
\begin{DEF}{\label{def-recovery-Bernoulli-obser-model}}
    We say an estimator $\mathcal X:=\mathcal X(\bm Y) \in \mathbb R^{N}$ achieves \emph{weak recovery}, if 
    \begin{align}{\label{eq-def-weak-recovery-Bernoulli-obser}}
        \mathbb E_{\Pb}\Bigg[ \frac{ \langle \mathcal X, \Theta \rangle }{ \| \mathcal X \| \| \Theta \| } \Bigg] \geq c \mbox{ for some constant } c>0 \,.
    \end{align}
    Similarly, we say that $\mathcal X:=\mathcal X(\bm Y) \in \mathbb R^{N}$ achieves \emph{strong recovery}, if 
    \begin{align}{\label{eq-def-strong-recovery-Bernoulli-obser}}
        \mathbb E_{\Pb}\Bigg[ \frac{ \langle \mathcal X, \Theta \rangle }{ \| \mathcal X \| \| \Theta \| } \Bigg] \to 1 \mbox{ as } n \to \infty \,.
    \end{align}
\end{DEF}

\subsection{Planted submatrix}{\label{subsec:main-result-planted-submatrix}}

\begin{DEF}[Planted submatrix model]{\label{def-planted-submatrix}}
    In the planted submatrix model, we observe the $n*n$ matrix 
    \begin{align*}
        \bm Y=\lambda \theta\theta^{\top}+\bm W \,,
    \end{align*}
    where $\lambda\geq 0$, and $\theta\in \{ 0,1 \}^n$ has i.i.d.\ $\mathsf{Ber}(\rho)$ entries for some $\rho\in(0,1)$. The noise matrix $\bm W$ is symmetric with entries satisfying $\bm W_{i,j}=\bm W_{j,i} \sim \mathcal N(0,1)$ for $i<j$ and $\bm W_{i,i}\sim\mathcal N(0,2)$, where the collection $\{ \bm W_{i,j}: i \leq j \}$ is independent. We assume throughout that the parameters $\lambda$ and $\rho$ are known. It is straightforward to see that $\bm Y$ is an additive Gaussian model with $\Theta=\lambda \theta\theta^{\top}$. 
\end{DEF}
Prior work has extensively studied this model and its variants \cite{KBRS11, ACD11, BI13, BIS15, MW15, CX16, DM14, CLR17, BBH18, GJS21}. The statistical limits of both detection \cite{BI13} and recovery \cite{KBRS11, BIS15} are by now well understood. The computational limits of detection (for an asymmetric version of the problem) have also been investigated under the planted clique hypothesis \cite{MW15, BBH18}. In addition, low-degree lower bounds for the quantity $\mathsf{MMSE}_{\leq D}$ were established in \cite{SW22, SW25}, providing evidence for a detection-recovery gap when $\rho\gg(\sqrt{n})^{-1}$.
\begin{thm}{\label{Main-thm-planted-submatrix}}
    Consider the planted submatrix problem in Definition~\ref{def-planted-submatrix}. Suppose that
    \begin{equation}{\label{eq-condition-planted-submatrix}}
        \frac{(\log n)^2}{\sqrt{n}} \ll \rho \ll \frac{1}{\log n} \mbox{ and } \lambda^2 \rho^2 n \ll \frac{1}{(\log n)^2} \,.
    \end{equation}
    Then, any algorithm that achieves weak recovery in the sense of Definition~\ref{def-recovery-add-Gaussian-model} requires running time $\exp(\Omega(\lambda^{-2}(\log n)^{-4}))$.
\end{thm}
The regime $\rho\gg(\sqrt{n})^{-1}$ is of particular interest because here a detection-recovery gap appears there. Specifically, when
\begin{align*}
    (\rho\sqrt{n})^{-2} \ll \lambda \ll (\rho\sqrt{n})^{-1} \,,
\end{align*}
detection is easy but no polynomial-time algorithm for weak recovery is known, and the best known algorithms in this regime run in time $\exp(\widetilde{O}(\lambda^{-2}))$ \cite{DKWB24, HSV20}. On the other hand, when $\lambda\gg (\rho\sqrt{n})^{-1}$, a simple spectral algorithm achieves strong recovery \cite{BWZ23}. Theorem~\ref{Main-thm-planted-submatrix} therefore provides a conditional lower bound within low-degree framework that strongly supports the optimality of the above algorithms by identifying the precise threshold $\lambda \approx (\rho\sqrt{n})^{-1}$ and the sharp degree requirement $\lambda^{-2}$, thereby recovering the predictions of \cite{SW22, SW25} based on bounds for $\mathsf{MMSE}_{\leq D}$. We also note that the lower bound in \cite{SW25} establishes an ``all-or-nothing'' phase transition for low-degree polynomials in the regime $\rho \gg \frac{1}{\sqrt{n}}$. More precisely, above the sharp threshold $\lambda>(\rho \sqrt{en})^{-1}$, an approximate message passing (AMP) algorithm achieves recovery in polynomial time; whereas below the threshold $\lambda<(\rho \sqrt{en})^{-1}$, there is evidence suggesting that any recovery algorithm requires running time $\exp(\widetilde{O}(\lambda^{-2}))$. It remains an open question whether our approach can be sharpened to recover this exact threshold $\lambda=(\rho \sqrt{en})^{-1}$.

In contrast, in the regime $\rho\ll(\sqrt{n})^{-1}$, we point out that a simple entrywise thresholding algorithm can exactly recover $\theta$ as long as $\lambda\gg 1$ \cite{KBRS11}. This regime does not have a detection-recovery gap, with the best known algorithms achieving exact recovery in runtime $\exp(\widetilde{O}(\lambda^{-2}))$ for all $(\sqrt{\rho n})^{-1} \ll \lambda = O(1)$ \cite{DKWB24, HSV20}. Finally, we note that the case $\rho=\Theta(1)$ has also been studied in prior work \cite{DM14, LKZ15, GJS21, MW25}. In this regime, the relevant scaling is $\lambda=\frac{c}{\sqrt{n}}$ for a constant $c$. The best known algorithm is again based on AMP \cite{DM14, LKZ15}, but in contrast to the all-or-nothing behavior described above, its mean squared error converges to a nontrivial constant depending on $c$ and $\rho$. It was shown in \cite{MW25} that constant-degree polynomials cannot improve upon the precise mean squared error achieved by AMP, and extending this result to higher degrees remains an interesting open problem. It has been shown that constant-degree polynomials cannot surpass the precise mean square error achieved by AMP \cite{MW25}, and extending this result to higher degree remains an interesting open question. Our results do not apply in the setting $\rho=\Theta(1)$, since in this case the trivial estimator $\mathcal X=\mathbbm{1}\mathbbm{1}^{\top}$ achieves weak recovery in the sense of Definition~\ref{def-recovery-add-Gaussian-model}.

\subsection{Planted dense subgraph}{\label{subsec:main-result-planted-dense-subgraph}}

\begin{DEF}[Planted dense subgraph model]{\label{def-planted-dense-subgraph}}
    For parameters $\rho\in[0,1]$ and $p_0,p_1\in[0,1]$, we observe a graph on $[n]$ with adjacency matrix $\bm Y=(\bm Y_{i,j})$ generated as follows.
    \begin{itemize}
        \item A planted signal $\theta\in\{ 0,1 \}^n$ is drawn with i.i.d.\ $\mathsf{Ber}(\rho)$ entries.
        \item Conditioned on $\theta$, we drawn 
        \begin{align*}
            \bm Y_{i,j} \sim \mathsf{Ber}(p_0+(p_1-p_0)\theta_i\theta_j)
        \end{align*}
        independently for each $i<j$.
    \end{itemize}
    It is straightforward to see that $\bm Y$ is a binary observation model with $p=p_0$ and $\Theta=(p_1-p_0)\theta\theta^{\top}$.
\end{DEF}
Without loss of generality we will assume $p_0\leq p_1$, since otherwise one can consider the complement graph instead. Compared with the planted submatrix model, the planted dense subgraph problem exhibits a broader range of behaviors, depending on the scaling of $p_0$ and $p_1$. A number of statistical results are known for this model \cite{AV14, VA15, CX16}, along with computational lower bounds for testing \cite{HWX15} and positive algorithmic results \cite{BCC+10, Ames13, CX16, Mon15}. Computational limits for estimation based on lower bounds for $\mathsf{MMSE}_{\leq D}$ were considered by \cite{SW22}, and \cite{SW25} provides a sharper refinement. In what follows, we focus on the regime $\rho\gg(\sqrt{n})^{-1}$, which is the most relevant for our purposes because of the conjectured detection-recovery gap (see \cite[Conjecture~2.2]{BBH18}). 
\begin{thm}{\label{Main-thm-planted-dense-subgraph}}
    Consider the planted dense subgraph model with $0 \leq p_0 \leq p_1 \leq 1$, and define 
    \begin{equation}{\label{eq-def-lambda}}
        \lambda = \frac{ p_1-p_0 }{ \sqrt{p_0(1-p_0)} } \,.
    \end{equation}
    Suppose that
    \begin{equation}{\label{eq-condition-planted-dense-subgraph}}
        \frac{1}{n} \ll p_0 \leq p_1 \ll 1, \quad \frac{(\log n)^2}{\sqrt{n}} \ll \rho \ll \frac{1}{\log n}, \quad \lambda^2 \rho^2 n \ll \frac{1}{(\log n)^2} \,.
    \end{equation}
    Then, any algorithm that achieves weak recovery in the sense of Definition~\ref{def-recovery-Bernoulli-obser-model} requires running time $\exp(\Omega(\lambda^{-2}(\log n)^{-4}))$.
\end{thm}
To simplify the discussion, we restrict attention to the regime considered in \cite{HWX15}, namely
\begin{align*}
    n^{-1} \ll p_0 \ll 1 \mbox{ and } p_1=cp_0 \mbox{ for some constant } c>1 \,.
\end{align*}
We further assume $(\sqrt{n})^{-1} \ll \rho \ll 1$ due to the emergence of detection-recovery gap. Under these assumptions, the results of \cite{SW22} provide matching upper and lower bounds for low-degree polynomials, with the threshold located at $n\rho^2 p_0 \approx 1$. The corresponding upper bound is achieved by a particularly simple algorithm that selects vertices of unusually large degree. Our result therefore recovers the predictions of \cite{SW22}. Another respect in which our result improves upon \cite{SW22} is the implied running-time lower bound. In the above scaling regime, we show that a running time of $\exp(\widetilde{\Theta}(1/p_0))$ may be necessary for any algorithm. This strengthens the conclusions of \cite{SW22} and matches the refined predictions of \cite{SW25}. We believe this bound is essentially optimal: although such a result does not appear explicitly in the literature, an algorithm with runtime $\exp(\widetilde{\Theta}(1/p_0))$ can likely be obtained when $(\rho n)^{-1} \ll p_0 \ll 1$ by a straightforward adaptation of the spiked Wigner results in \cite{DKWB24}.

As in the planted submatrix model, we note that the lower bound in \cite{SW25} establishes an ``all-or-nothing'' phase transition for low-degree polynomials in the regime $(\sqrt{n})^{-1}\ll\rho\ll 1$, with the sharp threshold located at $n\rho^2 p_0 = \frac{1}{e(c-1)^2}$. It is further expected that the matching upper bound is achieved by a polynomial-time algorithm based on an AMP-type approach, analogous to those in \cite{DM15, HWX18}. Whether our method can be sharpened to identify this exact threshold remains an open question.

Another interesting regime of the planted dense subgraph problem that we do not consider here is the log-density regime, where $\rho\ll (\sqrt{n})^{-1}$ and $p_1=n^{-\alpha},p_0=n^{-\beta}$ for constants $0<\alpha<\beta$ \cite{BCC+10}. In this regime, we do not expect a sharp threshold. The low-degree limits for testing have recently been characterized via a rather delicate conditioning argument \cite{DMW25}. Moreover, the best known algorithms for estimation match the lower bounds for testing, which is a priori an easier problem, suggesting that the thresholds for detection and recovery should coincide.

\subsection{Stochastic block model}{\label{subsec:main-result-SBM}}

\begin{DEF}[Stochastic block model]{\label{def-SBM}}
    Let $q \ge 2$ denote the number of communities. For an integer $n \ge 1$ and parameters $d>0,\lambda\in(0,1)$, we define the random graph $\bm Y$ as follows.
    \begin{itemize}
        \item Sample a community assignment $\sigma_* \in [q]^n$ uniformly at random.
        \item Conditional on $\sigma_*$, for each distinct pair $(i,j)\in \operatorname{U}_n$, let $\bm Y_{i,j}$ be an independent Bernoulli random variable, where $\bm Y_{i,j}=1$ indicates the presence of an undirected edge between $i$ and $j$. The conditional edge probability is
        \begin{align*}
            \Pb(\bm Y_{i,j}=1 \mid \sigma_*) = 
            \begin{cases}
                \frac{(1+(q-1)\lambda)d}{n} \,, & \sigma_*(i)=\sigma_*(j) \,; \\
                \frac{(1-\lambda)d}{n} \,, & \sigma_*(i) \neq \sigma_*(j) \,.
            \end{cases}
        \end{align*}
    \end{itemize}
    It is immediate that $\bm Y$ fits into the binary observation model framework, with $p=\frac{d}{n}$ and $\Theta_{i,j}=\frac{ (q \mathbf 1_{\sigma_*(i)=\sigma_*(j)}-1)\lambda d }{n}$.
\end{DEF}
The stochastic block model (SBM) is a special case of inhomogeneous random graphs \cite{BJR07}. It has been extensively studied as a model for community structure in statistics and the social sciences (see, e.g., \cite{HLL83, SN97, RCY11}), as well as for the analysis of clustering algorithms in theoretical computer science (see, e.g., \cite{DF89, CK01, McS01, Coj10}). We refer the reader to \cite{Abbe18} for a survey. We will mainly focus on the sparse regime, where $d$ and $\lambda$ are fixed constants, although our results apply more generally whenever $\lambda=o(n)$. In the sparse regime, the landmark work \cite{DKMZ11} first predicted a sharp computational phase transition at the so-called Kesten-Stigum (KS) threshold $d\lambda^2=1$, based on a heuristic analysis of the belief propagation (BP) algorithm. Originally identified by Kesten and Stigum \cite{KS66} in the context of multi-type branching processes, the KS threshold has since played an important role in several other areas such as phylogenetic reconstruction \cite{DMR11, RS17}.

A sequence of works has shown that polynomial-time algorithms achieve weak recovery above the KS threshold $d\lambda^2>1$ when $q=2$ \cite{Mas14, MNS18, BLM15}, and more generally when $q\ge 3$ is fixed \cite{AS15, AS18}. Below the KS threshold $d\lambda^2<1$, weak recovery is information-theoretically impossible for $q=2$ \cite{MNS15}, and also for $q=3,4$ when $d$ is sufficiently large \cite{MSS25a}. For $q\ge 5$, however, a statistical-computational gap emerges: no polynomial-time algorithm is known to succeed below the KS threshold, even though weak recovery is information-theoretically possible in this regime \cite{AS16, BMNN16, CKPZ18}. Nevertheless, it was conjectured in \cite{AS18}, based on the prediction of \cite{DKMZ11}, that no polynomial-time algorithm can achieve weak recovery when $d\lambda^2<1$. Existing low-degree lower bounds for hypothesis testing further support the conjectured hardness below the KS threshold \cite{HS17, BBK+21}. In this work, we provide evidence that algorithms with running time $\exp(n^{o(1)})$ fail to solve the weak recovery problem below the KS threshold $d\lambda^2<1$, provided that the number of communities satisfies $q\ll n^{1/8}$.
\begin{thm}{\label{Main-thm-SBM}}
    Consider the stochastic block model defined in Definition~\ref{def-SBM}. Suppose that 
    \begin{equation}
        2\leq q=n^{\frac{1}{8}-\Omega(1)} \mbox{ and } d\lambda^2<1-\delta \mbox{ for some constant } \delta>0 \,.
    \end{equation}
    Then, any algorithm that achieves weak recovery in the sense of Definition~\ref{def-recovery-Bernoulli-obser-model} requires running time $\exp(nq^{-6}(\log n)^{-3})$. In particular, when $q=n^{o(1)}$ and $d\lambda^2<1-\Omega(1)$, any algorithm that achieves weak recovery in the sense of Definition~\ref{def-recovery-Bernoulli-obser-model} requires running time $\exp(n^{1-o(1)})$.
\end{thm}
We note that when $q=\omega(1)$, a detection--recovery gap already appears. Indeed, in this regime, a simple triangle-counting statistic distinguishes the SBM from an \ER model with probability $1-o(1)$ for all fixed constants $\lambda,d>0$. Prior to our work, unconditional low-degree lower bounds on $\mathsf{MMSE}_{\leq D}$ for logarithmic degrees $D$ were obtained in \cite{LG24}, which also considered the case of a mildly growing number of communities and the related problem of graphon estimation. This result was substantially strengthened in \cite{SW25}, which provided low-degree evidence suggesting that, in the symmetric SBM with constant average degree $d$, community detection below the KS threshold requires running time $\exp(n^{\Omega(1)})$ whenever $q\ll n^{1/8}$. The techniques of \cite{SW25} were later extended in \cite{CMSW25} to show that the KS threshold remains the low-degree recovery threshold throughout the broader regime $q\ll \sqrt n$, and that this is tight in the sense that the KS threshold can be beaten in polynomial time when $q\gg \sqrt n$. An independent work \cite{DHSS25} likewise provides evidence for the hardness of weak recovery below the KS threshold for slowly growing $q=(\log n)^{O(1)}$, under a strengthened version of the low-degree conjecture proposed in \cite{MW23+}. Our result therefore recovers the prediction of \cite{SW25}, via an alternative and arguably simpler approach, and suggests that all recovery algorithms in this range require exponential running time when $q=n^{o(1)}$. By contrast, \cite{SW25, CMSW25} only rule out algorithms with running time $\exp(n^\iota)$ for some small constant $\iota$. Finally, we believe that our methods can be extended to the full regime $q\ll \sqrt n$. Doing so, however, would require suitable truncation arguments in order to obtain a sharp bound on the low-degree advantage, and we leave this direction for future work.
\begin{remark}{\label{rmk-extension-spiked-matrix}}
    Our results also extend readily to the spiked Wigner model, in which the observation takes the form
    \begin{align*}
        \bm Y=\frac{\lambda}{\sqrt{n}} \bm U \bm U^{\top} + \bm W \,,
    \end{align*}
    where $\bm W$ is a Wigner matrix and $\bm U \in \mathbb R^{n*m}$ with entries i.i.d.\ drawn from a prior $\pi$ with sub-Gaussian moments. Most prior work has focused on the case $m=O(1)$, and in particular on the rank-one setting $m=1$. In this regime, there is a sharp phase transition in the feasibility of detecting or estimating the signal $\bm U$, governed by a variant of the Baik–Ben Arous–P\'ech\'e (BBP) transition \cite{BBP05, FP07, BN11}. Above the BBP threshold $\lambda>1$, detection is possible using the top eigenvalue of the observation matrix, and moreover the top eigenvector of $\bm Y$ is nontrivially correlated with the true signal. By contrast, when $\lambda<1$, neither the largest eigenvalue nor its associated eigenvector carries reliable information about the signal in the high-dimensional limit.

    While on a statistical point of view, PCA can be improved upon for certain sparse priors \cite{PWBM18a, LM19, EKJ20}, it remains statistically optimal for many dense priors, where no statistic can surpass the spectral threshold \cite{DAM16, PWBM18a}. Examples of such settings include $\mathbb Z_2$ synchronization, angular synchronization, and many other synchronization and spiked random matrix models. On the computational side, \cite{KWB22, MW25} showed under mild assumptions on the prior that low-degree polynomials cannot surpass the spectral threshold. Our lower bound yields low-degree hardness for weak recovery when $\lambda<1$, suggesting that any recovery algorithm in this regime requires running time $\exp(n^{1-o(1)})$. This improves upon the prediction of \cite{SW25}, which only rules out estimation algorithms with running time $n^{\iota}$ for some small constant $\iota$. Indeed, the only additional ingredient needed in our framework is an upper bound on $\mathsf{Adv}(\Pb;\Qb)$, where $\Qb$ denotes the law of a Wigner matrix; such a bound was already established in \cite{BBK+21}, so we omit the details here.

    Our approach also applies to the setting of growing $m$, often referred to as symmetric matrix factorization, which has attracted recent interest \cite{Hua18, BABP16, MKMZ22, PBM24, BKR24}. In analogy with Theorem~\ref{Main-thm-SBM}, our method implies that the low-degree threshold for weak recovery remains at $\lambda=1$ as long as $m=n^{o(1)}$. Notably, \cite{SW25} provides evidence that the BBP threshold continues to hold throughout the broader regime $m=o(n)$. Whether our method can be extended to cover this full range remains an interesting open problem.
\end{remark}

\subsection{Multi-frequency angular synchronization}{\label{subsec:main-result-angular-synchronization}}

\begin{DEF}[Multi-frequency angular synchronization model]{\label{def-multi-frequency-group-synchronization}}
    We consider the synchronization problem over the isomorphic group
    \begin{equation}{\label{eq-def-group-S}}
        \mathbb S:=\mathsf{SO}(2) \cong \mathsf{U}(1) \,.
    \end{equation}
    Let $\bm x=(e^{i\varphi_1},\ldots,e^{i\varphi_n})$ such that $\varphi_1,\ldots,\varphi_n$ are drawn uniformly from $\mathsf{U}_1$ independently. We observe $L$ matrices given by
    \begin{equation}{\label{eq-def-multi-frequency-group-syncron}}
    \left\{ \begin{aligned}
        & \bm Y_1 = \frac{\lambda}{\sqrt{n}} \bm x\bm x^{*} + \bm W_1 \,, \\
        & \bm Y_2 = \frac{\lambda}{\sqrt{n}} \bm x^{(2)} \left(\bm x^{(2)} \right)^{*} + \bm W_2 \,, \\
        & \ldots \\
        & \bm Y_L = \frac{\lambda}{\sqrt{n}} \bm x^{(L)} \left( \bm x^{(L)} \right)^{*} + \bm W_L \,. \end{aligned} \right.
    \end{equation}
    Here $\bm x^{(k)}$ denotes the entrywise $k$-th power, and $\bm W_1,\ldots,\bm W_L$ are independent noise matrices sampled from $\mathsf{GUE}(n)$. In this setting, each $\bm Y_\ell$ fits into the additive Gaussian model, with signal matrix $\Theta_\ell = \frac{\lambda}{\sqrt{n}} \bm x^{(\ell)} (\bm x^{(\ell)})^{*}$.
\end{DEF}
In the single-frequency case $L=1$, the observation can also be viewed as a rank-one perturbation of a Wigner random matrix,
\begin{align}{\label{eq-def-spiked-Wigner}}
    \bm Y = \frac{\lambda}{\sqrt{n}} \bm x \bm x^* + \bm W \mbox{ where } \bm x^*=\overline{\bm x}^{\top} \,.
\end{align}
This setup, often called a spiked Wigner model, has been extensively studied in random matrix theory and statistics and exhibits a transition for estimation at the BBP threshold $\lambda=1$, as discussed in Remark~\ref{rmk-extension-spiked-matrix}. For the multi-frequency model, the statistical threshold for estimation was determined in \cite{PWBM18a} up to constant factors. Their results show that, for sufficiently large $L$, there exists a computationally inefficient algorithm that outperforms PCA applied to a single frequency. However, as in angular synchronization, it has remained unclear whether such an algorithm can be implemented efficiently, and more broadly, what fundamental limitations constrain computationally efficient algorithms in this setting. Since one can aggregate information across the $L$ frequencies, it is natural to ask whether reliable detection becomes possible at a lower signal level, potentially even for some $\lambda<1$. More generally, this leads to the following question for group synchronization problems: can efficient estimation succeed at a lower signal-to-noise ratio than in the single-frequency model? 

Using non‑rigorous statistical physics methods and numerical simulations, \cite{PWBM18b} initially predicted that the additional frequencies would not lower the computational threshold. A more recent rigorous analysis \cite{YWF25} confirmed the information-theoretic limits by deriving a replica formula for the asymptotic mutual information in the multi‑frequency model. Their analysis of the replica solution also led to conjectured phase transitions for computationally efficient algorithms. However, these predictions rely on the conjectured optimality of approximate message passing (AMP) algorithms, which is known not to capture the true computational barrier in all settings (see e.g., \cite{WEM19}). In \cite{KBK24+}, computational lower bounds were derived using the low‑degree polynomial framework, showing that for synchronization over $\mathbb S$, the BBP threshold $\lambda>1$ cannot be surpassed by efficient algorithms when the number of frequencies $L$ is any fixed constant. A key limitation of their approach, which relies on considering the easier detection task, is that it inherently restricts the analysis to the regime $L=O(1)$. Thus, (as stated again in \cite{BKMR25+}) it remains an intriguing question whether the BBP threshold can be surpassed by efficient algorithms when the number of frequencies $L=\omega(1)$. The main result in this work significantly improves the result in \cite{KBK24+}, suggesting that the BBP threshold represents an inherent computational barrier of all polynomial-time algorithms as long as the number of frequencies $L=n^{o(1)}$.
\begin{thm}{\label{Main-thm-angular-synchronization}}
    Consider the multi-frequency angular synchronization problem defined in Definition~\ref{def-multi-frequency-group-synchronization}. Suppose that 
    \begin{equation}
        L=n^{o(1)} \mbox{ and } \lambda<1-\delta \mbox{ for some constant } \delta>0 \,.
    \end{equation}
    Then, any algorithm for weak recovery of $\bm Y_{\ell}$ in the sense of Definition~\ref{def-recovery-add-Gaussian-model} requires runtime at least $\exp(n^{\Omega(1)})$.
\end{thm}
\begin{remark}{\label{rmk-reduce-to-detection}}
    The work \cite{KBK24+} provides evidence for the hardness of estimation by studying the easier detection problem. More precisely, let $\Pb=\Pb_{n,L,\lambda}$ denote the law of $(\bm Y_1,\ldots,\bm Y_L)$ defined in \eqref{eq-def-multi-frequency-group-syncron}, and let $\Qb=\Qb_{n,L}$ denote the law of $(\bm W_1,\ldots,\bm W_L)$. The authors of \cite{KBK24+} show, assuming the low-degree conjecture, that no efficient algorithm can distinguish $\Pb$ from $\Qb$ when $\lambda<1$ and $L=O(1)$. We remark, however, that this approach breaks down when $L=\omega(1)$. Indeed, in this regime, simply examining the Frobenius norms 
    \begin{align*}
        \left( \left\| \bm Y_1 \right\|_{\Fop}^2, \ldots, \left\| \bm Y_L \right\|_{\Fop}^2 \right)
    \end{align*}
    already allows one to distinguish $\Pb$ from $\Qb$ for any fixed $\lambda>0$. One may then ask whether this detection-recovery gap can be eliminated by rescaling the variance of the noise matrices $\bm W_\ell$. However, even after such a variance adjustment, detection remains possible for all $\lambda>0$ when $L=\omega(1)$, for instance by counting 4-cycles, 
    \begin{align*}
        \left( \sum_{i,j,k,m} \bm Y_\ell(i,j)\bm Y_\ell(j,k) \bm Y_\ell(k,m) \bm Y_\ell(m,i) \right)_{1 \leq \ell \leq L} \,.
    \end{align*}  
\end{remark}

\subsection{Orthogonal group synchronization}{\label{subsec:main-result-orthogonal-synchronization}}

\begin{DEF}[Orthogonal group synchronization model]{\label{def-orthogonal-group-synchronization}}
    For $d \in \mathbb N$, we denote the isomorphic groups by
    \begin{equation}{\label{eq-def-group-S-new}}
        \mathbb S:=\mathsf{O}(d) = \left\{ \bm O \in \mathbb R^{d*d}: \bm O^{\top} \bm O = \bm O \bm O^{\top} = \mathbb I_d \right\} \,.
    \end{equation}
    Let $(\bm O_1,\ldots,\bm O_n)$ be independent random matrices drawn uniformly from $\mathbb S$. We consider the orthogonal group synchronization problem, where the observations consist of $n^2$ matrices of size $d*d$, given by
    \begin{equation}{\label{eq-def-orthogonal-group-syncron}}
        \bm Y_{i,j} = \frac{\lambda}{\sqrt{n}} \bm O_i^{\top} \bm O_j + \bm Z_{i,j} \,,  
    \end{equation}
    where $\lambda>0$ and $\bm Z_{i,j}$ are independent $d*d$ GOE matrices. Defining $\bm U=(\bm O_1,\ldots,\bm O_n) \in \mathbb R^{d*nd}$, we can equivalently write the observation in matrix form as
    \begin{equation}{\label{eq-def-orthogonal-group-syncron-matrix-form}}
        \bm Y = \frac{\lambda}{\sqrt{n}} \bm U^{\top} \bm U+ \bm Z \,.
    \end{equation}
    In this case, it is clear that $\bm Y$ is an additive Gaussian model, with $\Theta = \frac{\lambda}{\sqrt{n}} \bm U^{\top} \bm U$.
\end{DEF}
The orthogonal group synchronization problem is motivated by the Fourier decomposition of the nonlinear objective in the non-unique games problem \cite{BCLS20}; see \cite[Remark~1]{KBK24+} for further discussion of this connection. It has found a wide range of applications, including computer vision \cite{OVBS17, SHSS16}, robotics \cite{RCBJ19}, and cryo-electron microscopy \cite{Singer18}. For synchronization problems over any continuous groups, Cram\'er–Rao bounds on the estimation error are established in \cite{BSAB14}, but few lower bounds that entirely preclude recovery are known. For the Gaussian synchronization model over $\mathsf{O}(d)$, the recovery threshold was predicted in \cite{JMR16} using methods from statistical physics; related Gaussian models had also appeared earlier in \cite{BBS17, JM13}. In \cite{PWBM16+}, the authors generalized this line of work by introducing a broad class of Gaussian synchronization models based on representation theory, and derived general statistical lower bounds applicable to arbitrary compact groups. They also observed that a simple spectral algorithm achieves nontrivial estimation whenever $\lambda>1$. Using non-rigorous methods from statistical physics together with numerical simulations, \cite{PWBM18b} further suggested that this spectral threshold constitutes a fundamental barrier for all efficient recovery algorithms. More recently, \cite{YWF25} rigorously confirmed the information-theoretic limits by deriving a replica formula for the asymptotic mutual information in the multi-frequency model. Their analysis of the replica solution also led to conjectured phase transitions for computationally efficient algorithms. However, these predictions rely on the conjectured optimality of approximate message passing (AMP), which is known not to capture the true computational threshold in all settings; see, for example, \cite{WEM19}. Our results support the optimality of spectral threshold with the low-degree framework when the dimension $d=n^{o(1)}$.
\begin{thm}{\label{Main-thm-group-synchronization}}
    Consider the orthogonal group synchronization problem defined in Definition~\ref{def-orthogonal-group-synchronization}. Suppose that
    \begin{equation}
        d=n^{o(1)} \mbox{ and } \lambda<1-\delta \mbox{ for some constant } \delta>0 \,.
    \end{equation}
    Then, any algorithm for weak recovery in the sense of Definition~\ref{def-recovery-add-Gaussian-model} requires runtime at least $\exp(n^{\Omega(1)})$.
\end{thm}
Note that when $d=\omega(1)$ a detection-recovery gap arises, since simply examining the Frobenius norm $\|\bm Y\|_{\Fop}^2$ already distinguishes $\Pb$ (the law of the orthogonal group synchronization model) and $\Qb$ (the law of a pure Gaussian matrix) for any fixed $\lambda>0$. Nevertheless, to the best of our knowledge, our results are new even for every fixed $d \geq 3$, where no detection-recovery occurs. In this regime, Lemma~\ref{lem-bound-adv-group-synchronization} yields a matching lower bound for detection. Before our work, computational lower bounds were known only for the cases $d=1$ (which correspond to $\mathbb Z_2$ synchronization) and $d=2$ (which correspond to angular synchronization discussed in Section~\ref{subsec:main-result-angular-synchronization}), where no statistical-computational gap is present. By contrast, when $d$ is a sufficiently large constant it was shown in \cite{PWBM16+} that a computationally inefficient algorithm can succeed below the spectral threshold.

\subsection{Multi-layer stochastic block models}{\label{subsec:main-result-multilayer-SBM}}

\begin{DEF}[Multi-layer stochastic block models]{\label{def-multilayer-SBM}}
    Given integers $n,q,L \in \mathbb N$, a correlation parameter $\rho>0$ and two parameter families $\{ d_{\ell}: 1 \leq \ell \leq L \}, \{ \lambda_\ell \in (0,1): 1 \leq \ell \leq L \}$, we generate $L$ random graphs $\bm Y_1,\ldots,\bm Y_L$ identified by their adjacency matrices as follows. 
    \begin{itemize}
        \item Sample a labeling $\sigma \in [q]^n$ uniformly at random.
        \item Conditioned on $\sigma$, sample i.i.d.\ random vectors $\sigma_1,\ldots,\sigma_L \in [q]^n$ such that 
        \begin{equation}{\label{eq-def-sigma-ell}}
            \begin{aligned}
                &\Pb\left( \sigma_{\ell}(i)=\sigma(i) \right)= \frac{1+(q-1)\rho}{q} \,, \\ 
                &\Pb\left( \sigma_{\ell}(i)=\tau \right) = \frac{1-\rho}{q} \mbox{ for all } \tau\in[q], \tau \neq \sigma(i)
            \end{aligned} 
        \end{equation}
        independently for all $1 \leq \ell \leq L, 1 \leq i \leq n$.
        \item Conditioned on $\sigma,\sigma_1,\ldots,\sigma_{L}$, the graphs $\bm Y_1,\ldots,\bm Y_L$ are independent and satisfy 
        \begin{equation}{\label{eq-def-multilayer-SBM}}
            \bm Y_{\ell} \sim \mathcal S(n,d_{\ell},\lambda_{\ell}) \mbox{ with labeling } \sigma_{\ell}\,, 1 \leq \ell \leq L \,.
        \end{equation}
    \end{itemize}
    Let $\Pb=\Pb_{n,q,\rho,\{ \lambda_{\ell} \},\{ \epsilon_{\ell} \}}$ denote the law of $(\bm Y_1,\ldots,\bm Y_L)$. In this case, each $\bm Y_\ell$ fits into the binary observation model with 
    \begin{align*}
        p=\frac{d_\ell}{n} \mbox{ and } \Theta_\ell(i,j)=\frac{ (q \mathbf 1_{\sigma_\ell(i)=\sigma_\ell(j)}-1)\lambda_\ell d_\ell }{n} \,.
    \end{align*} 
\end{DEF}
The multi-layer stochastic block model provides a natural extension of the classical stochastic block model, motivated by the problem of integrating information from multiple network layers to recover a common underlying community structure. Community detection in multilayer networks has recently received substantial attention. A large portion of the existing literature, particularly the early work, models multilayer networks using the stochastic block model (SBM) and assumes that the community assignments are identical across all layers; see, for example, \cite{BC20+, LCL20, MN23}. As noted in \cite{VMGP16, CLM22}, however, this assumption is often unrealistic in practical applications, where the community assignments across layers are more naturally modeled as correlated rather than identical. To capture this feature, \cite{CLM22} introduced a tractable model with correlated community assignments as in Definition~\ref{def-multilayer-SBM} and analyzed a two-step procedure consisting of spectral clustering followed by a maximum a posteriori (MAP) refinement. Regarding this model, most previous results are restricted to the two-community setting \cite{LZZ24, YLS25, GHL26+}, where no statistical-computational gap arises. Moreover, \cite{YLS25, GHL26+} showed that in this setting the estimation threshold is given by the generalized Kesten–Stigum (KS) threshold
\begin{align}{\label{eq-generalized-KS}}
    F(\rho,\{d_\ell\},\{\lambda_\ell\}):=\max\left\{ \max_{1 \leq \ell \leq L} \big\{ \lambda_{\ell}^2 d_{\ell} \big\}, \sum_{\ell=1}^{L} \frac{ \rho^4 \lambda_{\ell}^2 d_{\ell} }{ 1-(1-\rho^4) \lambda_{\ell}^2 d_{\ell} } \right\} = 1 \,.  
\end{align}
For a larger number of communities, one expects further statistical-computational gaps to arise, consistent with analogous phenomena in simpler models \cite{AS18, MSS25a}. Our results provide evidence that the generalized KS threshold remains the universal computational barrier for estimation, provided that both the number of communities and the number of layers are $n^{o(1)}$.
\begin{thm}{\label{Main-thm-multilayer-SBM}}
    Consider the multi-layer stochastic block model defined in Definition~\ref{def-multilayer-SBM}. Suppose that $q,L=n^{o(1)}$ and $F(\rho,\{d_\ell\},\{\lambda_\ell\})<1-\delta$ for some constant $\delta>0$. Then, any algorithm for weak recovery of $\bm Y_\ell$ in the sense of Definition~\ref{def-recovery-Bernoulli-obser-model} requires runtime at least $\exp(n^{\Omega(1)})$.
\end{thm}
Although this does not appear to have been stated explicitly in the literature, we expect that the path-counting algorithm of \cite{GHL26+} can be generalized to yield an efficient recovery algorithm whenever $F(\rho,\{d_\ell\},\{\lambda_\ell\})>1$, at least when $q$ is fixed. Taken together, our results therefore suggest that the generalized KS threshold should characterize the computational threshold for all $q,L=n^{o(1)}$. Finally, we note that a detection-recovery gap arises when $q=\omega(1)$, as discussed in Section~\ref{subsec:main-result-SBM}. We also expect such a gap to appear when $L=\omega(1)$. Indeed, if $L=\omega(1)$ and $\rho=0$, then a simple triangle count can distinguish $\Pb$ and $\Qb$ if $\min \{ \lambda_\ell^2 d_\ell \} =\Omega(1)$.

\section{Planted submatrix}{\label{sec:proof-planted-submatrix}}

This section is devoted to the proof of Theorem~\ref{Main-thm-planted-submatrix}. We will follow the strategy from Section~\ref{subsec:proof-overview}. The main steps were already outlined in Section~\ref{subsubsec:overview-strategy} and we fill in the remaining details here.

\subsection{Reducing to an imbalanced detection problem}{\label{subsec:reduce-to-detection-planted-submatrix}}

We first argue that the following weaker result suffices to imply Theorem~\ref{Main-thm-planted-submatrix}.
\begin{lemma}{\label{lem-simplified-goal-planted-submatrix}}
    Suppose that 
    \begin{equation}{\label{eq-condition-planted-submatrix-detail}}
        \frac{(\log n)^2}{\sqrt{n}} \ll \rho \ll \frac{1}{\log n} \mbox{ and } \lambda^2 \rho^2 n = \delta \mbox{ for some } \rho^2 \ll \delta \ll \frac{1}{(\log n)^2} \,.
    \end{equation}
    Then, any algorithm achieving weak recovery in the sense of Definition~\ref{def-recovery-add-Gaussian-model} requires running time $\exp(\Omega(\lambda^{-2}(\log n)^{-2}\delta))$.
\end{lemma}
We first show how to derive Theorem~\ref{Main-thm-planted-submatrix} from Lemma~\ref{lem-simplified-goal-planted-submatrix}.
\begin{proof}[Proof of Theorem~\ref{Main-thm-planted-submatrix} assuming Lemma~\ref{lem-simplified-goal-planted-submatrix}]
    Let $\gamma=\lambda^2 \rho^2 n=o(1)$. Suppose on the contrary that there exists an algorithm $\mathcal A$ with running time $\exp(\iota\lambda^{-2}(\log n)^{-2})$ for some $\iota=o((\log n)^{-2})$ that takes $\bm Y \sim \Pb_{\lambda,\rho,n}$ as input and achieves weak recovery in the sense of Definition~\ref{def-recovery-planted-submatrix}. Without loss of generality, we may assume that $\iota\gg (\log n)^{-3}$ and thus 
    \begin{align*}
        \lambda^{-2}(\log n)^{-2}\iota = \delta^{-1} \rho^2 n (\log n)^{-2} \iota \overset{\eqref{eq-condition-planted-submatrix-detail}}{\gg} (\log n)^4 \iota \gg \log n \,.
    \end{align*}
    Choose $N=N_n \in\mathbb N$ such that 
    \begin{align*}
        \lambda^2 \rho^2 N = \delta \mbox{ such that } \gamma,\iota \ll \delta \ll \frac{1}{(\log n)^2} \,.
    \end{align*}
    Using Lemma~\ref{lem-monotonicity-parameters}, we know that there is an algorithm $\mathcal A'$ with running time
    \begin{align*}
        \exp\left( \lambda^{-2}(\log n)^{-2}\iota \right) \cdot \operatorname{poly}(n) = \exp\left( O(\lambda^{-2}(\log N)^{-2}\iota) \right)
    \end{align*}
    that takes $\bm Y' \sim \Pb_{\lambda,\rho,N}$ as input and achieves weak recovery in the sense of Definition~\ref{def-recovery-add-Gaussian-model}. However, from Lemma~\ref{lem-simplified-goal-planted-submatrix} we have that all weak recovery algorithms for $\bm Y' \sim \Pb_{\lambda,\rho,N}$ requires time at least $\exp(\Omega(\lambda^{-2}(\log N)^{-2}\delta))$. This leads to a contradiction and thus proves Theorem~\ref{Main-thm-planted-submatrix}.
\end{proof}
From now on we will focus on Lemma~\ref{lem-simplified-goal-planted-submatrix}. The following lemma enables us to instead consider a detection problem with lopsided success probability.
\begin{lemma}{\label{lem-imbalanced-detection-planted-submatrix}}
    Let $\Qb=\Qb_n$ be the law of $\bm W$ in Definition~\ref{def-planted-submatrix}. Also fix a small constant $\kappa>0$ and let $\lambda'=\sqrt{1+\kappa^2} \lambda$. Assuming \eqref{eq-condition-planted-submatrix-detail} and suppose that there exists an algorithm with running time $\exp(o(\lambda^{-2}(\log n)^{-2}\delta))$ that takes $\bm Y\sim\Pb_{\lambda,\rho,n}$ as input and achieves weak recovery in the sense of Definition~\ref{def-recovery-add-Gaussian-model}. Then there is an $(T_n;c_n;\epsilon_n)$-test between $\Pb_{\lambda',\rho,n}$ and $\Qb_n$ with
    \begin{align*}
        T_n = o\left( \lambda^{-2}(\log n)^{-3}\delta \right); \quad c_n=\Omega(1); \quad \epsilon_n= \exp\left( -\Omega(\delta n) \right) \,.
    \end{align*}
\end{lemma}
\begin{proof}
    Recall that for $\bm Y\sim\Pb_{\lambda,\rho,n}$, $\bm Y$ is an additive Gaussian model with $\Theta=\lambda'\theta\theta^{\top}$, where $\theta\in \mathbb R^n$ has i.i.d.\ $\mathsf{Ber}(\rho)$ entries. It is straightforward that
    \begin{align*}
        \Pb_{\lambda,\rho,n}\left( \frac{M}{2} \leq \| \Theta \| \leq 2M \right)=1-o(1) \mbox{ where } M^2 = (1+\kappa^2) \lambda^2 \rho^2 n^2 \overset{\eqref{eq-condition-planted-submatrix-detail}}{=} \Theta(\delta n) \,.
    \end{align*}
    Thus, from Lemma~\ref{lem-reduce-to-detection-add-Gaussian-model} we see that there is an $(T_n;c_n;\epsilon_n)$-test between $\Pb_{\lambda',\rho,n}$ and $\Qb_n$ with
    \begin{equation*}
        T_n = o\left( \lambda^{-2}(\log n)^{-3}\delta \right); \quad c_n=\Omega(1); \quad \epsilon_n= \exp\left( -\Omega(\delta n) \right) \,.  \qedhere
    \end{equation*}
\end{proof}

\subsection{Bounding the low-degree advantage}{\label{subsec:bounding-low-deg-adv-planted-submatrix}}

\begin{lemma}{\label{lem-bound-adv-planted-submatrix}}
    Assuming \eqref{eq-condition-planted-submatrix-detail} and let $\lambda'=\sqrt{1+\kappa^2} \lambda$ for some small constant $\kappa>0$. We then have
    \begin{equation}{\label{eq-bound-adv-planted-submatrix}}
        \mathsf{Adv}_{\leq D}(\Pb_{\lambda',\rho,n};\Qb_n)^2 \leq \exp\left( \Theta(1) \cdot \lambda^{-2} \delta^{2} \right) \mbox{ for all } D \leq \lambda^{-2}(\log n)^{-2}\delta \,.
    \end{equation}
\end{lemma}
\begin{proof}
    We will follow the proof in Lemma~\ref{lem-bound-low-deg-adv}. Recall \eqref{eq-exp-leq-D}. Using \cite{KWB22}, we have
    \begin{align}{\label{eq-adv-planted-submatrix-relax-1}}
        \mathsf{Adv}_{\leq D}(\Pb_{\lambda',\rho,n};\Qb_n) &\leq \mathbb E_{\theta,\theta'}\Big[ \exp_{\leq D}\left\{ (\lambda')^2 \langle \theta,\theta' \rangle^2 \right\} \Big] \nonumber \\
        &= \mathbb E_{\theta,\theta'}\Big[ \exp_{\leq D}\left( (1+\kappa^2)\lambda^2 \langle \theta,\theta' \rangle^2 \right) \Big] \,,
    \end{align}
    where $\theta,\theta'$ are independent random vectors with i.i.d.\ $\mathsf{Ber}(\rho)$ entries. It suffices to bound the right hand side of \eqref{eq-adv-planted-submatrix-relax-1}. Clearly, we have
    \begin{align}{\label{eq-adv-planted-submatrix-relax-2}}
        \eqref{eq-adv-planted-submatrix-relax-1} = \sum_{k=0}^{D} \frac{ (1+\kappa^2)^k \lambda^{2k} \mathbb E_{\theta,\theta'}[ \langle \theta,\theta' \rangle^{2k} ] }{ k! } \,.
    \end{align}
    Note that $\langle \theta,\theta' \rangle \sim \mathsf{Bin}(n,\rho^2)$, we then have
    \begin{align*}
        \mathbb E_{\theta,\theta'}\left[ \langle \theta,\theta' \rangle^{2k} \right] &\leq (2n\rho^2)^{2k} + n^{2k} \cdot \Pb_{\theta,\theta'}\left( \langle \theta,\theta' \rangle>2n\rho^2  \right) \\
        &\leq (2n\rho^2)^{2k} + n^{2k} \cdot \exp\left( -\Theta(n\rho^2) \right) \\
        &= (2n\rho^2)^{2k} + \exp\left( 2k\log n - \Theta(n\rho^2) \right) \\
        &\leq (2n\rho^2)^{2k} + \exp\left( 2\lambda^{-2}(\log n)^{-1}\delta - \Theta(n\rho^2) \right) \\
        &= (2n\rho^2)^{2k}+o(1) \,,
    \end{align*}
    where the first inequality follows from $\langle \theta,\theta' \rangle \leq n$ and the standard Markov bound, the second inequality follows from the standard Chernoff's bound, the third inequality follows from $k \leq D \leq \lambda^{-2}(\log n)^{-2}\delta$, and the last equality follows from \eqref{eq-condition-planted-submatrix-detail}. Thus, we have
    \begin{align*}
        \eqref{eq-adv-planted-submatrix-relax-2} &\leq \sum_{k=0}^{D} \frac{ (1+\kappa^2)^k \lambda^{2k} }{ k! } \cdot (2n\rho^2)^{2k} \leq \exp\left( (1+\kappa^2)\lambda^2 (2n\rho^2)^2 \right) \\
        &= \exp\left( \Theta(1) \cdot n^2 \lambda^2 \rho^4 \right) = \exp\left( \Theta(1) \cdot \lambda^{-2}\delta^{2} \right) \,,
    \end{align*}
    where the last equality follows from \eqref{eq-condition-planted-submatrix-detail}.
\end{proof}

\subsection{Putting it together}{\label{subsec:conclusion-planted-submatrix}}

We can now conclude Lemma~\ref{lem-simplified-goal-planted-submatrix}, thereby finishing the proof of Theorem~\ref{Main-thm-planted-submatrix}.
\begin{proof}[Proof of Lemma~\ref{lem-simplified-goal-planted-submatrix}]
    Suppose on the contrary that there exists an algorithm $\mathcal A$ with running time $\exp(o(\lambda^{-2}(\log n)^{-2}\delta))$ that takes $\bm Y\sim \Pb_{\lambda,\rho,n}$ as input and achieves weak recovery in the sense of Definition~\ref{def-recovery-planted-submatrix}. Using Lemma~\ref{lem-reduce-to-one-sided-test-planted-submatrix}, we see that there exists an $(T_n;c_n;\epsilon_n)$-test between $\Pb_{\lambda',\rho,n}$ and $\Qb_n$ with
    \begin{align*}
        T_n = o\left(\lambda^{-2}(\log n)^{-3}\delta\right), \quad c_n=\Omega(1), \quad \epsilon_n = \exp(-\delta n) \,.
    \end{align*}
    
    However, using Lemma~\ref{lem-bound-adv-planted-submatrix} and Proposition~\ref{thm-alg-contiguity}, we see that (assuming the low-degree heuristic) there is no $(T_n';c_n';\epsilon_n')$-test between $\Pb_{\lambda',\rho,n}$ and $\Qb_n$ such that 
    \begin{equation*}
        T_n' = o\left(\lambda^{-2}(\log n)^{-3}\delta\right), \quad c_n'=\Omega(1), \quad \epsilon_n' = \exp\left( -\omega(\lambda^{-2}\delta^2) \right) \,.
    \end{equation*}
    This forms a contradiction since from \eqref{eq-condition-planted-submatrix-detail} we have $\delta n \gg \lambda^{-2} \delta^2$, and thus yields Lemma~\ref{lem-simplified-goal-planted-submatrix}.
\end{proof}

\section{Planted dense subgraph}{\label{sec:proof-planted-dense-subgraph}}

This section is devoted to the proof of Theorem~\ref{Main-thm-planted-dense-subgraph}.

\subsection{Reducing to an imbalanced detection problem}{\label{subsec:reduce-to-detection-planted-dense-subgraph}}

Similar to Section~\ref{subsec:reduce-to-detection-planted-submatrix}, we first argue that the following weaker result suffices to imply Theorem~\ref{Main-thm-planted-dense-subgraph}.
\begin{lemma}{\label{lem-simplified-goal-planted-dense-subgraph}}
    Suppose that 
    \begin{equation}{\label{eq-condition-planted-dense-subgraph-detail}}
        \frac{1}{n} \ll p_0 \leq p_1 \ll 1,\ \frac{\log n}{\sqrt{n}} \ll \rho \ll \frac{1}{\log n} \mbox{ and } \lambda^2 \rho^2 n = \delta \mbox{ for some } \rho^2 \ll \delta \ll \frac{1}{(\log n)^2} \,.
    \end{equation}
    Then, any algorithm achieving weak recovery in the sense of Definition~\ref{def-recovery-Bernoulli-obser-model} requires running time $\exp(\Omega(\lambda^{-2}(\log n)^{-2}\delta))$.
\end{lemma}
We first show how to deduce Theorem~\ref{Main-thm-planted-dense-subgraph} from Lemma~\ref{lem-simplified-goal-planted-dense-subgraph}. The key observation is the following analogue of Lemma~\ref{lem-monotonicity-parameters}. Let $\Pb_{p_0,p_1,\rho,n}$ denote the law of a planted dense subgraph on $[n]$ with parameters $p_0=p_{0,n},p_1=p_{1,n},\rho=\rho_n$.
\begin{lemma}{\label{lem-monotonicity-parameters-addition}}
    Suppose that an algorithm $\mathcal A$ with running time $n^{T}$ takes $\bm Y\sim \Pb_{p_0,p_1,\rho,n}$ as input and achieves weak recovery in the sense of Definition~\ref{def-recovery-Bernoulli-obser-model}. Then, for any $K=K_n\in \mathbb N$, there exists an algorithm $\mathcal A'$ with running time $K n^{T}$ that takes $\bm Y \sim \Pb_{p_0,p_1,\rho,Kn}$ as input and achieves weak recovery in the sense of Definition~\ref{def-recovery-Bernoulli-obser-model}.
\end{lemma}
\begin{proof}
    It is straightforward to verify that for any $n*n$ principal submatrix $\bm Y'$ of $\bm Y \sim \Pb_{p_0,p_1,\rho,Kn}$, we have $\bm Y' \sim \Pb_{p_0,p_1,\rho,n}$. Thus, weak recovery is easily achieved by applying $\mathcal A$ to each of the $K$ disjoint principal submatrices of $\bm Y$.
\end{proof}
We first prove Theorem~\ref{Main-thm-planted-dense-subgraph} based on Lemma~\ref{lem-simplified-goal-planted-dense-subgraph}.
\begin{proof}[Proof of Theorem~\ref{Main-thm-planted-dense-subgraph} assuming Lemma~\ref{lem-simplified-goal-planted-dense-subgraph}]
    Let $\gamma=\lambda^2 \rho^2 n=o((\log n)^{-2})$. Suppose on the contrary that there exists an algorithm $\mathcal A$ with running time $\exp(\iota\lambda^{-2}(\log n)^{-2})$ for some $\iota=o((\log n)^{-2})$ that takes $\bm Y \sim \Pb_{p_0,p_1,\rho,n}$ as input and achieves weak recovery in the sense of Definition~\ref{def-recovery-Bernoulli-obser-model}. Without loss of generality, we may assume that $\iota\gg (\log n)^{-3}$ and thus 
    \begin{align*}
        \lambda^{-2}(\log n)^{-2}\iota = \delta^{-1} \rho^2 n (\log n)^{-2} \iota \overset{\eqref{eq-condition-planted-dense-subgraph-detail}}{\gg} (\log n)^4 \iota \gg \log n \,.
    \end{align*}
    Choose $N=N_n \in\mathbb N$ such that 
    \begin{align*}
        \lambda^2 \rho^2 N = \delta \mbox{ such that } \gamma,\iota \ll \delta \ll 1 \,.
    \end{align*}
    Using Lemma~\ref{lem-monotonicity-parameters-addition}, we know that there is an algorithm $\mathcal A'$ with running time
    \begin{align*}
        \exp\left( \lambda^{-2}(\log n)^{-2}\iota \right) \cdot \operatorname{poly}(n) = \exp\left( O(\lambda^{-2}(\log N)^{-2}\iota) \right)
    \end{align*}
    that takes $\bm Y' \sim \Pb_{p_0,p_1,\rho,N}$ as input and achieves weak recovery in the sense of Definition~\ref{def-recovery-Bernoulli-obser-model}. However, from Lemma~\ref{lem-simplified-goal-planted-dense-subgraph} we have that all weak recovery algorithms for $\bm Y' \sim \Pb_{p_0,p_1,\rho,N}$ requires time at least $\exp(\Omega(\lambda^{-2}(\log N)^{-2}\delta))$. This leads to a contradiction and thus proves Theorem~\ref{Main-thm-planted-dense-subgraph}.
\end{proof}
From now on we will focus on Lemma~\ref{lem-simplified-goal-planted-dense-subgraph}. The following lemma enables us to instead consider a detection problem with lopsided success probability.
\begin{lemma}{\label{lem-imbalanced-detection-planted-dense-subgraph}}
    Fix a constant $\kappa>1$ close to $1$, and let $p_0'=\kappa p_0, p_1'=\kappa p_1$. Also define $\Qb=\Qb_{p_0',n}$ be the law of an \ER graph on $[n]$ with parameter $p_0'$. Assuming \eqref{eq-condition-planted-dense-subgraph-detail} and suppose that there exists an algorithm with running time $\exp(o(\lambda^{-2}(\log n)^{-2}\delta))$ that takes $\bm Y\sim\Pb_{p_0,p_1,\rho,n}$ as input and achieves weak recovery in the sense of Definition~\ref{def-recovery-Bernoulli-obser-model}. Then there is an $(T_n;c_n;\epsilon_n)$-test between $\Pb_{p_0',p_1',\rho,n}$ and $\Qb_{p_0',n}$ with
    \begin{align*}
        T_n = o\left( \lambda^{-2}(\log n)^{-3}\delta \right); \quad c_n=\Omega(1); \quad \epsilon_n= \exp\left( -\Omega(1)\cdot\frac{n\delta p_0}{p_1} \right) \,.
    \end{align*}
\end{lemma}
\begin{proof}
    Define
    \begin{equation}{\label{eq-def-mathcal-K-planted-subgraph}}
        \mathcal K=\Big\{ \mathcal Y \in \mathbb R^{n*n}: \| \mathcal Y \|_{\infty} \leq p_1,\ \mathcal Y \succ \mathbb O \Big\} \,.
    \end{equation}
    Also define $M \in \mathbb R$ such that
    \begin{equation}{\label{eq-def-M-planted-subgraph}}
        M^2 = (p_1-p_0)^2 \rho^2 n^2 = [1+o(1)] \lambda^2 \rho^2 n^2 p_0 \,.
    \end{equation}
    Recall that under $\Pb_{p_0,p_1,\rho,n}$, $\bm Y$ is a binary observation model with $\Theta=(p_1-p_0)\theta\theta^{\top}$. It is clear that
    \begin{align*}
        \Pb\left( \frac{M}{2} \leq \|\Theta\|_{\Fop} \leq 2M, \Theta\in\mathcal K \right)=1-o(1) \,.
    \end{align*}
    Suppose an algorithm $\mathcal X(\bm Y)$ with running time $\exp(o(\lambda^{-2}(\log n)^{-2}\delta))$ that takes $\bm Y\sim \Pb_{p_0,p_1,\rho,n}$ as input and achieves weak recovery in the sense of Definition~\ref{def-recovery-Bernoulli-obser-model}. Using Lemma~\ref{lem-cor-pre-proj}, we see that there exists an algorithm $\widehat{\mathcal X}(\bm Y)$ with running time $\exp(o(\lambda^{-2}(\log n)^{-2}\delta))$ that takes $\bm Y\sim \Pb_{p_0,p_1,\rho,n}$ as input and
    \begin{align*}
        \widehat{\mathcal X} \in \mathcal K, \quad \|\widehat{\mathcal X}\|_{\Fop} \geq \Omega(M), \quad \Pb_{p_0,p_1,\rho,n}\left[ \frac{ \langle \Theta,\mathcal X \rangle }{ \|\Theta\|_{\Fop} \| \mathcal X \|_{\Fop} } \geq \Omega(1) \right] \geq \Omega(1) \,.
    \end{align*}
    In particular, we have $\|\widehat{\mathcal X}\|_{\Fop} \geq \Theta(M)$ and $\|\widehat{\mathcal X}\|_{\infty}\leq p_1$. Applying Lemma~\ref{lem-reduce-to-detection-Ber-obser-model} with $a=b=\kappa^{-1}$, we see that there is an $(T_n;c_n;\epsilon_n)$-test between $\Pb_{p_0',p_1',\rho,n}$ and $\Qb_{p_0',n}$ with
    \begin{align*}
        T_n=\left( \lambda^{-2}(\log n)^{-3}\delta \right); \quad c_n=\Omega(1)
    \end{align*}
    and
    \begin{align*}
        \epsilon_n &= \exp\left( -\Omega(1)\cdot\frac{M^2}{p_1+p_0} \right) \\
        &\overset{\eqref{eq-def-M-planted-subgraph}}{=} \exp\left( -\Omega(1)\cdot\frac{\lambda^2\rho^2n^2 p_0}{p_1} \right) \overset{\eqref{eq-condition-planted-dense-subgraph-detail}}{=} \exp\left( -\Omega(1)\cdot\frac{n\delta p_0}{p_1} \right) \,,
    \end{align*}
    thereby finishing the proof.
\end{proof}

\subsection{Bounding the low-degree advantage}{\label{subsec:bounding-low-deg-adv-planted-dense-subgraph}}

\begin{lemma}{\label{lem-bound-adv-planted-dense-subgraph}}
    Assuming \eqref{eq-condition-planted-dense-subgraph-detail} and recall that $p_0'=\kappa p_0,p_1'=\kappa p_1$ where $\kappa>1$ is a constant close to $1$. We then have
    \begin{equation}{\label{eq-bound-adv-planted-dense-subgraph}}
        \mathsf{Adv}_{\leq D}(\Pb_{p_0',p_1',\rho,n};\Qb_{p_0',n})^2 \leq \exp\left( \Theta(1) \cdot \lambda^{-2} \delta^{2} \right) \mbox{ for all } D \leq \lambda^{-2}(\log n)^{-2}\delta \,.
    \end{equation}
\end{lemma}
\begin{proof}
    Define
    \begin{equation}{\label{eq-def-lambda'}}
        (\lambda')^2 = \frac{ (p_0'-p_1')^2 }{ p_0'(1-p_0') } \overset{\eqref{eq-condition-planted-dense-subgraph-detail}}{=} [1+o(1)] \kappa \lambda^2 \,.
    \end{equation}
    In addition, define $\mathbf P=\mathbf P_{\lambda',\rho,n}$ to be the law of planted submatrix $\bm Y=\lambda'\theta\theta^{\top}+\bm Z$ such that $\theta\in\{ 0,1 \}^n$ has i.i.d.\ $\mathsf{Ber}(\rho)$ entries, and define $\mathbf Q=\mathbf Q_n$ to be the law of a symmetric Gaussian matrix. We will show that
    \begin{equation}{\label{eq-goal-lem-4.4}}
        \mathsf{Adv}_{\leq D}\left( \Pb_{\lambda',\rho,n};\Qb_{n} \right)^2 \leq \mathsf{Adv}_{\leq D}\left( \mathbf P_{\lambda',\rho,n};\mathbf Q_{n} \right)^2 \,.
    \end{equation}
    \eqref{eq-bound-adv-planted-dense-subgraph} then follows directly from \eqref{eq-goal-lem-4.4} and Lemma~\ref{lem-bound-adv-planted-submatrix}. Now we focus on \eqref{eq-goal-lem-4.4}. To this end, for any $S \Subset \mathsf{K}_n$, define
    \begin{equation}{\label{eq-def-f-H}}
        f_{S}(\bm Y) = \prod_{(i,j)\in E(S)} \frac{ \bm Y_{i,j}-p_0' }{ \sqrt{p_0'(1-p_0')} } \,.
    \end{equation}
    It is well known in \cite{HS17} that $\{ f_S:S \Subset \mathsf{K}_n, |E(S)|\leq D \}$ constitutes a standard orthogonal basis of $\mathcal P_{D}$ under $\Qb_{p_0',n}$. Thus, using \cite{Hopkins18} we have 
    \begin{align}
        \mathsf{Adv}_{\leq D}\left( \Pb_{p_0',p_1',\rho,n};\Qb_{p_0',n} \right)^2 = \sum_{ S \Subset \mathsf{K}_n, |E(S)|\leq D } \mathbb E_{ \Pb_{p_0',p_1',\rho,n} } \left[ f_S(\bm Y) \right]^2 \,.  \label{eq-lem-4.4-LHS}
    \end{align}
    In addition, for any $S \Subset \mathsf{K}_n$, define
    \begin{equation}{\label{eq-def-g-H}}
        g_{S}(\bm Y) = \prod_{(i,j)\in E(S)} \bm Y_{i,j} \,.
    \end{equation}
    It is clear that $\{ f_S: S \Subset \mathsf{K}_n, |E(S)|\leq D \}$ are orthonormal under $\mathbf Q_n$. Thus, using Parseval's inequality we have
    \begin{align}
        \mathsf{Adv}_{\leq D}\left( \mathbf P_{\lambda',\rho,n};\mathbf Q_{n} \right)^2 \geq \sum_{ S \Subset \mathsf{K}_n, |E(S)|\leq D } \mathbb E_{ \mathbf P_{\lambda',\rho,n} } \left[ f_S(\bm Y) \right]^2 \,.  \label{eq-lem-4.4-RHS}
    \end{align}
    In addition, straightforward calculation yields that
    \begin{align*}
        \mathbb E_{ \Pb_{p_0',p_1',\rho,n} } \left[ f_S(\bm Y) \right] &= \mathbb E_{\theta}\left\{ \mathbb E_{ \Pb_{p_0',p_1',\rho,n}(\cdot\mid\theta) } \left[ \prod_{(i,j)\in E(S)} \frac{ \bm Y_{i,j}-p_0' }{ \sqrt{p_0'(1-p_0')} } \right] \right\} \\
        &= \mathbb E_{\theta}\left\{ \mathbb E_{ \Pb_{p_0',p_1',\rho,n}(\cdot\mid\theta) } \left[ \prod_{(i,j)\in E(S)} \frac{ (p_1'-p_0') \theta_i \theta_j }{ \sqrt{p_0'(1-p_0')} } \right] \right\} \\
        &= \mathbb E_{\theta}\left\{ \mathbb E_{ \Pb_{p_0',p_1',\rho,n}(\cdot\mid\theta) } \left[ (\lambda')^{|E(S)|} \prod_{i \in V(S)} \theta_i \right] \right\} = (\lambda')^{|E(S)|} \rho^{|V(S)|} \,;
    \end{align*}
    as well as
    \begin{align*}
        \mathbb E_{ \mathbf P_{\lambda',\rho,n} } \left[ g_S(\bm Y) \right] &= \mathbb E_{\theta}\left\{ \mathbb E_{ \mathbf P_{\lambda',\rho,n}(\cdot\mid\theta) } \left[ \prod_{(i,j)\in E(S)} \bm Y_{i,j} \right] \right\} \\
        &= \mathbb E_{\theta}\left\{ \mathbb E_{ \mathbf P_{\lambda',\rho,n}(\cdot\mid\theta) } \left[ \prod_{(i,j)\in E(S)} (\lambda' \theta_i \theta_j) \right] \right\} = (\lambda')^{|E(S)|} \rho^{|V(S)|}  \,.
    \end{align*}
    Thus, we have 
    \begin{align*}
        \mathsf{Adv}_{\leq D}\left( \Pb_{p_0',p_1',\rho,n};\Qb_{p_0',n} \right)^2 &\overset{\eqref{eq-lem-4.4-LHS}}{=} \sum_{ S \Subset \mathsf{K}_n, |E(S)|\leq D } \mathbb E_{ \Pb_{p_0',p_1',\rho,n} } \left[ f_S(\bm Y) \right]^2 \\
        &= \sum_{ S \Subset \mathsf{K}_n, |E(S)|\leq D } (\lambda')^{2|E(S)|} \rho^{2|V(S)|} \\
        &= \sum_{ S \Subset \mathsf{K}_n, |E(S)|\leq D } \mathbb E_{ \mathbf P_{\lambda',\rho,n} } \left[ f_S(\bm Y) \right]^2 \\
        &\overset{\eqref{eq-lem-4.4-RHS}}{\leq} \mathsf{Adv}_{\leq D}\left( \mathbf P_{\lambda',\rho,n} \| \mathbf Q_{n} \right)^2 \,,
    \end{align*}
    leading to \eqref{eq-goal-lem-4.4}.
\end{proof}

\subsection{Putting it together}{\label{subsec:conclusion-planted-dense-submatrix}}

We can now conclude Lemma~\ref{lem-simplified-goal-planted-dense-subgraph}, thereby finishing the proof of Theorem~\ref{Main-thm-planted-dense-subgraph}.
\begin{proof}[Proof of Lemma~\ref{lem-simplified-goal-planted-dense-subgraph}]
    Suppose on the contrary that there exists an algorithm $\mathcal A$ with running time $\exp(o(\lambda^{-2}(\log n)^{-2}\delta))$ that takes $\bm Y\sim \Pb_{p_0,p_1,\rho,n}$ as input and achieves weak recovery in the sense of Definition~\ref{def-recovery-Bernoulli-obser-model}. Using Lemma~\ref{lem-imbalanced-detection-planted-dense-subgraph}, we see that there exists an $(T_n;c_n;\epsilon_n)$-test between $\Pb_{p_0',p_1',\rho,n}$ and $\Qb_{p_0',n}$ such that
    \begin{align*}
        T_n = o\left(\lambda^{-2}(\log n)^{-3}\delta\right), \quad c_n=\Omega(1), \quad \epsilon_n = \exp\left( -\Omega(1)\cdot\frac{n\delta p_0}{p_1} \right) \,.
    \end{align*}
    However, using Lemma~\ref{lem-bound-adv-planted-dense-subgraph} and Proposition~\ref{thm-alg-contiguity}, we see that there is no $(T_n';c_n';\epsilon_n')$-test between $\Pb_{p_0',p_1',\rho,n}$ and $\Qb_{p_0',n}$ such that 
    \begin{equation*}
        T_n' = o\left(\lambda^{-2}(\log n)^{-3}\delta\right), \quad c_n'=\Omega(1), \quad \epsilon_n' = \exp\left( -\lambda^{-2}\delta \right) \,.
    \end{equation*}
    Note that using \eqref{eq-condition-planted-dense-subgraph-detail}, we have 
    \begin{align*}
        \lambda^{-2} =
        \begin{cases}
            \delta^{-1}\rho^2 n \ll \Theta(n) = \frac{np_0}{p_1} \,, &  p_1 \leq 2p_0 \,; \\
            O(p_0/p_1^2) \ll \frac{np_0}{p_1} \,, & p_1 \geq 2p_0 \,.
        \end{cases}
    \end{align*}
    This forms a contradiction and thus yields Lemma~\ref{lem-simplified-goal-planted-dense-subgraph}.
\end{proof}

\section{Stochastic block model}{\label{sec:proof-SBM}}

This section is devoted to the proof of Theorem~\ref{Main-thm-SBM}. Note that by keeping each edge independently with probability $p$ will reduce to an SBM with parameter $(dp,\lambda)$. Thus, without loss of generality, we will assume throughout this section that 
\begin{equation}{\label{eq-condition-SBM}}
    \Theta(1)=d\lambda^2 \leq 1-\delta \mbox{ for some constant } \delta>0 \mbox{ and } 2\leq q \ll n^{\frac{1}{8}-\Omega(1)} \,.
\end{equation}

\subsection{Reducing to an imbalanced detection problem}{\label{subsec:reduce-to-detection-SBM}}

The following lemma enables us to instead consider a detection problem with lopsided success probability. Fix a constant $\kappa>1$ close to $1$ such that
\begin{equation}{\label{eq-choice-kappa}}
    \kappa d\lambda^2<1-\Omega(1) \,.
\end{equation}
Define $d'=\kappa d$. In addition, let $\Qb=\Qb_{d',n}$ be the law of an \ER graph on $[n]$ with parameter $\frac{d'}{n}$ and let $\Pb=\Pb_{d',\lambda,q,n}$ be the law of an SBM on $[n]$ with parameters $(d',\epsilon,q)$. 
\begin{lemma}{\label{lem-imbalanced-detection-SBM}}
    Suppose that $2\leq q\ll n^{\frac{1}{8}}$ and suppose that there exists an algorithm with running time $\exp(o(nq^{-6}(\log n)^{-3}))$ that takes $\bm Y\sim\Pb_{d,\lambda,q,n}$ as input and achieves weak recovery in the sense of Definition~\ref{def-recovery-Bernoulli-obser-model}. Then there is an $(T_n;c_n;\epsilon_n)$-test between $\Pb_{d',\lambda,q,n}$ and $\Qb_{d',n}$ with
    \begin{align*}
        T_n = o\left( nq^{-6}(\log n)^{-3} \right); \quad c_n=\Omega(1); \quad \epsilon_n= \exp\left( -\Omega(n) \right) \,.
    \end{align*}
\end{lemma}
\begin{proof}
    Define
    \begin{equation}{\label{eq-def-mathcal-K-SBM}}
        \mathcal K=\left\{ \mathcal Y \in \mathbb R^{n*n}: \| \mathcal Y \|_{\infty} \leq \frac{(q-1)\lambda d}{n},\ \mathcal Y+\frac{\lambda d}{n} \mathbb I \succ \mathbb O \right\} \,.
    \end{equation}
    Also define $M \in \mathbb R$ such that
    \begin{equation}{\label{eq-def-M-SBM}}
        M^2 = \frac{n^2}{q} \cdot \left( \frac{((q-1)\lambda)d}{n} \right)^2 + \frac{(q-1)n^2}{q} \cdot \left( \frac{(1-\lambda)d}{n} \right)^2 = \left( (q-1)\lambda^2+1 \right) d^2 \,.
    \end{equation}
    Recall that under $\Pb_{d,\lambda,q,n}$, $\bm Y$ is a binary observation model with 
    \begin{align*}
        p=\frac{d}{n} \mbox{ and } \Theta=\frac{(q\mathbf 1_{\sigma_*(i)=\sigma_*(j)}-1)\lambda d}{n} \,.
    \end{align*}
    It is clear that
    \begin{align*}
        \Pb\left( \frac{M}{2} \leq \|\Theta\|_{\Fop} \leq 2M, \Theta\in\mathcal K \right)=1-o(1) \,.
    \end{align*}
    Suppose there exists an algorithm $\mathcal X(\bm Y)$ with running time $\exp(o(nq^{-2}(\log n)^{-2}))$ that takes $\bm Y\sim \Pb_{d,\lambda,q,n}$ as input and achieves weak recovery in the sense of Definition~\ref{def-recovery-Bernoulli-obser-model}. Using Lemma~\ref{lem-cor-pre-proj}, we see that there exists an algorithm $\widehat{\mathcal X}(\bm Y)$ with running time $\exp(nq^{-2}(\log n)^{-2})$ that takes $\bm Y\sim \Pb_{d,\lambda,q,n}$ as input and
    \begin{align*}
        \widehat{\mathcal X} \in \mathcal K, \quad \|\widehat{\mathcal X}\|_{\Fop} \geq \Omega(M), \quad \Pb_{d,\lambda,q,n}\left( \frac{ \langle \Theta,\mathcal X \rangle }{ \|\Theta\|_{\Fop} \| \mathcal X \|_{\Fop} } \geq \Omega(1) \right) \geq \Omega(1) \,.
    \end{align*}
    In particular, we have $\|\widehat{\mathcal X}\|_{\Fop} \geq \Theta(M)$ and $\|\widehat{\mathcal X}\|_{\infty}\leq \frac{(q-1)\lambda d}{n}$. Applying Lemma~\ref{lem-reduce-to-detection-Ber-obser-model} with $a=b=\kappa^{-1}$, we see that there is an $(T_n;c_n;\epsilon_n)$-test between $\Pb_{d',\lambda,q,n}$ and $\Qb_{d',n}$ with
    \begin{align*}
        T_n &=o\left( nq^{-6}(\log n)^{-3} \right) \,; \quad c_n=\Omega(1) 
    \end{align*}
    and 
    \begin{align*}
        \epsilon_n &=\exp\left( -\Omega(1)\cdot\frac{M^2}{ \frac{d}{n} + \frac{(1+(q-1)\lambda)d}{n} } \right) \\
        &\overset{\eqref{eq-def-M-SBM}}{=} \exp\left( -\Omega(1)\cdot\frac{nd(\lambda^2(q-1))}{1+(q-1)\lambda} \right) \overset{\eqref{eq-condition-SBM}}{=} \exp\left( -\Omega(n) \right) \,,
    \end{align*}
    thereby finishing the proof. 
\end{proof}

\subsection{Bounding the low-degree advantage}{\label{subsec:bounding-low-deg-adv-SBM}}

We will prove the following theorem on the low-degree lower bound for the stochastic block model, which strengthened the results in \cite{Hopkins18, BBK+21}. 
\begin{lemma}{\label{lem-bound-adv-SBM}}
    For all $d\lambda^2 \leq 1-\delta$ where $\delta=\Omega(1)$ and $2\leq q\ll n^{\frac{1}{8}}$, we have
    \begin{equation}{\label{eq-bound-adv-SBM}}
        \mathsf{Adv}_{\leq D}\left( \Pb_{d,\lambda,q,n};\Qb_{d,n} \right)^2 \leq \exp\left( O(q^2)+O\left( \frac{D}{\log n} \right) \right) \mbox{ for all } D=o\left( \frac{n}{q^6(\log n)^3} \right) \,.
    \end{equation}
\end{lemma}
The rest of this subsection is devoted to the proof of Lemma~\ref{lem-bound-adv-SBM}. To this end, for any $H \Subset \mathsf K_n$, define
\begin{equation}{\label{eq-def-orthogonal-basis-SBM}}
    f_{H}(\bm Y):= \prod_{(i,j)\in E(H)} \frac{ \bm Y(i,j)-\frac{d}{n} }{ \sqrt{ \frac{d}{n} (1-\frac{d}{n}) } } \,.
\end{equation}
It is straightforward to verify that $\{ f_{H}: |E(H)| \}$ constitutes a standard orthogonal basis of $\mathcal P_D$ under $\Qb$, in the sense that
\begin{equation}{\label{eq-standard-orthogonal-SBM}}
    \mathbb E_{\Qb}\left[ f_{H}(\bm Y) f_{H'}(\bm Y) \right] = \mathbf 1_{ \{ H=H' \} } \,.
\end{equation}
Thus, it is standard that (see, e.g., \cite{Hopkins18})
\begin{equation}{\label{eq-Adv-SBM-transform-I}}
    \mathsf{Adv}_{\leq D}(\Pb;\Qb)^2 = \sum_{ |E(H)|\leq D } \mathbb E_{\Pb}\Big[ f_{H}(\bm Y) \Big]^2 \,.
\end{equation}
We now calculate $\mathbb E_{\Pb}[ f_{H}(\bm Y) ]$. For all $\sigma,\tau\in [q]$, define
\begin{equation}{\label{eq-def-omega(sigma,tau)}}
    \omega(\sigma,\tau)= q \mathbf 1_{\{\sigma=\tau\}} -1 \,.
\end{equation}
By conditioning on $\{ \sigma(i): 1 \leq i \leq n \}$, we have
\begin{align*}
    \mathbb E_{\Pb}\Big[ f_{H}(\bm Y) \Big] =\ & \mathbb E_{\sigma}\Big\{ \mathbb E_{\Pb}\big[ f_{H}(\bm Y) \mid \{ \sigma(i) \} \big] \Big\} \\
    =\ & \mathbb E\left[ \prod_{(i,j)\in E(H)} \frac{d}{n} \cdot \frac{ \lambda \omega(\sigma(i),\sigma(j)) }{ \sqrt{ \frac{d}{n} (1-\frac{d}{n}) } } \right] \\
    =\ & [1+o(1)] \left( (\lambda^2 d)^{ \frac{1}{2}|E(H)| } n^{-\frac{1}{2}|E(H)|} \right) \cdot \mathbb E\left[ \prod_{(i,j)\in E(H)} \omega(\sigma(i),\sigma(j)) \right] \,,
\end{align*}
Thus, we have
\begin{align}
    \eqref{eq-Adv-SBM-transform-I} = \sum_{ |E(H)| \leq D } \left( \frac{1-\delta}{n} \right)^{|E(H)|} \cdot \mathbb E\left[ \prod_{(i,j)\in E(H)} \omega(\sigma(i),\sigma(j)) \right]^2  \,.   \label{eq-Adv-SBM-transform-II}
\end{align}
It remains to bound the right hand side of \eqref{eq-Adv-SBM-transform-II}. To this end, we first show the following lemma.
\begin{lemma}{\label{lem-single-graph-signal}}
    Recall that we use $\tau(H)=|E(H)|-|V(H)|$ to denote the excess of a graph. We have
    \begin{align*}
        \left| \mathbb E\left[ \prod_{(i,j)\in E(H)} \omega(\sigma(i),\sigma(j)) \right] \right| \leq \mathbf 1_{ \{ \mathsf{L}(H)=\emptyset \} } \cdot q^{3\tau(H)+\mathfrak C(H)} \,.
    \end{align*}
\end{lemma}
\begin{proof}
    Using Lemma~\ref{lem-leaf-cancellation}, we have 
    \begin{align*}
        \mathbb E\left[ \prod_{(i,j)\in E(H)} \omega(\sigma(i),\sigma(j)) \right] =0 \mbox{ if } \mathsf L(H) \neq \emptyset \,.
    \end{align*}
    Thus, it suffices to focus on the case $\mathsf L(H)=\emptyset$. In this case, from Lemma~\ref{lem-decomposition-H-Subset-S} with $\emptyset \subset H$, we can decompose $H$ into  
    \begin{align*}
        H = P_{\mathtt 1} \sqcup \ldots \sqcup P_{\mathtt t} \sqcup C_{\mathtt 1} \sqcup \ldots \sqcup C_{\mathtt m}
    \end{align*}
    such that Items~(1), (2) in Lemma~\ref{cor-revised-decomposition-H-Subset-S} are satisfied. In particular, we have $\mathtt t \leq 3\tau(H)$ and $\mathtt m=\mathfrak C(H)$. Thus, by conditioning on $\mathfrak S=\{ \sigma(u): u \in \cup_{\mathtt i} \mathsf{EndP}(P_{\mathtt i}) \}$ we have
    \begin{align*}
        &\mathbb E\left[ \prod_{(i,j)\in E(H)} \omega(\sigma(i),\sigma(j)) \right] \\
        =\ & \mathbb E\left\{ \prod_{\mathtt 1 \leq \mathtt i \leq \mathtt t} \mathbb E\left[ \prod_{(i,j)\in E(P_{\mathtt i})} \omega(\sigma(i),\sigma(j)) \mid \mathfrak S \right] \prod_{\mathtt 1 \leq \mathtt j \leq \mathtt m} \mathbb E\left[ \prod_{(i,j)\in E(C_{\mathtt j})} \omega(\sigma(i),\sigma(j)) \mid \mathfrak S \right] \right\} \\
        =\ & \mathbb E\left\{ \prod_{\mathtt 1 \leq \mathtt i \leq \mathtt t} \omega(\mathsf{EndP}(P_{\mathtt i})) \prod_{\mathtt 1 \leq \mathtt j \leq \mathtt m} (q-1) \right\} \leq q^{\mathtt t+\mathtt m} \,,
    \end{align*}
    leading to the desired result.
\end{proof}
Now we can finish the proof of Lemma~\ref{lem-bound-adv-SBM}.
\begin{proof}[Proof of Lemma~\ref{lem-bound-adv-SBM}]
    Using Lemma~\ref{lem-single-graph-signal}, we have that
    \begin{align}
        \eqref{eq-Adv-SBM-transform-II} = \sum_{ \substack{ |E(H)| \leq D \\ \mathsf L(H)=\emptyset } } \left( \frac{1-\delta}{n} \right)^{|E(H)|} \cdot q^{6\tau(H)+2\mathfrak C(H)} \,.  \label{eq-Adv-SBM-transform-III}
    \end{align}
    For every $\mathsf L(H)=\emptyset$, denote $H'$ be the maximum subgraph of $H$ such that $\mathfrak C(H')=\emptyset$. Also denote $p_m=\mathfrak C_m(H)$. It is clear that
    \begin{align}
        \eqref{eq-Adv-SBM-transform-III} &\leq \sum_{ \substack{ |E(H')| \leq D \\ \mathsf L(H'),\mathfrak C(H')=\emptyset \\ 3p_3+\ldots+Dp_D\leq D } } \left( \frac{1-\delta}{n} \right)^{|E(H')|+3p_3+\ldots+Dp_D} q^{6\tau(H')+2(p_3+\ldots+p_D)} \cdot \frac{n^{3p_3+\ldots+Dp_D}}{p_3!\ldots p_m!} \nonumber \\
        &\leq \sum_{ \substack{ |E(H')| \leq D \\ \mathsf L(H'),\mathfrak C(H')=\emptyset } } \left( \frac{1-\delta}{n} \right)^{|E(H')|} q^{6\tau(H')} \prod_{3 \leq m \leq D} \sum_{p_m\geq 0} (1-\delta)^{mp_m} q^{2p_m}  \nonumber \\
        &= \sum_{ \substack{ |E(H')| \leq D \\ \mathsf L(H'),\mathfrak C(H')=\emptyset } } \left( \frac{1-\delta}{n} \right)^{|E(H')|} q^{6\tau(H')} \prod_{3 \leq m \leq D} \exp\left( (1-\delta)^m q^2 \right) \nonumber \\
        &= \exp(O(q^2)) \cdot \sum_{ \substack{ |E(H')| \leq D \\ \mathsf L(H'),\mathfrak C(H')=\emptyset } } \left( \frac{1-\delta}{n} \right)^{|E(H')|} q^{6\tau(H')} \,.  \label{eq-Adv-SBM-transform-IV}
    \end{align}
    Let $v=|V(H')| \leq D$ and $t=\tau(H)\leq D$, we see that
    \begin{align*}
        &\sum_{ \substack{ |E(H')| \leq D \\ \mathsf L(H'),\mathfrak C(H')=\emptyset } } \left( \frac{1-\delta}{n} \right)^{|E(H')|} q^{6\tau(H')} \\
        \leq\ & \sum_{ 1 \leq v,t \leq D } (1-\delta)^{v+t} n^{-(v+t)} q^{6t} \#\{ H' \subset \mathsf K_n: |V(H')|=v,\tau(H')=t \} \\
        \leq\ & \sum_{ 1 \leq v,t \leq D } (1-\delta)^{v+t} n^{-(v+t)} q^{6t} \cdot \frac{ D^{2t} n^v }{ t! } \frac{1}{p_3!\ldots p_D!}   \\
        \leq\ & \sum_{ \substack{ 1 \leq v,t \leq D \\ 3p_3+\ldots+Dp_D\leq v } } \frac{ (n^{-1}q^6v^2)^t }{ t! } \cdot \frac{(1-\delta)^{3p_3+\ldots+Dp_D}}{p_3!\ldots p_D!} \\
        \leq\ & \exp(O(1)) \cdot \sum_{ \substack{ 1 \leq v,t \leq D } } \frac{ (n^{-1}q^6v^2)^t }{ t! } \,,
    \end{align*}
    where the second inequality follows from Lemma~\ref{lem-enu-large-graph}. Note that when $t \leq \frac{v}{(\log n)^2}$, we have (note that $v \leq D$)
    \begin{align*}
        \frac{ (n^{-1}q^6v^2)^t }{ t! } \leq n^{10t} = \exp(O(D/\log n)) \,.
    \end{align*}
    And when $t \geq \frac{v}{(\log n)^2}$, we have (note that $v \leq D \leq n/q^6(\log n)^3$)
    \begin{align*}
        \frac{ (n^{-1}q^6v^2)^t }{ t! } \leq \exp\left( \frac{eq^6v^2}{nt} \right) \leq \exp\left( \frac{evq^6(\log n)^2}{n} \right) = o(1) \,.
    \end{align*}
    Plugging the above estimation into \eqref{eq-Adv-SBM-transform-IV}, we have that
    \begin{align*}
        \eqref{eq-Adv-SBM-transform-IV} \leq \exp(O(q^2)) \cdot \exp(O(D/\log n)) \,,
    \end{align*}
    as desired.
\end{proof}

\subsection{Putting it together}{\label{subsec:conclusion-SBM}}

We can now finish the proof of Theorem~\ref{Main-thm-SBM}.
\begin{proof}[Proof of Theorem~\ref{Main-thm-SBM}]
    Suppose on the contrary that there exists an algorithm $\mathcal A$ with running time $\exp(o(nq^{-6}(\log n)^{-3}))$ that takes $\bm Y\sim \Pb_{d,\lambda,q,n}$ as input and achieves weak recovery in the sense of Definition~\ref{def-recovery-Bernoulli-obser-model}. Using Lemma~\ref{lem-imbalanced-detection-SBM}, we see that there exists an $(T_n;c_n;\epsilon_n)$-test between $\Pb_{d',\lambda,q,n}$ and $\Qb_{d,n}$ such that
    \begin{align*}
        T_n = o\left( nq^{-6}(\log n)^{-3} \right), \quad c_n=\Omega(1), \quad \epsilon_n = \exp\left( -\Omega(n) \right) \,.
    \end{align*}
    However, using Lemma~\ref{lem-bound-adv-SBM} and Proposition~\ref{thm-alg-contiguity}, we see that there is no $(T_n';c_n';\epsilon_n')$-test between $\Pb_{d',\lambda,q,n}$ and $\Qb_{d',n}$ such that 
    \begin{equation*}
        T_n' = o\left(nq^{-6}(\log n)^{-3}\right), \quad c_n'=\Omega(1), \quad \epsilon_n' = \exp\left( -\omega(q^2+nq^{-6}(\log n)^{-3}) \right) \,.
    \end{equation*}
    This forms a contradiction and thus yields Theorem~\ref{Main-thm-SBM} (recall that $q^2\ll n$).
\end{proof}

\section{Multi-frequency angular synchronization}{\label{sec:proof-angular-synchronization}}

This section is devoted to the proof of Theorem~\ref{Main-thm-angular-synchronization}.

\subsection{Reducing to an imbalanced detection problem}{\label{subsec:reduce-to-detection-angular-synchronization}}

The following lemma enables us to instead consider a detection problem with lopsided success probability. Recall that $\lambda=1-\delta$ for some constant $\delta>0$. Fix a constant $\kappa>0$ such that
\begin{equation}{\label{eq-choice-kappa-angular-synchronization}}
    \lambda'=(1+\kappa^2)\lambda<1-\Omega(1) \,.
\end{equation}
In addition, let $\Qb=\Qb_{L,n}$ be the law of $L$ independent $n*n$ GUE matrices, and let $\Pb=\Pb_{\lambda',L,n}$ be the law of multi-frequency angular synchronization model with parameter $\lambda'$ defined in Definition~\ref{def-multi-frequency-group-synchronization}.
\begin{lemma}{\label{lem-imbalanced-detection-angular-synchronization}}
    Suppose that there exists an algorithm with running time $\exp(n^{o(1)})$ that takes $\bm Y\sim\Pb_{\lambda,L,n}$ as input and achieves weak recovery of $\bm Y_\ell$ in the sense of Definition~\ref{def-recovery-add-Gaussian-model}. Then there is an $(T_n;c_n;\epsilon_n)$-test between $\Pb_{\lambda',L,n}$ and $\Qb_{L,n}$ with
    \begin{align*}
        T_n=n^{o(1)}; \quad c_n=\Omega(1); \quad \epsilon_n= \exp\left( -\Omega(n) \right) \,.
    \end{align*}
\end{lemma}
\begin{proof}
    Recall that for $\bm Y\sim\Pb_{\lambda,L,n}$, $\bm Y_\ell$ is an additive Gaussian model with 
    \begin{align*}
        \Theta=\frac{\lambda}{\sqrt{n}} \left( \bm x^{(\ell)} \right) \left( \bm x^{(\ell)} \right)^{*} \,,
    \end{align*}
    and thus $\|\Theta\|_{\Fop}^2 = \lambda^2 n$. Thus, from Lemma~\ref{lem-reduce-to-detection-add-Gaussian-model} we see that there is a $(T_n;c_n;\epsilon_n)$-test between $\Pb_{\lambda',L,n}$ and $\Qb_{L,n}$ with
    \begin{equation*}
        T_n = n^{o(1)}; \quad c_n=\Omega(1); \quad \epsilon_n= \exp\left( -\Omega(n) \right) \,.  \qedhere
    \end{equation*}
\end{proof}

\subsection{Bounding the low-degree advantage}{\label{subsec:bounding-low-deg-adv-angular-synchronization}}

\begin{lemma}{\label{lem-bound-adv-angular-synchronization}}
    Suppose that $\lambda<1-\Omega(1)$. We have $\mathsf{Adv}_{\leq D}(\Pb_{\lambda,L,n};\Qb_{L,n})^2 \leq \exp(O(L))$ for any $D,L=n^{o(1)}$.
\end{lemma}
The rest of this subsection is devoted to the proof of Lemma~\ref{lem-bound-adv-angular-synchronization}.
To this end, we first use standard tools for general Gaussian additive models \cite{KWB22} to derive the following bound on low-degree advantage.
\begin{lemma}{\label{lem-adv-angular-relax-1}}
    Recall \eqref{eq-exp-leq-D}. We then have 
    \begin{align}{\label{eq-bound-Adv-relax-1}}
        \mathsf{Adv}_{\leq D}(\Pb;\Qb)^2 \leq \mathbb E_{\theta_1, \ldots,\theta_n \sim \operatorname{Unif}[0,2\pi]}\left\{ \prod_{1 \leq \ell \leq L} \exp_{\leq D}\left( \lambda^2 U_{\ell}^2 \right) \prod_{1 \leq \ell \leq L} \exp_{\leq D}\left( \lambda^2 V_{\ell}^2 \right) \right\} \,.
    \end{align}
    Here 
    \begin{align}{\label{eq-def-U-ell-V-ell}}
        U_{\ell} = \frac{1}{\sqrt{n}} \sum_{t=1}^{n} \sin(\ell\theta_t), \quad V_{\ell} = \frac{1}{\sqrt{n}} \sum_{t=1}^{n} \cos(\ell\theta_t) \,.
    \end{align}
\end{lemma}
\begin{proof}
    Using \cite[Lemma~4.3]{KBK24+}, we have that (below we use $\mathbbm 1_n$ to denote the $n$-dimensional all-one vector)
    \begin{equation*}
        \mathsf{Adv}_{\leq D}(\Pb;\Qb)^2 = \mathbb E_{\bm x}\left\{ \exp_{\leq D}\left( \frac{\lambda^2}{n} \sum_{1\leq\ell\leq L} | \langle \bm x^{(\ell)}, \mathbbm 1_n \rangle |^2 \right) \right\} \,.
    \end{equation*}
    Note that we have $\bm x=(e^{i\theta_1},\ldots,e^{i\theta_n})$ where $\theta_1,\ldots,\theta_n \sim \operatorname{Unif}[0,2\pi]$, so
    \begin{align*}
        \frac{1}{n} | \langle \bm x^{(\ell)}, \mathbbm 1_n \rangle |^2 = U_{\ell}^2 + V_{\ell}^2 \,,
    \end{align*}
    we then have
    \begin{align*}
        \mathsf{Adv}_{\leq D}(\Pb;\Qb)^2 &= \mathbb E_{\theta_1,\ldots,\theta_n \sim \operatorname{Unif}[0,2\pi]}\left\{ \exp_{\leq D}\left( \lambda^2 \sum_{1\leq\ell\leq L} (U_{\ell}^2+V_{\ell}^2) \right) \right\} \\
        &\leq \mathbb E_{\theta_1,\ldots,\theta_n \sim \operatorname{Unif}[0,2\pi]}\left\{ \prod_{1 \leq \ell \leq L} \exp_{\leq D}\left( \lambda^2 U_{\ell}^2 \right) \prod_{1 \leq \ell \leq L} \exp_{\leq D}\left( \lambda^2 V_{\ell}^2 \right) \right\} \,,
    \end{align*}
    where the inequality follows from $\exp_{\leq D}(x+y) \leq \exp_{\leq D}(x) \exp_{\leq D}(y)$ for $x,y \geq 0$.
\end{proof}

Provided with Lemma~\ref{lem-adv-angular-relax-1}, it suffices to bound the right hand side of \eqref{eq-bound-Adv-relax-1} when $\lambda<1-\Omega(1)$. Note that
\begin{align}
    \eqref{eq-bound-Adv-relax-1} &= \mathbb E\left\{ \prod_{1 \leq \ell \leq L} \left( \sum_{ 0 \leq k_{\ell} \leq D } \frac{ \lambda^{2k_{\ell}} U_\ell^{2k_{\ell}} }{ k_{\ell}! } \right) \prod_{1 \leq \ell \leq L} \left( \sum_{ 0 \leq m_{\ell} \leq D } \frac{ \lambda^{2m_{\ell}} V_\ell^{2m_{\ell}} }{ m_{\ell}! } \right) \right\} \nonumber \\
    &= \sum_{ \substack{ 0 \leq k_1,\ldots,k_L \leq D \\ 0 \leq m_1,\ldots,m_L \leq D } } \frac{ \lambda^{ 2(k_1+\ldots+k_L+m_1+\ldots+m_L) } }{ k_1!\ldots k_L! m_1!\ldots m_L! } \mathbb E\left\{ U_1^{2k_1} \ldots U_L^{2k_L} V_1^{2m_1} \ldots V_L^{2m_L} \right\} \,.  \label{eq-bound-Adv-relax-2}
\end{align}
The basic intuition behind our approach of bounding \eqref{eq-bound-Adv-relax-2} is that according to \eqref{eq-def-U-ell-V-ell} and central limit theorem, we should expect that 
\begin{align}{\label{eq-Gaussian-approx-intuition}}
    \left( U_1,\ldots,U_L,V_1,\ldots,V_L \right) \mbox{ behaves like } \left( \zeta_1,\ldots,\zeta_L,\eta_1,\ldots,\eta_L \right) \sim \mathcal N(0,\frac{1}{2}\mathbb I_{2L}) \,.
\end{align}
The key to our proof is to show that such Gaussian approximation is indeed valid for all low-order moments via a delicate Lindeberg’s interpolation argument. To this end, we sample 
\begin{align*}
    \left\{ \zeta_\ell(t), \eta_\ell(t): 1 \leq t \leq n, 1 \leq \ell \leq L \right\}
\end{align*}
i.i.d.\ from $\mathcal N(0,\frac{1}{2})$. 
Define
\begin{equation}{\label{eq-def-U-ell(t)-V-ell(t)}}
    \begin{aligned}
        &U_\ell(t)= \frac{1}{\sqrt{n}}\left( \sum_{ 1 \leq j \leq t } \zeta_\ell(j) + \sum_{t+1 \leq j \leq n} \sin(\ell\theta_j) \right) \,; \\ 
        &V_\ell(t)= \frac{1}{\sqrt{n}}\left( \sum_{ 1 \leq j \leq t } \eta_\ell(j) + \sum_{t+1 \leq j \leq n} \cos(\ell\theta_j) \right) \,.
    \end{aligned}
\end{equation}
And define
\begin{align}{\label{eq-def-F-t}}
    F_t = \sum_{ \substack{ 0 \leq k_1,\ldots,k_L \leq D \\ 0 \leq m_1,\ldots,m_L \leq D } } \frac{ \lambda^{ 2(k_1+\ldots+k_L+m_1+\ldots+m_L) } }{ k_1!\ldots k_L! m_1!\ldots m_L! } \mathbb E\left\{ U_1(t)^{2k_1} \ldots U_L(t)^{2k_L} V_1(t)^{2m_1} \ldots V_L(t)^{2m_L} \right\} \,.
\end{align}
In particular, we have (note that $\{ U_{\ell}(n),V_{\ell}(n):1 \leq \ell \leq L \} \overset{d}{=} \{ \zeta_\ell,\eta_\ell:1 \leq \ell \leq L \}$)
\begin{align*}
    F_0 &= \sum_{ \substack{ 0 \leq k_1,\ldots,k_L \leq D \\ 0 \leq m_1,\ldots,m_L \leq D } } \frac{ \lambda^{ 2(k_1+\ldots+k_L+m_1+\ldots+m_L) } }{ k_1!\ldots k_L! m_1!\ldots m_L! } \mathbb E\left\{ U_1^{2k_1} \ldots U_L^{2k_L} V_1^{2m_1} \ldots V_L^{2m_L} \right\} = \eqref{eq-bound-Adv-relax-2} \,; \\
    F_n &= \sum_{ \substack{ 0 \leq k_1,\ldots,k_L \leq D \\ 0 \leq m_1,\ldots,m_L \leq D } } \frac{ \lambda^{ 2(k_1+\ldots+k_L+m_1+\ldots+m_L) } }{ k_1!\ldots k_L! m_1!\ldots m_L! } \mathbb E\left\{ \zeta_1^{2k_1} \ldots \zeta_L^{2k_L} \eta_1^{2m_1} \ldots \eta_L^{2m_L} \right\} \,.
\end{align*}
We now show that $F_t$ and $F_{t+1}$ are ``close'' in a certain sense, as incorporated in the next lemma.
\begin{lemma}{\label{lem-interpolate-F-t-F-t+1}}
    Suppose that $\lambda\leq 1$ and $D,L=n^{o(1)}$, we then have $F_{t+1}=[1+O(n^{-\frac{3}{2}+o(1)})]F_t$ for all $0 \leq t \leq n-1$.
\end{lemma}
\begin{proof}
    For $0 \leq t \leq n-1$, define 
    \begin{equation}{\label{eq-def-U-ell,*(t)-V-ell,*(t)}}
        \begin{aligned}
            &U_{\ell,*}(t)= \frac{1}{\sqrt{n}}\left( \sum_{ 1 \leq j \leq t } \zeta_\ell(j) + \sum_{t+2 \leq j \leq n} \sin(\ell\theta_j) \right) \,; \\ 
            &V_{\ell,*}(t)= \frac{1}{\sqrt{n}}\left( \sum_{ 1 \leq j \leq t } \eta_\ell(j) + \sum_{t+2 \leq j \leq n} \cos(\ell\theta_j) \right) \,.
        \end{aligned}
    \end{equation}
    It is clear that
    \begin{align*}
        & U_{\ell}(t)=U_{\ell,*}(t)+ \frac{1}{\sqrt{n}} \sin(\ell\theta_{t+1}), \quad V_{\ell}(t)=V_{\ell,*}(t)+ \frac{1}{\sqrt{n}}\cos(\ell\theta_{t+1}) \,; \\
        & U_{\ell}(t+1)=U_{\ell,*}(t)+\frac{1}{\sqrt{n}}\zeta_{\ell}(t+1), \quad V_{\ell}(t+1)=V_{\ell,*}(t)+ \frac{1}{\sqrt{n}} \eta_{\ell}(t+1) \,.
    \end{align*}
    Thus, we have that 
    \begin{align}
        F_t &= \sum_{ \substack{ 0 \leq k_1,\ldots,k_L \leq D \\ 0 \leq m_1,\ldots,m_L \leq D } } \mathbb E\left\{ \prod_{1 \leq \ell \leq L} \frac{ \lambda^{2k_\ell} U_\ell(t)^{2k_\ell} }{ k_\ell! } \cdot \frac{ \lambda^{2m_\ell} V_\ell(t)^{2m_\ell} }{ m_\ell! } \right\} \nonumber \\
        &= \sum_{ \substack{ 0 \leq k_1,\ldots,k_L \leq D \\ 0 \leq m_1,\ldots,m_L \leq D } } \mathbb E\left\{ \prod_{1 \leq \ell \leq L} \frac{ \lambda^{2k_\ell} (U_{\ell,*}(t)+ \frac{1}{\sqrt{n}} \sin(\ell\theta_{t+1}))^{2k_\ell} }{ k_\ell! } \cdot \frac{ \lambda^{2m_\ell} (V_{\ell,*}(t)+ \frac{1}{\sqrt{n}} \cos(\ell\theta_{t+1}))^{2m_\ell} }{ m_\ell! } \right\} \nonumber \\
        &= \sum_{ \substack{ 0 \leq \alpha_1,\ldots,\alpha_L \leq 2D \\ 0 \leq \beta_1,\ldots,\beta_L \leq 2D } } \mathbb E\left\{ \prod_{1 \leq \ell \leq L} \left( \frac{\sin(\ell\theta_{t+1})}{\sqrt{n}} \right)^{\alpha_\ell} \left( \frac{\cos(\ell\theta_{t+1})}{\sqrt{n}} \right)^{\beta_\ell} \right\} \Lambda(\alpha_1,\ldots,\alpha_L;\beta_1,\ldots,\beta_L) \,,  \label{eq-F-t-expand}
    \end{align}
    where $\Lambda(\alpha_1,\ldots,\alpha_L;\beta_1,\ldots,\beta_L)$ is defined to be
    \begin{align}{\label{eq-def-Lambda(alpha-ell-beta-ell)}}
        \sum_{ \substack{ \frac{1}{2}\alpha_\ell \leq k_\ell \leq D \\ \frac{1}{2}\beta_\ell \leq m_\ell \leq D } } \mathbb E\left\{ \prod_{1 \leq \ell \leq L} \frac{ \binom{2k_\ell}{\alpha_\ell} \lambda^{2k_\ell} U_{\ell,*}(t)^{2k_\ell-\alpha_\ell} }{ k_\ell! } \prod_{1 \leq \ell \leq L} \frac{ \binom{2m_\ell}{\beta_\ell} \lambda^{2m_\ell} V_{\ell,*}(t)^{2m_\ell-\beta_\ell} }{ m_\ell! } \right\} \,.
    \end{align}
    Here in the derivation of \eqref{eq-F-t-expand} we also use the independence between 
    \begin{align*}
        \Big\{ \sin(\ell\theta_{t+1}), \cos(\ell\theta_{t+1}): 1 \leq \ell \leq L \Big\} \mbox{ and } \Big\{ U_{\ell,*}(t), V_{\ell,*}(t): 1 \leq \ell \leq L \Big\} \,.
    \end{align*}
    Similarly, we have
    \begin{align}
        F_{t+1} = \sum_{ \substack{ 0 \leq \alpha_1,\ldots,\alpha_L \leq 2D \\ 0 \leq \beta_1,\ldots,\beta_L \leq 2D } } \mathbb E\left\{ \prod_{1 \leq \ell \leq L} \left( \frac{\zeta_\ell(t+1)}{\sqrt{n}} \right)^{\alpha_\ell} \left( \frac{\eta_\ell(t+1)}{\sqrt{n}} \right)^{\beta_\ell} \right\} \Lambda(\alpha_1,\ldots,\alpha_L;\beta_1,\ldots,\beta_L) \,.  \label{eq-F-t+1-expand}
    \end{align}
    We first argue that
    \begin{align}{\label{eq-bound-growth-Lambda(alpha-ell-beta-ell)}}
        \Lambda(\alpha_1,\ldots,\alpha_L;\beta_1,\ldots,\beta_L) \leq (16D^2)^{\alpha_1+\ldots+\alpha_L+\beta_1+\ldots+\beta_L} \Lambda(0,\ldots,0;0,\ldots,0) \,.
    \end{align}
    To this end, note that (below we use $\binom{b}{a}\leq b^a$) 
    \begin{align*}
        & \binom{2k_\ell}{\alpha_\ell} U_{\ell,*}(t)^{2k_\ell-\alpha_\ell} \leq \frac{1}{2} \left( (2k_\ell)^{2\lfloor\alpha_\ell/2\rfloor} U_{\ell,*}(t)^{2k_\ell-2\lfloor\alpha_\ell/2\rfloor} + (2k_\ell)^{2\lceil\alpha_\ell/2\rceil} U_{\ell,*}(t)^{2k_\ell-2\lceil\alpha_\ell/2\rceil} \right) \,; \\
        & \binom{2m_\ell}{\beta_\ell} V_{\ell,*}(t)^{2m_\ell-\beta_\ell} \leq \frac{1}{2} \left( (2m_\ell)^{2\lfloor\beta_\ell/2\rfloor} V_{\ell,*}(t)^{2m_\ell-2\lfloor\beta_\ell/2\rfloor} + (2m_\ell)^{2\lceil\beta_\ell/2\rceil} V_{\ell,*}(t)^{2m_\ell-2\lceil\beta_\ell/2\rceil} \right) \,.
    \end{align*}
    We have that 
    \begin{align*}
        &\Lambda(\alpha_1,\ldots,\alpha_L;\beta_1,\ldots,\beta_L) \\
        \leq\ & \frac{1}{2^{2L}} \sum_{ \substack{ \alpha_\ell' \in \{ \lfloor\alpha_\ell/2\rfloor, \lceil\alpha_\ell/2\rceil \} \\ \beta_\ell' \in \{ \lfloor\beta_\ell/2\rfloor, \lceil\beta_\ell/2\rceil \} \\ \frac{1}{2}\alpha_\ell \leq k_\ell \leq D \\ \frac{1}{2}\beta_\ell \leq m_\ell \leq D } } \mathbb E\left\{ \prod_{1 \leq \ell \leq L} \frac{ \lambda^{2k_\ell} (2k_\ell)^{2\alpha_\ell'} U_{\ell,*}(t)^{2k_\ell-2\alpha_\ell'} }{ k_\ell! } \prod_{1 \leq \ell \leq L} \frac{ \lambda^{2m_\ell} (2m_\ell)^{2\beta_\ell'} V_{\ell,*}(t)^{2m_\ell-2\beta_\ell'} }{ m_\ell! } \right\} \\ 
        \leq\ & \frac{1}{2^{2L}} \sum_{ \substack{ \alpha_\ell' \in \{ \lfloor\alpha_\ell/2\rfloor, \lceil\alpha_\ell/2\rceil \} \\ \beta_\ell' \in \{ \lfloor\beta_\ell/2\rfloor, \lceil\beta_\ell/2\rceil \} \\ \alpha_\ell' \leq k_\ell \leq D \\ \beta_\ell' \leq m_\ell \leq D } }  \mathbb E\left\{ \prod_{1 \leq \ell \leq L} \frac{ \lambda^{2k_\ell} (2k_\ell)^{2\alpha_\ell'} U_{\ell,*}(t)^{2k_\ell-2\alpha_\ell'} }{ k_\ell! } \prod_{1 \leq \ell \leq L} \frac{ \lambda^{2m_\ell} (2m_\ell)^{2\beta_\ell'} V_{\ell,*}(t)^{2m_\ell-2\beta_\ell'} }{ m_\ell! } \right\} \\
        :=\ & \frac{1}{2^{2L}} \sum_{ \substack{ \alpha_\ell' \in \{ \lfloor\alpha_\ell/2\rfloor, \lceil\alpha_\ell/2\rceil \} \\ \beta_\ell' \in \{ \lfloor\beta_\ell/2\rfloor, \lceil\beta_\ell/2\rceil \} } } 2^{ 2(\alpha_1'+\ldots+\alpha_L'+\beta_1'+\ldots+\beta_L') } \overline{\Lambda}(2\alpha_1',\ldots,2\alpha_L';2\beta_1',\ldots,2\beta_L') \,,
    \end{align*}
    where in the second inequality we use the fact that $k_\ell \geq \frac{\alpha_\ell}{2}$ implies that $k_\ell \geq \lceil \frac{\alpha_\ell}{2} \rceil \geq \lfloor \frac{\alpha_\ell}{2} \rfloor$ as $k_\ell \in \mathbb N$.
    Thus, it suffices to show that (note that $\lfloor \frac{\alpha_\ell}{2} \rfloor \leq \lceil \frac{\alpha_\ell}{2} \rceil \leq \alpha_\ell$)
    \begin{align}{\label{eq-bound-growth-Lambda(alpha-ell-beta-ell)-even}}
        \overline{\Lambda}(2\alpha_1,\ldots,2\alpha_L; 2\beta_1,\ldots,2\beta_L) \leq (2D)^{2\alpha_1+\ldots+2\alpha_L+2\beta_1+\ldots+2\beta_L} \Lambda(0,\ldots,0;0,\ldots,0) \,.
    \end{align}
    Note that 
    \begin{align*}
        &\overline{\Lambda}(2\alpha_1,\ldots,2\alpha_L; 2\beta_1,\ldots,2\beta_L) \\
        \leq\ & \sum_{ \substack{ \alpha_\ell \leq k_\ell \leq D \\ \beta_\ell \leq m_\ell \leq D } } \mathbb E\left\{ \prod_{1 \leq \ell \leq L} \frac{ \lambda^{2k_\ell} k_\ell^{2\alpha_\ell} U_{\ell,*}(t)^{2k_\ell-2\alpha_\ell} }{ k_\ell! } \prod_{1 \leq \ell \leq L} \frac{ \lambda^{2m_\ell} m_\ell^{2\beta_\ell} V_{\ell,*}(t)^{2m_\ell-2\beta_\ell} }{ m_\ell! } \right\} \\ 
        \leq\ & \sum_{ \substack{ 0 \leq k_\ell \leq D-\alpha_\ell \\ 0 \leq m_\ell \leq D-\beta_\ell } } \mathbb E\left\{ \prod_{1 \leq \ell \leq L} \frac{ \lambda^{2(k_\ell+\alpha_\ell)} (k_\ell+\alpha_\ell)^{2\alpha_\ell} U_{\ell,*}(t)^{2k_\ell} }{ (k_\ell+\alpha_\ell)! } \prod_{1 \leq \ell \leq L} \frac{ \lambda^{2(m_\ell+\beta_\ell)} (m_\ell+\beta_\ell)^{2\beta_\ell} V_{\ell,*}(t)^{2m_\ell} }{ (m_\ell+\beta_\ell)! } \right\} \\
        \leq\ & \sum_{ \substack{ 0 \leq k_\ell \leq D-\alpha_\ell \\ 0 \leq m_\ell \leq D-\beta_\ell } } \mathbb E\left\{ \prod_{1 \leq \ell \leq L} \frac{ \lambda^{2k_\ell} (2D)^{2\alpha_\ell} U_{\ell,*}(t)^{2k_\ell} }{ k_\ell! } \prod_{1 \leq \ell \leq L} \frac{ \lambda^{2m_\ell} (2D)^{2\beta_\ell}V_{\ell,*}(t)^{2m_\ell} }{ m_\ell! } \right\} \\
        \leq\ & (4D^2)^{\alpha_1+\ldots+\alpha_L+\beta_1+\ldots+\beta_L} \sum_{ \substack{ 0 \leq k_\ell \leq D-\alpha_\ell \\ 0 \leq m_\ell \leq D-\beta_\ell } } \mathbb E\left\{ \prod_{1 \leq \ell \leq L} \frac{ \lambda^{2k_\ell} U_{\ell,*}(t)^{2k_\ell} }{ k_\ell! } \prod_{1 \leq \ell \leq L} \frac{ \lambda^{2m_\ell} V_{\ell,*}(t)^{2m_\ell} }{ m_\ell! } \right\} \\
        \leq\ & (4D^2)^{\alpha_1+\ldots+\alpha_L+\beta_1+\ldots+\beta_L} \Lambda(0,\ldots,0;0,\ldots,0) \,,
    \end{align*}
    yielding \eqref{eq-bound-growth-Lambda(alpha-ell-beta-ell)-even} and thus leading to \eqref{eq-bound-growth-Lambda(alpha-ell-beta-ell)}. Now, as $D=n^{o(1)}$, we have
    \begin{align}
        F_t &\overset{\eqref{eq-F-t-expand}}{=} \Lambda(0,\ldots,0;0,\ldots,0) + \sum n^{-\frac{1}{2}(\alpha_1+\ldots+\alpha_L+\beta_1+\ldots+\beta_L)} \Lambda(\alpha_1,\ldots,\alpha_L;\beta_1,\ldots,\beta_L) \nonumber \\
        &\overset{\eqref{eq-bound-growth-Lambda(alpha-ell-beta-ell)}}{=} (1+n^{-\frac{1}{2}+o(1)}) \Lambda(0,\ldots,0;0,\ldots,0) \,. \label{eq-F-t-approx-Lambda-0}
    \end{align}
    In addition, using \eqref{eq-F-t-expand} and \eqref{eq-F-t+1-expand} and noting that
    \begin{align*}
        &\mathbb E\left\{ \prod_{1 \leq \ell \leq L} \left( \frac{\sin(\ell\theta_{t+1})}{\sqrt{n}} \right)^{\alpha_\ell} \left( \frac{\cos(\ell\theta_{t+1})}{\sqrt{n}} \right)^{\beta_\ell} \right\} - \mathbb E\left\{ \prod_{1 \leq \ell \leq L} \left( \frac{\zeta_\ell(t+1)}{\sqrt{n}} \right)^{\alpha_\ell} \left( \frac{\eta_\ell(t+1)}{\sqrt{n}} \right)^{\beta_\ell} \right\} \\
        =\ &
        \begin{cases}
            0, & \alpha_1+\ldots+\alpha_L+\beta_1+\ldots+\beta_L \leq 2 \,; \\
            n^{ -(\frac{1}{2}+o(1))(\alpha_1+\ldots+\alpha_L+ \beta_1+\ldots+\beta_L) }, & \mbox{otherwise} \,.
        \end{cases}
    \end{align*}
    We then have
    \begin{align*}
        F_{t+1}-F_t &= \sum_{ \sum \alpha_\ell + \sum \beta_\ell \geq 3 } n^{ -(\frac{1}{2}+o(1))(\alpha_1+\ldots+\alpha_L+ \beta_1+\ldots+\beta_L) } \cdot \Lambda(\alpha_1,\ldots,\alpha_L;\beta_1,\ldots,\beta_L) \\
        &\overset{\eqref{eq-bound-growth-Lambda(alpha-ell-beta-ell)}}{=} \sum_{ \sum \alpha_\ell + \sum \beta_\ell \geq 3 } n^{ -(\frac{1}{2}+o(1))(\alpha_1+\ldots+\alpha_L+ \beta_1+\ldots+\beta_L) } \cdot \Lambda(0,\ldots,0;0,\ldots,0) \\
        &= n^{-\frac{3}{2}+o(1)} \Lambda(0,\ldots,0;0,\ldots,0) \overset{\eqref{eq-F-t-approx-Lambda-0}}{=} n^{-\frac{3}{2}+o(1)} F_t \,,
    \end{align*}
    where in the third equality we use the fact that $\#\{ (\alpha_\ell,\beta_\ell)_{1 \leq \ell \leq L}: \sum \alpha_\ell + \sum \beta_\ell =k \} \leq (2L)^k$ and $L=n^{o(1)}$.
\end{proof}
We now provide the proof of Lemma~\ref{lem-bound-adv-angular-synchronization}.
\begin{proof}[Proof of Lemma~\ref{lem-bound-adv-angular-synchronization}]
    Using Lemma~\ref{lem-interpolate-F-t-F-t+1}, we have 
    \begin{align*}
        \eqref{eq-bound-Adv-relax-2} &= F_0 = [1+O(n^{-\frac{3}{2}+o(1)})]^n F_n = [1+n^{-\frac{1}{2}+o(1)}] F_n \\
        &= \mathbb E\left\{ \prod_{1 \leq \ell \leq L} \exp_{\leq D}\left( \lambda^2 \zeta_{\ell}^2 \right) \prod_{1 \leq \ell \leq L} \exp_{\leq D}\left( \lambda^2 \eta_{\ell}^2 \right) \right\} \\
        &\leq \prod_{1 \leq \ell \leq L} \mathbb E\left\{ \exp\left( \lambda^2 \zeta_{\ell}^2 \right) \right\} \prod_{1 \leq \ell \leq L} \mathbb E\left\{ \exp\left( \lambda^2 \eta_{\ell}^2 \right) \right\} \,,
    \end{align*}
    where in the second inequality we use the independence in $\{ \zeta_\ell,\eta_\ell:1 \leq \ell \leq L \}$ and the fact that $\exp_{\leq D}(x) \leq \exp(x)$ for any $x \geq 0$. Since $\zeta_\ell,\eta_\ell \overset{i.i.d.}{\sim} \mathcal N(0,\frac{1}{2})$ and $\lambda<1-\Omega(1)$, we then have 
    \begin{align*}
        \mathbb E\left\{ \exp\left( \lambda^2 \zeta_{\ell}^2 \right) \right\}, \ \mathbb E\left\{ \exp\left( \lambda^2 \eta_{\ell}^2 \right) \right\} = O(1) \,,
    \end{align*}
    which implies that $\eqref{eq-bound-Adv-relax-2}=\exp(O(L))$ and leads to the desired result.
\end{proof}

\subsection{Putting it together}{\label{subsec:conclusion-angular-synchronization}}

We can now finish the proof of Theorem~\ref{Main-thm-angular-synchronization}.
\begin{proof}[Proof of Theorem~\ref{Main-thm-angular-synchronization}]
    Suppose on the contrary that there exists an algorithm $\mathcal A$ with running time $\exp(n^{o(1)})$ that takes $\bm Y\sim \Pb_{\lambda,L,n}$ as input and achieves weak recovery in the sense of Definition~\ref{def-recovery-add-Gaussian-model}. Using Lemma~\ref{lem-imbalanced-detection-angular-synchronization}, we see that there exists an $(T_n;c_n;\epsilon_n)$-test between $\Pb_{\lambda',L,n}$ and $\Qb_{L,n}$ such that
    \begin{align*}
        T_n = n^{o(1)}, \quad c_n=\Omega(1), \quad \epsilon_n = \exp\left( -\Omega(n) \right) \,.
    \end{align*}
    However, using Lemma~\ref{lem-bound-adv-angular-synchronization} (by putting $D_n=L_n\log n+T_n=n^{o(1)}$) and Proposition~\ref{thm-alg-contiguity}, we see that there is no $(T_n';c_n';\epsilon_n')$-test between $\Pb_{\lambda',L,n}$ and $\Qb_{L,n}$ such that 
    \begin{equation*}
        T_n' = n^{o(1)}, \quad c_n'=\Omega(1), \quad \epsilon_n' = \exp\left( -\omega(L) \right) \,.
    \end{equation*}
    This forms a contradiction and thus yields Theorem~\ref{Main-thm-angular-synchronization} (recall that $L=n^{o(1)}$).
\end{proof}

\section{Orthogonal group synchronization}{\label{sec:proof-group-synchronization}}

This section is devoted to the proof of Theorem~\ref{Main-thm-group-synchronization}.

\subsection{Reducing to an imbalanced detection problem}{\label{subsec:reduce-to-detection-group-synchronization}}

The following lemma enables us to instead consider a detection problem with lopsided success probability. Recall that $\lambda=1-\delta$ for some constant $\delta>0$. Fix a constant $\kappa>0$ such that
\begin{equation}{\label{eq-choice-kappa-group-synchronization}}
    \lambda'=(1+\kappa^2)\lambda<1-\Omega(1) \,.
\end{equation}
In addition, let $\Qb=\Qb_{d,n}$ be the law of a $dn*dn$ GOE matrix, and let $\Pb=\Pb_{\lambda',d,n}$ be the law of the orthogonal group synchronization model with parameter $\lambda'$ defined in Definition~\ref{def-orthogonal-group-synchronization}. 
\begin{lemma}{\label{lem-imbalanced-detection-group-synchronization}}
    Suppose that there exists an algorithm with running time $\exp(n^{o(1)})$ that takes $\bm Y\sim\Pb_{\lambda,d,n}$ as input and achieves weak recovery in the sense of Definition~\ref{def-recovery-add-Gaussian-model}. Then there exists a $(T_n;c_n;\epsilon_n)$-test between $\Pb_{\lambda',d,n}$ and $\Qb_{d,n}$ with
    \begin{align*}
        T_n=n^{o(1)}; \quad c_n=\Omega(1); \quad \epsilon_n= \exp\left( -\Omega(n) \right) \,.
    \end{align*}
\end{lemma}
\begin{proof}
    Recall that for $\bm Y\sim\Pb_{\lambda,d,n}$, $\bm Y$ is an additive Gaussian model with $\Theta=\frac{\lambda}{\sqrt{n}} \bm U \bm U^{\top}$, and thus $\|\Theta\|_{\Fop}^2 = \lambda^2 dn$. Thus, from Lemma~\ref{lem-reduce-to-detection-add-Gaussian-model} we see that there is a $(T_n;c_n;\epsilon_n)$-test between $\Pb_{\lambda',d,n}$ and $\Qb_{d,n}$ with
    \begin{equation*}
        T_n = n^{o(1)}; \quad c_n=\Omega(1); \quad \epsilon_n= \exp\left( -\Omega(n) \right) \,.  \qedhere
    \end{equation*}
\end{proof}

\subsection{Bounding the low-degree advantage}{\label{subsec:bounding-low-deg-adv-group-synchronization}}

\begin{lemma}{\label{lem-bound-adv-group-synchronization}}
    Suppose that $\lambda=1-\Omega(1)$, we have $\mathsf{Adv}_{\leq D}(\Pb_{\lambda,d,n};\Qb_{d,n})^2 \leq \exp(O(d^2))$ for any $D,d=n^{o(1)}$.
\end{lemma}
The rest of this subsection is devoted to the proof of Lemma~\ref{lem-bound-adv-group-synchronization}. We will denote $\mathbb S=\mathsf{O}(d)$ throughout this subsection. Similar to the proof of Lemma~\ref{lem-bound-adv-angular-synchronization}, we first use standard tools for general Gaussian additive models \cite{KWB22} to derive the following bound on low-degree advantage.
\begin{lemma}{\label{lem-adv-group-relax-1}}
    Recall \eqref{eq-exp-leq-D} and recall . We then have 
    \begin{align}{\label{eq-bound-Adv-group-relax-1}}
        \mathsf{Adv}_{\leq D}(\Pb\|\Qb)^2 \leq \mathbb E_{\bm O_1,\ldots,\bm O_n \sim \operatorname{Unif}(\mathbb S)}\left\{ \prod_{1 \leq \ell,m \leq d} \exp_{\leq D}\left( \lambda^2 U(\ell,m)^2 \right) \right\} \,.
    \end{align}
    Here 
    \begin{align}{\label{eq-def-U-ell-m}}
        U(\ell,m) = \frac{1}{\sqrt{n}} \sum_{t=1}^{n} \bm O_t(\ell,m) \,.
    \end{align}
\end{lemma}
\begin{proof}
    Using \cite{KWB22}, we have (below we write $\widetilde{\bm U}=(\widetilde{\bm O}_1,\ldots,\widetilde{\bm O}_n)$ to be an independent copy of $\bm U$)
    \begin{align*}
        \mathsf{Adv}_{\leq D}(\Pb\|\Qb)^2 &= \mathbb E_{\bm U,\widetilde{\bm U}}\left\{ \exp_{\leq D}\left( \frac{\lambda^2}{n} \left\langle \bm U^{\top} \bm U, \widetilde{\bm U}^{\top} \widetilde{\bm U} \right\rangle \right) \right\} = \mathbb E_{\bm U,\widetilde{\bm U}}\left\{ \exp_{\leq D}\left( \frac{\lambda^2}{n} \left\| \bm U \widetilde{\bm U}^{\top} \right\|_{\Fop}^2 \right) \right\} \\
        &= \mathbb E_{\bm U,\widetilde{\bm U}}\left\{ \exp_{\leq D}\left( \frac{\lambda^2}{n} \left\| \sum_{i=1}^{n} \bm O_i \widetilde{\bm O}_i^{\top} \right\|_{\Fop}^2 \right) \right\} = \mathbb E_{\bm U,\widetilde{\bm U}}\left\{ \exp_{\leq D}\left( \frac{\lambda^2}{n} \left\| \sum_{i=1}^{n} \bm O_i \right\|_{\Fop}^2 \right) \right\} \\
        &= \mathbb E_{\bm O_1,\ldots,\bm O_n \sim \operatorname{Unif}(\mathbb S)}\left\{ \exp_{\leq D}\left( \lambda^2 \sum_{1 \leq \ell,m \leq d} U(\ell,m)^2 \right) \right\} \\
        &\leq \mathbb E_{\bm O_1,\ldots,\bm O_n \sim \operatorname{Unif}(\mathbb S)}\left\{ \prod_{1 \leq \ell,m \leq d} \exp_{\leq D}\left( \lambda^2 U(\ell,m)^2 \right) \right\} \,,
    \end{align*}
    where the fourth equality follows from $( \bm O_1 \widetilde{\bm O}_1^{\top}, \ldots, \bm O_n \widetilde{\bm O}_n^{\top}) \overset{d}{=} ( \bm O_1, \ldots, \bm O_n )$, the fifth equality follows from \eqref{eq-def-U-ell-m}, and the inequality follows from $\exp_{\leq D}(x+y) \leq \exp_{\leq D}(x) \exp_{\leq D}(y)$ for $x,y \geq 0$.
\end{proof}
Provided with Lemma~\ref{lem-adv-group-relax-1}, it suffices to bound the right hand side of \eqref{eq-bound-Adv-group-relax-1} when $\lambda<1-\Omega(1)$. Note that
\begin{align}
    \eqref{eq-bound-Adv-group-relax-1} &= \mathbb E\left\{ \prod_{1 \leq \ell,m \leq d} \left( \sum_{ 0 \leq k_{\ell,m} \leq D } \frac{ \lambda^{2k_{\ell,m}} U(\ell,m)^{2k_{\ell,m}} }{ k_{\ell,m}! } \right) \right\} \nonumber \\
    &= \sum_{ \substack{ \{ k_{\ell,m}: 1 \leq \ell,m \leq d \} \\ k_{\ell,m} \leq D } } \prod_{ 1 \leq \ell,m \leq d } \frac{ \lambda^{ 2k_{\ell,m} } }{ k_{\ell,m}! } \mathbb E\left\{ \prod_{ 1 \leq \ell,m \leq d } U(\ell,m)^{2k_{\ell,m}} \right\} \,.  \label{eq-bound-Adv-group-relax-2}
\end{align}
The basic intuition behind our approach of bounding \eqref{eq-bound-Adv-group-relax-2} is that according to \eqref{eq-def-U-ell-m} and central limit theorem, we should expect that 
\begin{align}{\label{eq-Gaussian-approx-intuition-group}}
    \Big( U(\ell,m):1\leq \ell,m \leq d \Big) \mbox{ behaves like } \Big( \zeta(\ell,m):1 \leq \ell,m \leq d \Big) \overset{i.i.d.}{\sim} \mathcal N\left( 0,\frac{1}{d} \right) \,.
\end{align}
The key of our proof is to show that such Gaussian approximation is indeed valid for all low-order moments via a delicate Lindeberg’s interpolation argument. To this end, we sample 
\begin{align*}
    \Big\{ \zeta_t(\ell,m): 1 \leq t \leq n, 1 \leq \ell,m \leq d \Big\} \overset{i.i.d.}{\sim} \mathcal N\left( 0,\frac{1}{d} \right)
\end{align*}
Define
\begin{equation}{\label{eq-def-U-ell-m(t)}}
    U_t(\ell,m)= \frac{1}{\sqrt{n}}\left( \sum_{ 1 \leq j \leq t } \zeta_j(\ell,m) + \sum_{t+1 \leq j \leq n} \bm O_j(\ell,m) \right) \,.
\end{equation}
And define
\begin{align}{\label{eq-def-F-t-group}}
    F_t = \sum_{ \substack{ \{ k(\ell,m): 1 \leq \ell,m \leq d \} \\ k(\ell,m) \leq D } } \prod_{ 1 \leq \ell,m \leq d } \frac{ \lambda^{ 2k(\ell,m) } }{ k(\ell,m)! } \mathbb E\left\{ \prod_{ 1 \leq \ell,m \leq d } U_t(\ell,m)^{2k(\ell,m)} \right\} \,.
\end{align}
In particular, we have (note that $\{ U_{n}(\ell,m):1 \leq \ell,m \leq d \} \overset{d}{=} \{ \zeta(\ell,m):1 \leq \ell,m \leq d \}$)
\begin{align*}
    F_0 &= \sum_{ \substack{ \{ k(\ell,m): 1 \leq \ell,m \leq d \} \\ k(\ell,m) \leq D } } \prod_{ 1 \leq \ell,m \leq d } \frac{ \lambda^{ 2k(\ell,m) } }{ k(\ell,m)! } \mathbb E\left\{ \prod_{ 1 \leq \ell,m \leq d } U(\ell,m)^{2k(\ell,m)} \right\} = \eqref{eq-bound-Adv-group-relax-2} \,; \\
    F_n &= \sum_{ \substack{ \{ k(\ell,m): 1 \leq \ell,m \leq d \} \\ k(\ell,m) \leq D } } \prod_{ 1 \leq \ell,m \leq d } \frac{ \lambda^{ 2k(\ell,m) } }{ k(\ell,m)! } \mathbb E\left\{ \prod_{ 1 \leq \ell,m \leq d } \zeta(\ell,m)^{2k(\ell,m)} \right\} \,.
\end{align*}
We now show that $F_t$ and $F_{t+1}$ are ``close'' in a certain sense, as incorporated in the next lemma.
\begin{lemma}{\label{lem-interpolate-F-t-F-t+1-group}}
    Suppose that $\lambda\leq 1$ and $D,d=n^{o(1)}$, we then have $F_{t+1}=[1+O(n^{-\frac{3}{2}+o(1)})]F_t$ for all $0 \leq t \leq n-1$.
\end{lemma}
\begin{proof}
    For $0 \leq t \leq n-1$, define 
    \begin{equation}{\label{eq-def-U-t-*-(ell,m)}}
        U_{t,*}(\ell,m)= \frac{1}{\sqrt{n}}\left( \sum_{ 1 \leq j \leq t } \zeta_j(\ell,m) + \sum_{t+2 \leq j \leq n} \bm O_j(\ell,m) \right) \,.
    \end{equation}
    It is clear that
    \begin{align*}
        & U_{t}(\ell,m)=U_{t,*}(\ell,m)+ \frac{1}{\sqrt{n}} \bm O_{t+1}(\ell,m) \,, \\ 
        & U_{t+1}(\ell,m)=U_{t,*}(\ell,m)+\frac{1}{\sqrt{n}} \zeta_{t+1}(\ell,m) \,.
    \end{align*}
    Thus, we have that 
    \begin{align}
        F_t &= \sum_{ \substack{ \{ k(\ell,m): 1 \leq \ell,m \leq d \} \\ k(\ell,m) \leq D } } \mathbb E\left\{ \prod_{ 1 \leq \ell,m \leq d } \frac{ \lambda^{2k(\ell,m)} U_t(\ell,m)^{2k(\ell,m)} }{ k(\ell,m)! } \right\} \nonumber \\
        &= \sum_{ \substack{ \{ k(\ell,m): 1 \leq \ell,m \leq d \} \\ k(\ell,m) \leq D } } \mathbb E\left\{ \prod_{ 1 \leq \ell,m \leq d } \frac{ \lambda^{2k(\ell,m)} (U_{t,*}(\ell,m)+\frac{1}{\sqrt{n}} \bm O_{t+1}(\ell,m))^{2k(\ell,m)} }{ k(\ell,m)! } \right\}  \nonumber \\
        &= \sum_{ \substack{ \{ \alpha(\ell,m): 1 \leq \ell,m \leq d \} \\ \alpha(\ell,m) \leq 2D } } \Lambda\left( \{\alpha(\ell,m)\} \right) \cdot \mathbb E\left\{ \prod_{1 \leq \ell,m \leq d} \left( \frac{ \bm O_{t+1}(\ell,m) }{ \sqrt{n} } \right)^{\alpha(\ell,m)} \right\} \,,  \label{eq-F-t-expand-group}
    \end{align}
    where $\Lambda( \{\alpha(\ell,m)\} )$ is defined to be
    \begin{align}{\label{eq-def-Lambda(alpha-ell-m)}}
        \sum_{ \substack{ \frac{1}{2}\alpha(\ell,m) \leq k(\ell,m) \leq D } } \mathbb E\left\{ \prod_{1 \leq \ell,m \leq d} \binom{2k(\ell,m)}{\alpha(\ell,m)} \frac{ \lambda^{2k(\ell,m)} U_{t,*}(\ell,m)^{2k(\ell,m)-\alpha(\ell,m)} }{ k(\ell,m)! }  \right\} \,.
    \end{align}
    Similarly, we have
    \begin{align}
        F_{t+1} = \sum_{ \substack{ \{ \alpha(\ell,m): 1 \leq \ell,m \leq d \} \\ \alpha(\ell,m) \leq 2D } } \Lambda\left( \{\alpha(\ell,m)\} \right) \cdot \mathbb E\left\{ \prod_{1 \leq \ell,m \leq d} \left( \frac{ \zeta_{t+1}(\ell,m) }{ \sqrt{n} } \right)^{\alpha(\ell,m)} \right\} \,.  \label{eq-F-t+1-expand-group}
    \end{align}
    We first argue that
    \begin{align}{\label{eq-bound-growth-Lambda(alpha-ell-m)}}
        \Lambda(\{\alpha(\ell,m)\}) \leq (16D^2)^{ \sum_{1 \leq \ell,m \leq d} \alpha(\ell,m) } \Lambda(\{0\}) \,.
    \end{align}
    To this end, note that (below we use $\binom{b}{a}\leq b^a$) 
    \begin{align*}
        \binom{2k(\ell,m)}{\alpha(\ell,m)} U_{\ell,*}(t)^{2k(\ell,m)-\alpha(\ell,m)} \leq\ & \frac{1}{2} (2k(\ell,m))^{2\lfloor\alpha(\ell,m)/2\rfloor} U_{t,*}(\ell,m)^{2k(\ell,m)-2\lfloor\alpha(\ell,m)/2\rfloor} \\
        + & \frac{1}{2} (2k(\ell,m))^{2\lceil\alpha(\ell,m)/2\rceil} U_{t,*}(\ell,m)^{2k(\ell,m)-2\lceil\alpha(\ell,m)/2\rceil}  \,.
    \end{align*}
    We have that 
    \begin{align*}
        &\Lambda(\{\alpha(\ell,m)\}) \\
        \leq\ & \frac{1}{2^{d^2}} \sum_{ \substack{ \alpha'(\ell,m) \in \{ \lfloor\alpha(\ell,m)/2\rfloor, \lceil\alpha(\ell,m)/2\rceil \} \\ \frac{1}{2}\alpha(\ell,m) \leq k(\ell,m) \leq D } } \mathbb E\left\{ \prod_{1 \leq \ell,m \leq d} \frac{ \lambda^{2k(\ell,m)} (2k(\ell,m))^{2\alpha'(\ell,m)} U_{t,*}(\ell,m)^{2k(\ell,m)-2\alpha'(\ell,m)} }{ k(\ell,m)! } \right\} \\ 
        \leq\ & \frac{1}{2^{d^2}} \sum_{ \substack{ \alpha'(\ell,m) \in \{ \lfloor\alpha(\ell,m)/2\rfloor, \lceil\alpha(\ell,m)/2\rceil \} \\ \alpha'(\ell,m) \leq k(\ell,m) \leq D } } \mathbb E\left\{ \prod_{1 \leq \ell,m \leq d} \frac{ \lambda^{2k(\ell,m)} (2k(\ell,m))^{2\alpha'(\ell,m)} U_{t,*}(\ell,m)^{2k(\ell,m)-2\alpha'(\ell,m)} }{ k(\ell,m)! } \right\} \\
        :=\ & \frac{1}{2^{d^2}} \sum_{ \substack{ \alpha'(\ell,m) \in \{ \lfloor\alpha(\ell,m)/2\rfloor, \lceil\alpha(\ell,m)/2\rceil \} } } 2^{ 2\sum_{1\leq \ell,m \leq d} \alpha(\ell,m) } \overline{\Lambda}(\{2\alpha(\ell,m)'\}) \,,
    \end{align*}
    where in the second inequality we use the fact that $k(\ell,m) \geq \frac{\alpha(\ell,m)}{2}$ implies that $k(\ell,m) \geq \lceil \frac{\alpha(\ell,m)}{2} \rceil \geq \lfloor \frac{\alpha(\ell,m)}{2} \rfloor$ as $k(\ell,m) \in \mathbb N$. Thus, it suffices to show that (note that $\lfloor \frac{\alpha_\ell}{2} \rfloor \leq \lceil \frac{\alpha_\ell}{2} \rceil \leq \alpha_\ell$)
    \begin{align}{\label{eq-bound-growth-Lambda(alpha-ell-m)-even}}
        \overline{\Lambda}(\{2\alpha(\ell,m)\}) \leq (2D)^{2\sum_{1 \leq \ell,m \leq d}\alpha(\ell,m)} \Lambda(\{0\}) \,.
    \end{align}
    Note that 
    \begin{align*}
        &\overline{\Lambda}(\{2\alpha(\ell,m)\}) \\
        \leq\ & \sum_{ \substack{ \alpha(\ell,m) \leq k(\ell,m) \leq D } } \mathbb E\left\{ \prod_{1 \leq \ell,m \leq d} \frac{ \lambda^{2k(\ell,m)} (2k(\ell,m))^{2\alpha'(\ell,m)} U_{t,*}(\ell,m)^{2k(\ell,m)-2\alpha'(\ell,m)} }{ k(\ell,m)! } \right\} \\ 
        \leq\ & \sum_{ \substack{ 0 \leq k(\ell,m) \leq D-\alpha(\ell,m) } } \mathbb E\left\{ \prod_{1 \leq \ell,m \leq d} \frac{ \lambda^{2(k(\ell,m)+\alpha(\ell,m))} (k(\ell,m)+\alpha(\ell,m))^{2\alpha(\ell,m)} U_{t,*}(t)^{2k(\ell,m)} }{ (k(\ell,m)+\alpha(\ell,m))! }  \right\} \\
        \leq\ & \sum_{ \substack{ 0 \leq k(\ell,m) \leq D-\alpha(\ell,m) } } \mathbb E\left\{ \prod_{1 \leq \ell,m \leq d} \frac{ \lambda^{2k(\ell,m)} (2D)^{2\alpha(\ell,m)} U_{t,*}(\ell,m)^{2k(\ell,m)} }{ k(\ell,m)! } \right\} \\
        \leq\ & (4D^2)^{\sum_{1\leq \ell,m \leq d} \alpha(\ell,m)} \sum_{ \substack{ 0 \leq k(\ell,m) \leq D-\alpha(\ell,m) } } \mathbb E\left\{ \prod_{1 \leq \ell,m \leq d} \frac{ \lambda^{2k(\ell,m)} U_{t,*}(\ell,m)^{2k(\ell,m)} }{ k(\ell,m)! } \right\} \\
        \leq\ & (4D^2)^{\sum_{1\leq \ell,m \leq d} \alpha(\ell,m)} \Lambda(\{0\}) \,,
    \end{align*}
    yielding \eqref{eq-bound-growth-Lambda(alpha-ell-m)-even} and thus leading to \eqref{eq-bound-growth-Lambda(alpha-ell-m)}. Now, as $D=n^{o(1)}$, we have
    \begin{align}
        F_t &\overset{\eqref{eq-F-t-expand-group}}{=} \Lambda(\{0\}) + \sum n^{-\frac{1}{2}\sum_{1\leq \ell,m\leq d}\alpha(\ell,m)} \Lambda(\{\alpha(\ell,m)\}) \nonumber \\
        &\overset{\eqref{eq-bound-growth-Lambda(alpha-ell-m)}}{=} \left( 1+n^{-\frac{1}{2}+o(1)} \right) \Lambda(\{0\}) \,. \label{eq-F-t-approx-Lambda-0-group}
    \end{align}
    In addition, using \eqref{eq-F-t-expand} and \eqref{eq-F-t+1-expand} and noting that
    \begin{align*}
        &\mathbb E\left\{ \prod_{1 \leq \ell,m \leq d} \left( \frac{\bm O_{t+1}(\ell,m)}{\sqrt{n}} \right)^{\alpha(\ell,m)} \right\} - \mathbb E\left\{ \prod_{1 \leq \ell,m \leq d} \left( \frac{\zeta_{t+1}(\ell,m)}{\sqrt{n}} \right)^{\alpha(\ell,m)} \right\} \\
        =\ &
        \begin{cases}
            0, & \sum_{1 \leq \ell,m \leq d} \alpha(\ell,m) \leq 2 \,; \\
            n^{ -(\frac{1}{2}+o(1))\sum_{1 \leq \ell,m \leq d} \alpha(\ell,m) }, & \mbox{otherwise} \,.
        \end{cases}
    \end{align*}
    We then have
    \begin{align*}
        F_{t+1}-F_t &= \sum_{ \sum \alpha(\ell,m) \geq 3 } n^{ -(\frac{1}{2}+o(1))\sum_{1 \leq \ell,m \leq d} \alpha(\ell,m) } \cdot \Lambda(\{\alpha(\ell,m)\}) \\
        &\overset{\eqref{eq-bound-growth-Lambda(alpha-ell-m)}}{=} \sum_{ \sum \alpha(\ell,m) \geq 3 } n^{ -(\frac{1}{2}+o(1))\sum_{1 \leq \ell,m \leq d} \alpha(\ell,m) } \cdot \Lambda(\{0\}) \\
        &= n^{-\frac{3}{2}+o(1)} \Lambda(\{0\}) \overset{\eqref{eq-F-t-approx-Lambda-0-group}}{=} n^{-\frac{3}{2}+o(1)} F_t \,,
    \end{align*}
    where in the third equality we use the fact that $\#\{ (\alpha(\ell,m))_{1 \leq \ell,m \leq d}: \sum \alpha(\ell,m)=k \} \leq (2d^2)^k$ and $d=n^{o(1)}$.
\end{proof}

\subsection{Putting it together}{\label{subsec:conclusion-group-synchronization}}

We can now finish the proof of Theorem~\ref{Main-thm-group-synchronization}.
\begin{proof}[Proof of Theorem~\ref{Main-thm-group-synchronization}]
    Suppose on the contrary that there exists an algorithm $\mathcal A$ with running time $\exp(n^{o(1)})$ that takes $\bm Y\sim \Pb_{\lambda,d,n}$ as input and achieves weak recovery in the sense of Definition~\ref{def-recovery-add-Gaussian-model}. Using Lemma~\ref{lem-imbalanced-detection-group-synchronization}, we see that there exists an $(T_n;c_n;\epsilon_n)$-test between $\Pb_{\lambda',d,n}$ and $\Qb_{d,n}$ such that
    \begin{align*}
        T_n = n^{o(1)}, \quad c_n=\Omega(1), \quad \epsilon_n = \exp\left( -\Omega(n) \right) \,.
    \end{align*}
    However, using Lemma~\ref{lem-bound-adv-group-synchronization} (by putting $D_n=d_n\log n+T_n=n^{o(1)}$) and Proposition~\ref{thm-alg-contiguity}, we see that there is no $(T_n';c_n';\epsilon_n')$-test between $\Pb_{\lambda',d,n}$ and $\Qb_{d,n}$ such that 
    \begin{equation*}
        T_n' = n^{o(1)}, \quad c_n'=\Omega(1), \quad \epsilon_n' = \exp\left( -\omega(d^2) \right) \,.
    \end{equation*}
    This forms a contradiction and thus yields Theorem~\ref{Main-thm-group-synchronization} (recall that $d=n^{o(1)}$).
\end{proof}

\section{Multi-layer stochastic block model}{\label{sec:proof-multilayer-SBM}}

This section is devoted to the proof of Theorem~\ref{Main-thm-multilayer-SBM}. Note that due to similar reasons as in Section~\ref{sec:proof-SBM}, we will assume throughout this section that 
\begin{equation}{\label{eq-condition-multilayer-SBM}}
    d_1\lambda_1^2=\Theta(1) \mbox{ and } F(\rho,\{d_\ell\},\{\lambda_\ell\}) \leq 1-\delta \mbox{ for some constant } \delta>0  \,.
\end{equation}
Also, since $F(\rho,\{ d_\ell \},\{ \lambda_\ell \})$ is $\delta^{-2}$-Lipschitz to $\lambda_\ell^2 d_\ell$ when $\lambda_\ell^2 d_\ell \leq 1-\delta$, we can also assume that
\begin{equation}{\label{eq-condition-multilayer-SBM-technical}}
    d_\ell \lambda_\ell^2 \geq \frac{\delta^3}{2L^2} \mbox{ for all } 1 \leq \ell \leq L \,.
\end{equation}

\subsection{Reducing to an imbalanced detection problem}{\label{subsec:reduce-to-detection-multilayer-SBM}}

The following lemma enables us to instead consider a detection problem with lopsided success probability. Throughout this section, we will also fix a constant $\kappa>1$ such that
\begin{equation}{\label{eq-choice-kappa-multilayer-SBM}}
    F(\rho,\{d_\ell'\},\{\lambda_\ell\}) \leq 1-\delta<1-\Omega(1) \mbox{ where } d_\ell'=\kappa d_\ell \mbox{ for } 1 \leq \ell \leq L \,.
\end{equation} 
In addition, let $\Qb=\Qb_{\{d_\ell'\},n}$ be the law of $L$ \ER graphs on $[n]$ with parameter $\frac{d_\ell'}{n}$ respectively, and let $\Pb=\Pb_{\rho,q,\{d_\ell'\},\{\lambda_\ell\},n}$ be the law of multi-layer SBM defined in Definition~\ref{def-multilayer-SBM}. 
\begin{lemma}{\label{lem-imbalanced-detection-multilayer-SBM}}
    Suppose that there exists an algorithm with running time $\exp(n^{o(1)})$ that takes $\bm Y\sim \Pb_{\rho,q,\{d_\ell'\},\{\lambda_\ell\},n}$ as input and achieves weak recovery of $\bm Y_1$ in the sense of Definition~\ref{def-recovery-Bernoulli-obser-model}. Then there is an $(T_n;c_n;\epsilon_n)$-test between $\Pb_{\rho,q,\{d_\ell'\},\{\lambda_\ell\},n}$ and $\Qb_{\{d_\ell'\},n}$ with
    \begin{align*}
        T_n = n^{o(1)}; \quad c_n=\Omega(1); \quad \epsilon_n= \exp\left( -\Omega(n) \right) \,.
    \end{align*}
\end{lemma}
\begin{proof}
    Define
    \begin{equation}{\label{eq-def-mathcal-K-multilayer-SBM}}
        \mathcal K=\left\{ \mathcal Y \in \mathbb R^{n*n}: \| \mathcal Y \|_{\infty} \leq \frac{(1+(q-1)\lambda_1)d_1}{n},\ \mathcal Y+\frac{\lambda_1 d_1}{n} \mathbb I \succ \mathbb O \right\} \,.
    \end{equation}
    Also define $M \in \mathbb R$ such that
    \begin{equation}{\label{eq-def-M-multilayer-SBM}}
        M^2 = \frac{n^2}{q} \cdot \left( \frac{(1+(q-1)\lambda_1)d_1}{n} \right)^2 + \frac{(q-1)n^2}{q} \cdot \left( \frac{(1-\lambda_1)d_1}{n} \right)^2 = \left( (q-1)\lambda_1^2+1 \right) d_1^2 \,.
    \end{equation}
    Recall that under $\Pb_{\rho,q,\{d_\ell\},\{\lambda_\ell\},n}$, $\bm Y$ is a binary observation model with 
    \begin{align*}
        p=\frac{d_1}{n} \mbox{ and } \Theta= \frac{(q\mathbf 1_{\sigma_1(i)=\sigma_1(j)}-1)\lambda_1 d_1}{n} \,.
    \end{align*}
    It is clear that
    \begin{align*}
        \Pb\left( \frac{M}{2} \leq \|\Theta\|_{\Fop} \leq 2M, \Theta\in\mathcal K \right)=1-o(1) \,.
    \end{align*}
    Suppose an algorithm $\mathcal X(\bm Y)$ with running time $\exp(n^{o(1)})$ that takes $\bm Y\sim \Pb_{\rho,q,\{d_\ell'\},\{\lambda_\ell\},n}$ as input and achieves weak recovery of $\bm Y_1$ in the sense of Definition~\ref{def-recovery-Bernoulli-obser-model}. Using Lemma~\ref{lem-cor-pre-proj}, we see that there exists an algorithm $\widehat{\mathcal X}(\bm Y)$ with running time $\exp(n^{o(1)})$ that takes $\bm Y\sim \Pb_{\rho,q,\{d_\ell'\},\{\lambda_\ell\},n}$ as input and
    \begin{align*}
        \widehat{\mathcal X} \in \mathcal K, \quad \|\widehat{\mathcal X}\|_{\Fop} \geq \Omega(M), \quad \mathbb E_{\Pb}\left[ \frac{ \langle \Theta,\mathcal X \rangle }{ \|\Theta\|_{\Fop} \| \mathcal X \|_{\Fop} } \geq \Omega(1) \right] \geq \Omega(1) \,.
    \end{align*}
    In particular, we have 
    \begin{align*}
        \|\widehat{\mathcal X}\|_{\Fop} \geq \Theta(M) \mbox{ and } \|\widehat{\mathcal X}\|_{\infty}\leq \frac{(1+(q-1)\lambda_1)d_1}{n} \,.
    \end{align*}
    Recall \eqref{eq-choice-kappa-multilayer-SBM}. Applying Lemma~\ref{lem-reduce-to-detection-Ber-obser-model} with $a=b=\kappa^{-1}$, we see that there exists an $(T_n;c_n;\epsilon_n)$-test between $\Pb_{\rho,q,\{d_\ell'\},\{\lambda_\ell\},n}$ and $\Qb_{\{d_\ell'\},n}$ with
    \begin{align*}
        T_n &=n^{o(1)}, \quad c_n=\Omega(1) \,, 
    \end{align*}
    and 
    \begin{align*}
        \epsilon_n &=\exp\left( -\Omega(1)\cdot\frac{M^2}{ \frac{d}{n} + \frac{(1+(q-1)\lambda)d}{n} } \right) \\
        &\overset{\eqref{eq-def-M-multilayer-SBM}}{=} \exp\left( -\Omega(1)\cdot\frac{nd(1+\lambda^2(q-1))}{1+(q-1)\lambda} \right) \overset{\eqref{eq-condition-multilayer-SBM}}{=} \exp\left( -\Omega(n) \right) \,,
    \end{align*}
    thereby finishing the proof.
\end{proof}

\subsection{Bounding the low-degree advantage}{\label{subsec:bounding-low-deg-adv-multilayer-SBM}}

\begin{lemma}{\label{lem-bound-adv-multilayer-SBM}}
    Suppose that $F(\rho,\{d_\ell\},\{\lambda_\ell\}) \leq 1-\Omega(1)$, we then have $\mathsf{Adv}_{\leq D}(\Pb;\Qb)^2 \leq \exp(O(q^2 L^4))$ for any $D,q,L=n^{o(1)}$.
\end{lemma}
The rest of this subsection is devoted to the proof of Lemma~\ref{lem-bound-adv-multilayer-SBM}. To this end, for any $L$-tuple $H_1,\ldots,H_L \Subset \mathsf K_n$, define
\begin{equation}{\label{eq-def-orthogonal-basis}}
    f_{H_1,\ldots,H_L}(\bm Y_1,\ldots,\bm Y_L):= \prod_{1 \leq \ell \leq L} \prod_{(i,j)\in E(H_\ell)} \frac{ \bm Y_\ell(i,j)-\frac{d_\ell}{n} }{ \sqrt{ \frac{d_\ell}{n} (1-\frac{d_\ell}{n}) } } \,.
\end{equation}
It is straightforward to verify that $\{ f_{H_1,\ldots,H_L}: |E(H_1)|+\ldots+|E(H_L)|\leq D \}$ constitutes a standard orthogonal basis of $\mathcal P_D$ under $\Qb$, in the sense that
\begin{equation}{\label{eq-standard-orthogonal}}
    \mathbb E_{\Qb}\left[ f_{H_1,\ldots,H_L}(\bm Y_1,\ldots,\bm Y_L) f_{H_1',\ldots,H_L'}(\bm Y_1,\ldots,\bm Y_L) \right] = \mathbf 1_{ \{ (H_1,\ldots,H_L)=(H_1',\ldots,H_L') \} } \,.
\end{equation}
Thus, it is standard that (see, e.g., \cite{Hopkins18})
\begin{equation}{\label{eq-Adv-transform-I}}
    \mathsf{Adv}_{\leq D}(\Pb;\Qb)^2 = \sum_{ |E(H_1)|+\ldots+|E(H_L)| \leq D } \mathbb E_{\Pb}\Big[ f_{H_1,\ldots,H_L}(\bm Y_1,\ldots,\bm Y_L) \Big]^2 \,.
\end{equation}
We now calculate $\mathbb E_{\Pb}[ f_{H_1,\ldots,H_L}(\bm Y_1,\ldots,\bm Y_L) ]$. Recall \eqref{eq-def-omega(sigma,tau)}. By conditioning on $\Sigma=\{ \sigma_\ell(i): 1 \leq i \leq n, 1 \leq \ell \leq L\}$, we have
\begin{align*}
    &\mathbb E_{\Pb}\Big[ f_{H_1,\ldots,H_L}(\bm Y_1,\ldots,\bm Y_L) \Big] \\
    =\ & \mathbb E_{\Sigma}\Big\{ \mathbb E_{\Pb}\left[ f_{H_1,\ldots,H_L}(\bm Y_1,\ldots,\bm Y_L) \mid \Sigma \right] \Big\} \\
    =\ & \mathbb E_{\Sigma}\left[ \prod_{1 \leq \ell \leq L} \prod_{(i,j)\in E(H_\ell)} \frac{ \frac{d_\ell }{n} \cdot \lambda_\ell \omega(\sigma_\ell(i),\sigma_\ell(j)) }{ \sqrt{ \frac{d_\ell}{n} (1-\frac{d_\ell}{n}) } } \right] \\
    =\ & \prod_{1 \leq \ell \leq L} \left( \Delta_\ell^{ \frac{1}{2}|E(H_\ell)| } n^{-\frac{1}{2}|E(H_\ell)|} \right) \cdot \mathbb E\left[ \prod_{1 \leq \ell \leq L} \prod_{(i,j)\in E(H_\ell)} \omega(\sigma_\ell(i),\sigma_\ell(j)) \right] \,,
\end{align*}
where 
\begin{align*}
    \Delta_\ell = \frac{ \lambda_\ell^2 d_\ell }{ 1-\frac{d_\ell}{n} } = [1+o(1)] \lambda_\ell^2 d_\ell \,.
\end{align*}
Using \eqref{eq-condition-multilayer-SBM}, it is straightforward to check that
\begin{equation}{\label{eq-condition-multilayer-SBM-transform}}
    F(\{ \Delta_\ell \}):= \max\left\{ \max_{1 \leq \ell \leq L}\{ \Delta_\ell \}, \sum_{\ell=1}^{L} \frac{ \rho^2 \Delta_\ell }{ 1-(1-\rho^4) \Delta_\ell } \right\} \leq 1-\delta \mbox{ for some constant } \delta>0 \,.
\end{equation}
Thus, we have
\begin{align}
    \eqref{eq-Adv-transform-I} = \sum_{ |E(H_1)|+\ldots+|E(H_L)| \leq D } \prod_{1 \leq \ell \leq L} \left( \frac{\Delta_\ell}{n} \right)^{|E(H_\ell)|} \cdot \mathbb E\left[ \prod_{1 \leq \ell \leq L} \prod_{(i,j)\in E(H_\ell)} \omega(\sigma_\ell(i),\sigma_\ell(j)) \right]^2  \,.   \label{eq-Adv-transform-II}
\end{align}
It remains to bound the right hand side of \eqref{eq-Adv-transform-II}. To this end, we first introduce some additional notation on multigraphs.
\begin{defn}{\label{eq-def-multi-graph}}
    We say $S \subset \mathsf K_n$ is an $L$-decorated multigraph (or simply multigraph), if $S=(V(S),E(S))$ such that 
    \begin{align*}
        V(S)\subset [n], \quad E(S)\subset \{ (i,j;\ell): i,j \in V(S), i \neq j, 1 \leq \ell \leq L \} \,.
    \end{align*}
    For a multigraph $S$, denote $\widetilde{S} \subset \mathsf K_n$ with
    \begin{align*}
        V(\widetilde S)=V(S), \quad E(\widetilde S)=\{ (i,j): \exists 1 \leq \ell \leq L, (i,j;\ell) \in E(S) \}
    \end{align*}
    as the simple graph (ignore the multiplicity) corresponding to $S$. We say $S \subset \mathsf K_n$ is an $L$-decorated graph, if every edge $(i,j)$ of $S$ has multiplicity $1$.
\end{defn}
To this end, for all $H_1,\ldots,H_L \Subset \mathsf K_n$, define
\begin{align}{\label{eq-def-phi(H-ell)}}
    \psi(H_1,\ldots,H_L):= \cup_{1 \leq \ell \leq L} H_\ell \,.
\end{align}
In addition, define $\phi(H_1,\ldots,H_L)$ to be the multigraph induced by 
\begin{equation}{\label{eq-def-psi(H-ell)}}
    \left\{ (i,j;\ell): (i,j) \in E(H_\ell) \cap E(H_{\ell'}) \mbox{ for some } 1 \leq \ell<\ell' \leq L \right\} \,,
\end{equation}
and let $\varphi(H_1,\ldots,H_L)$ be the graph induced by $\phi(H_1,\ldots,H_L)$. We now have (recall from \eqref{eq-condition-multilayer-SBM-transform} we have $\Delta_\ell\leq 1-\delta$)
\begin{align}
    \eqref{eq-Adv-transform-II} \leq \sum_{ |E(S)|\leq D } \left( \frac{1-\delta}{n} \right)^{|E(S)|} \sum_{ \substack{ (H_1,\ldots,H_L):\sum_\ell |E(H_\ell)|\leq D \\ \phi(H_1,\ldots,H_L)=S } } \Lambda_S(H_1,\ldots,H_L) \,,   \label{eq-Adv-transform-III}
\end{align} 
where 
\begin{equation}{\label{eq-def-Lambda-S(H-ell)}}
    \Lambda_S(H_1,\ldots,H_L)= \prod_{1 \leq \ell \leq L} \left( \frac{\Delta_\ell}{n} \right)^{|E(H_\ell)\setminus E(S)|} \cdot \mathbb E\left[ \prod_{1 \leq \ell \leq L} \prod_{(i,j)\in E(H_\ell)} \omega(\sigma_\ell(i),\sigma_\ell(j)) \right]^2 \,.
\end{equation}
\begin{lemma}{\label{lem-bound-given-S-K}}
    Fix any $S \subset K$, we have
    \begin{align*}
        & \sum_{ \substack{ (H_1,\ldots,H_L):\sum_\ell |E(H_\ell)|\leq D \\ \phi(H_1,\ldots,H_L)=S \\ \psi(H_1,\ldots,H_L)=K } }  \Lambda_S(H_1,\ldots,H_L) \\
        \leq\ & \mathbf 1_{ \{ \mathsf L(K)\subset V(S) \} } \cdot n^{-|E(S)|-|E(K) \setminus E(S)|} \cdot (qL^2)^{10(\tau(K)-\tau(S))+2|\mathfrak C(K;S)|+2|E(S)|} \cdot (1-\delta)^{|E(K)|-|E(S)|} \,.
    \end{align*}
\end{lemma}
\begin{proof}
    Note that $\Lambda_S(H_1,\ldots,H_L)=0$ if $\mathsf L(K) \not\subset V(S)$. Thus, we only need to consider the case $\mathsf L(K) \subset V(S)$. Using Lemma~\ref{cor-revised-decomposition-H-Subset-S} we can write 
    \begin{align*}
        K \setminus S = P_{\mathtt 1} \sqcup \ldots \sqcup P_{\mathtt t} \sqcup C_{\mathtt 1} \sqcup \ldots \sqcup C_{\mathtt m}
    \end{align*}
    such that Items~(1), (2) in Lemma~\ref{cor-revised-decomposition-H-Subset-S} are satisfied (note here the edges in $K \setminus S$ are decorated with a number in $[L]$). In particular, we have $\mathtt t \leq 5(\tau(K)-\tau(S))$. For $1 \leq \mathtt t\leq \mathtt T$ and $1 \leq \mathtt m \leq \mathtt M$, define
    \begin{align*}
        &\kappa_{\mathtt t}: E(P_{\mathtt t}) \longrightarrow [L], \kappa((i,j))=\ell \mbox{ if and only if } (i,j) \in E(H_\ell) \,; \\
        &\gamma_{\mathtt m}: E(C_{\mathtt m}) \longrightarrow [L], \kappa((i,j))=\ell \mbox{ if and only if } (i,j) \in E(H_\ell) \,; 
    \end{align*}
    It is clear that given $\psi(H_1,\ldots,H_L)=K$ and $\phi(H_1,\ldots,H_L)=S$, $(H_1,\ldots,H_\ell)$ is determined by $\{ \kappa_{\mathtt t} \}$ and $\{ \gamma_{\mathtt m} \}$. Write $(i,j;\ell) \in (P_{\mathtt t},\gamma_{\mathtt t})$ if $(i,j) \in E(P_{\mathtt t})$ and $\gamma_{\mathtt t}(i,j)=\ell$. By letting $\mathtt V=\cup \mathsf{EndP}(P_{\mathtt t})$ and by conditioning on $\{ \sigma_\ell(i): 1 \leq \ell \leq L, v\in \mathtt V \}$ we have
    \begin{align}
        &\mathbb E\left[ \prod_{1 \leq \ell \leq L} \prod_{(i,j)\in E(H_\ell)} \omega(\sigma_\ell(i),\sigma_\ell(j)) \right] \nonumber \\
        =\ & \mathbb E\left\{ \mathbb E\left[ \prod_{1 \leq \ell \leq L} \prod_{(i,j)\in E(H_\ell)} \omega(\sigma_\ell(i),\sigma_\ell(j)) \mid \{ \sigma_\ell(i): 1 \leq \ell \leq L, v\in \mathtt V \} \right] \right\} \nonumber \\
        =\ & \mathbb E\left\{ \prod_{(i,j;\ell)\in E(S)} \omega(\sigma_\ell(i),\sigma_\ell(j)) \prod_{1 \leq \mathtt t \leq \mathtt T} F(P_{\mathtt t}) \prod_{1 \leq \mathtt m \leq \mathtt M} F(C_{\mathtt m}) \right\} \,,  \label{eq-cal-expectation}
    \end{align}
    where for a path $P$ or a cycle $C$
    \begin{align*}
        F(P) &= \mathbb E\left[ \prod_{(i,j;\ell)\in E(P)} \omega(\sigma_\ell(i),\sigma_\ell(j)) \mid \{ \sigma_\ell(i): 1 \leq \ell \leq L, i\in \mathsf{EndP}(P_{\mathtt t}) \} \right] \,; \\
        F(C) &= \mathbb E\left[ \prod_{(i,j;\ell)\in E(C)} \omega(\sigma_\ell(i),\sigma_\ell(j)) \right] \,.
    \end{align*}
    Using Lemma~\ref{lem-exp-over-chain}, we have 
    \begin{align*}
        F(P) &= \rho^{2|\mathsf{dif}(P)|} \omega(\sigma_\ell(i),\sigma_{\ell'}(j)) \mbox{ for } \mathsf{EndP}(P)=\{ i,j \} \mbox{ with } i \in V_\ell(P), j \in V_{\ell'}(P) \,; \\
        F(C) &= \rho^{2|\mathsf{dif}(C)|} (q-1) \,.
    \end{align*}
    Thus, we have
    \begin{align*}
        |\eqref{eq-cal-expectation}| &= \left| \mathbb E\left\{ \prod_{(i,j;\ell)\in E(S)} \prod_{1 \leq \mathtt t \leq \mathtt T} \rho^{2|\mathsf{dif}(P)|} \omega(\mathsf{EndP}(P))  \prod_{1 \leq \mathtt m \leq \mathtt M} \rho^{2|\mathsf{dif}(P)|} (q-1) \right\} \right| \\
        &\leq \prod_{1 \leq \mathtt t \leq \mathtt T} \rho^{2|\mathsf{dif}(P)|} \prod_{1 \leq \mathtt m \leq \mathtt M} \rho^{2|\mathsf{dif}(C)|} \cdot q^{\mathtt M+\mathtt T+|E(S)|} \\
        &\leq q^{|\mathfrak C(K,S)|+5(\tau(K)-\tau(S))+|E(S)|} \prod_{1 \leq \mathtt t \leq \mathtt T} \rho^{2|\mathsf{dif}(\kappa_{\mathtt t})|} \prod_{1 \leq \mathtt m \leq \mathtt M} \rho^{2|\mathsf{dif}(\gamma_{\mathtt m})|}  \,.
    \end{align*}
    Thus, it suffices to show that
    \begin{align}
        &\sum_{ \kappa_{\mathtt t} } \rho^{2|\mathsf{dif}(\kappa_{\mathtt t})|} \prod_{1 \leq \ell \leq L} \Delta_\ell^{ |\{ (i,j): \kappa_{\mathtt t}(i,j)=\ell \}| } \leq L^2 (1-\delta)^{|E(P_{\mathtt t})|} \,;  \label{eq-goal-sum-chain-1} \\
        &\sum_{ \gamma_{\mathtt m} } \rho^{2|\mathsf{dif}(\gamma_{\mathtt m})|} \prod_{1 \leq \ell \leq L} \Delta_\ell^{ |\{ (i,j): \kappa_{\mathtt m}(i,j)=\ell \}| } \leq L^2 (1-\delta)^{|E(C_{\mathtt m})|} \,.  \label{eq-goal-sum-chain-2}
    \end{align}
    We will only show \eqref{eq-goal-sum-chain-1} since the proof of \eqref{eq-goal-sum-chain-2} is almost identical. Assume $|E(P_{\mathtt t})|=\aleph$ and define $\mathcal W=\mathcal W(\aleph)$ to be the set of $\omega=(\omega_1,\ldots,\omega_{\aleph}) \in \{ 1,\ldots,L \}^{\aleph}$. In addition, for $\omega=(\omega_1,\ldots,\omega_{\aleph})\in\mathcal W$, we write
    \begin{align*}
        & E_{\ell}(\omega)=\big\{ 1\leq i\leq \aleph: \omega_i=\ell \big\} \mbox{ for } 1 \leq \ell \leq L \,; \\
        & \mathsf{dif}(\omega)=\big\{ 1 \leq i \leq \aleph-1: \{ \omega_i,\omega_{i+1} \}=\{ \ell,\ell' \} \mbox{ for some } 1 \leq \ell<\ell' \leq L \big\} \,.
    \end{align*}
    It is clear that
    \begin{align*}
        \eqref{eq-goal-sum-chain-1} = \sum_{ \omega\in\mathcal W } \rho^{4|\mathsf{dif}(\omega)|} \prod_{1 \leq \ell \leq L} \Delta_\ell^{|E_\ell(\omega)|}  \,.
    \end{align*}
    Using \cite[Section~B.1]{GHL26+}, we have that
    \begin{align*}
        \eqref{eq-goal-sum-chain-1} = \mathsf X_1(\aleph) + \ldots + \mathsf X_L(\aleph) \,, \mbox{ where } 
        \begin{pmatrix}
            \mathsf X_1(t+1) \\ \ldots \\ \mathsf X_L(t+1)
        \end{pmatrix}
        = \mathbf P 
        \begin{pmatrix}
            \mathsf X_1(t) \\ \ldots \\ \mathsf X_L(t)
        \end{pmatrix} \,.
    \end{align*}
    Here  
    \begin{equation}{\label{eq-def-mathbf-P}}
        \mathbf P = 
        \begin{pmatrix}
            \Delta_1 & 0 & \ldots & 0 \\
            0 & \Delta_2 & \ldots & 0 \\
            \vdots & \vdots & \ddots & \vdots  \\
            0 & 0 & \ldots & \Delta_L
        \end{pmatrix}
        \begin{pmatrix}
            1 & \rho^4 & \ldots & \rho^4 \\
            \rho^4 & 1 & \ldots & \rho^4 \\
            \vdots & \vdots & \ddots & \vdots \\
            \rho^4 & \rho^4 & \ldots & 1 
        \end{pmatrix} \,.
    \end{equation}
    Define $\sigma_+(\mathbf P)$ to be the largest eigenvalue of the matrix $\mathbf P$. It was shown in \cite[Section~B.1]{GHL26+} that $\sigma_+(\mathbf P)<1-\Omega(1)$ if \eqref{eq-condition-multilayer-SBM-transform} holds. Thus, it is straightforward to see that
    \begin{align*}
        \sum_{\ell=1}^{L} \mathsf X_\ell(\aleph)^2 \leq  \sum_{\ell=1}^{L} \Delta_\ell^{-1} \mathsf X_\ell(\aleph)^2 \leq \sigma_+(\mathbf P)^{2\aleph} \sum_{\ell=1}^{L} \Delta_\ell^{-1} \mathsf X_\ell(0)^2
        \leq (1-\delta)^{2\aleph} \sum_{\ell=1}^{L} O(L^2) \,,
    \end{align*}
    where the last inequality follows from \eqref{eq-condition-multilayer-SBM-technical} and $\mathsf X_\ell(0)=O(1)$. This yields \eqref{eq-goal-sum-chain-1}, thus finishing the proof of the lemma.
\end{proof}
To this end, using Lemma~\ref{lem-bound-given-S-K}, we have \eqref{eq-Adv-transform-III} equals
\begin{align}
    & \sum_{|E(S)|\leq D} \left( \frac{q^2L^4}{n} \right)^{|E(S)|} \sum_{ \substack{ K:S \subset K \\ \mathsf L(K)\subset V(S) } } n^{-|E(K) \setminus E(S)|} (qL^2)^{10(\tau(K)-\tau(S))+2|\mathfrak C(K;S)|} \cdot (1-\delta)^{|E(K) \setminus E(S)|} \nonumber \\
    =\ & \sum_{|E(S)|\leq D} \left( \frac{q^2L^4}{n} \right)^{|E(S)|} \sum_{ \{ p_m \},a,b \geq 0 } \left(\frac{1-\delta}{n}\right) ^{-(a+b+\sum_{m \geq 3}mp_m)} (qL^2)^{10a+2\sum_{m\geq 3}p_m} \mathsf{Enum}(a,b,\{ p_m \}) \,, \label{eq-Adv-transform-IV}
\end{align}
where 
\begin{align*}
    \mathsf{Enum}(a,b,\{ p_m \})= \#\Bigg\{ & K: S \subset K, \mathsf{L}(K) \in V(S), \mathfrak C_m(K,S)=p_m, \tau(K)-\tau(S)=a,  \\
    & |E(K)|-|E(S)|=a+b+\sum_m mp_m \Bigg\} \,.
\end{align*}
We can now finish the proof of Lemma~\ref{lem-bound-adv-multilayer-SBM}.
\begin{proof}[Proof of Lemma~\ref{lem-bound-adv-multilayer-SBM}]
    Using Lemma~\ref{lem-enu-Subset-large-graph}, we have
    \begin{align*}
        \mathsf{Enum}(a,b,\{ p_m \}) \leq (2D)^{4a} n^{b+\sum_{m \geq 3}mp_m} \prod_{m=3}^{D} \frac{1}{p_m!} \,.
    \end{align*}
    Thus, we have
    \begin{align}
        \eqref{eq-Adv-transform-IV} &\leq \sum_{|E(S)|\leq D} \left( \frac{q^2L^4}{n} \right)^{|E(S)|} \sum_{ \{ p_m \},a,b \geq 0 } n^{-a} (q^{10}D)^{a} (1-\delta)^{\sum mp_m+b} (q^2L^4)^{\sum p_m} \prod_{m \geq 3} \frac{1}{p_m!} \nonumber \\
        &\leq O_{\delta}(1) \cdot \sum_{|E(S)|\leq D} \left( \frac{q^2L^4}{n} \right)^{|E(S)|} \prod_{m \geq 3} \sum_{ p_m } \frac{ (1-\delta)^{mp_m} (q^2L^4)^{p_m} }{ p_m! } \nonumber \\
        &\leq O_{\delta}(1) \cdot \sum_{|E(S)|\leq D} \left( \frac{q^2L^4}{n} \right)^{|E(S)|} \prod_{m \geq 3} \exp\left( (1-\delta)^m q^2 L^4 \right) \nonumber \\
        &\leq \exp\left( O(1) \cdot q^2 \right) \sum_{|E(S)|\leq D} \left( \frac{q^2L^4}{n} \right)^{|E(S)|} \,.  \label{eq-Adv-transform-V}
    \end{align}
    Denote $\mathcal H$ be the set of multigraphs in $\mathsf K_n$ such that each edge has multiplicity at least $2$, it suffices to show that
    \begin{align}
        \sum_{ S \subset \mathcal H:|E(S)|\leq D } \left( \frac{q^2L^4}{n} \right)^{|E(S)|} \leq \exp\left( O(q^2L^4) \right) \mbox{ for all } D=n^{o(1)} \,.  \label{eq-Adv-transform-V-goal}
    \end{align}
    Denote $S$ has $m$ independent edges and the other parts of $S$ can be decomposed into connected components $S_1,\ldots,S_t$ with $|E(\widetilde{S}_i)|\geq 2$ and $|V(S_i)|\leq |E(\widetilde{S}_i|+1$. It is clear that the left hand side of \eqref{eq-Adv-transform-V-goal} is bounded by
    \begin{align}
        &\sum_{m} \frac{ n^{2m} }{ m! } \cdot \left( \frac{q^2L^4}{n} \right)^{2m} \sum_{t\geq 0} \sum_{S_1,\ldots,S_t} \prod_{1 \leq i \leq t} \left( \frac{q^2L^4}{n} \right)^{|E(S_i)|} \nonumber \\
        \leq\ & \sum_{m} \frac{ n^{2m} }{ m! } \cdot \left( \frac{q^2L^4}{n} \right)^{2m} \sum_{0 \leq t\leq D} \prod_{1 \leq i \leq t} \sum_{|E(S_i)|\leq D} \left( \frac{q^2L^4}{n} \right)^{|E(S_i)|} \,. \label{eq-Adv-transform-V-goal-I}
    \end{align}
    Denote $\overline{S}_i \subset S_i$ be the multigraph such that each edge in $S_i$ has multiplicity exactly $2$ in $\overline{S}_i$. It is clear that given $\overline{S}_i$ with $|E(\overline{S}_i)|\leq D$, the possible enumeration of $S_i$ is bounded by $(2D)^{|E(S_i)|-|E(\overline{S}_i)|}$. Thus, we have
    \begin{align}
        \eqref{eq-Adv-transform-V-goal-I} &\leq \sum_{m} \frac{ n^{2m} }{ m! } \cdot \left( \frac{q^2L^4}{n} \right)^{2m} \sum_{0 \leq t\leq D} \prod_{1 \leq i \leq t} \sum_{|E(\overline{S}_i)|\leq D} \left( \frac{q^2L^4}{n} \right)^{|E(\overline{S}_i)|} \sum_{a_i \geq 0} \left( \frac{2Dq^2}{n} \right)^{a_i} \nonumber \\ 
        &\leq [1+o(1)] \sum_{m} \frac{ n^{2m} }{ m! } \cdot \left( \frac{q^2L^4}{n} \right)^{2m} \sum_{0 \leq t\leq D} \prod_{1 \leq i \leq t} \sum_{|E(\overline{S}_i)|\leq D} \left( \frac{q^2L^4}{n} \right)^{|E(\overline{S}_i)|} \,.  \label{eq-Adv-transform-V-goal-II}
    \end{align}
    Note that $|E(\widetilde{S}_i)|\geq 2$ and $|E(\widetilde{S}_i)|\geq |V(\widetilde{S}_i)|-1$, we have 
    \begin{align*}
        \sum_{ 2\leq |E(S)|\leq D } \left( \frac{q^2L^4}{n} \right)^{2|E(S)|} &= \sum_{ 2 \leq t\leq D } \sum_{ m \leq t+1 } \left( \frac{q^2L^4}{n} \right)^{2t} \#\left\{ S \subset \mathsf K_n: |E(S)|=t, |V(S)|=m \right\} \\
        &\leq \sum_{ 2 \leq t\leq D } \sum_{ m \leq t+1 } \left( \frac{q^2L^4}{n} \right)^{2t} \binom{n}{m} \binom{m^2}{t} \\
        &\leq \sum_{ 2 \leq t\leq D } \sum_{ m \leq t+1 } \left( \frac{q^2L^4 D^2}{n} \right)^{2t} n^{m} \\
        &\leq \sum_{ 2 \leq t\leq D } \sum_{ m \leq t+1 } n^{ -(2-o(1))t+m } = n^{-1+o(1)} \,.
    \end{align*}
    Thus, we have
    \begin{align*}
        \eqref{eq-Adv-transform-V-goal-II} &\leq \sum_{m} \frac{ n^{2m} }{ m! } \cdot \left( \frac{q^2L^4}{n} \right)^{2m} \sum_{0 \leq t\leq D} n^{(-1+o(1))t} \\
        &\leq [1+o(1)] \sum_{m} \frac{ n^{2m} }{ m! } \cdot \left( \frac{q^2L^4}{n} \right)^{2m} \leq \exp(O(q^2L^4)) \,, 
    \end{align*}
    finishing the proof of Lemma~\ref{lem-bound-adv-multilayer-SBM}.
\end{proof}

\subsection{Putting it together}{\label{subsec:conclusion-multilayer-SBM}}

We can now finish the proof of Theorem~\ref{Main-thm-multilayer-SBM}.
\begin{proof}[Proof of Theorem~\ref{Main-thm-multilayer-SBM}]
    Suppose on the contrary that there exists an algorithm $\mathcal A$ with running time $\exp(n^{o(1)})$ that takes $\bm Y\sim \Pb_{\rho,q,\{d_\ell\},\{\lambda_\ell\},n}$ as input and achieves weak recovery in the sense of Definition~\ref{def-recovery-Bernoulli-obser-model}. Using Lemma~\ref{lem-imbalanced-detection-multilayer-SBM}, we see that there exists an $(T_n;c_n;\epsilon_n)$-test between $\Pb_{\rho,q,\{d_\ell'\},\{\lambda_\ell\},n}$ and $\Qb_{\{d_\ell\},n}$ such that
    \begin{align*}
        T_n = n^{o(1)}, \quad c_n=\Omega(1), \quad \epsilon_n = \exp\left( -\Omega(n) \right) \,.
    \end{align*}
    However, using Lemma~\ref{lem-bound-adv-multilayer-SBM} (by putting $D_n=q_n^2 L_n\log n+T_n=n^{o(1)}$) and Proposition~\ref{thm-alg-contiguity}, we see that there is no $(T_n';c_n';\epsilon_n')$-test between $\Pb_{\rho,q,\{d_\ell'\},\{\lambda_\ell\},n}$ and $\Qb_{\{d_\ell\},n}$ such that 
    \begin{equation*}
        T_n' = n^{o(1)}, \quad c_n'=\Omega(1), \quad \epsilon_n' = \exp\left( -\omega(q^2 L^4) \right) \,.
    \end{equation*}
    This forms a contradiction and thus yields Theorem~\ref{Main-thm-multilayer-SBM} (recall that $q,L=n^{o(1)}$).
\end{proof}

\appendix

\section{Useful toolkit}{\label{sec:useful-toolkit}}

\subsection{Additive Gaussian model}{\label{subsec:fact-add-Gaussian-model}}

Throughout this subsection, we will work on $\bm Y=\Theta+\bm W$ sampled according to Definition~\ref{def-add-Gaussian-model}. Fix a small constant $\kappa>0$ and recall that we let $\Qb$ be the law of $\bm W$, a random vector in $\mathbb R^N$ with i.i.d.\ standard normal entries. We will introduce external randomness $\bm Z \sim \Qb$ independent of $\bm Y$ and set
\begin{align}{\label{eq-split-add-Gaussian-model}}
    \bm A = \frac{ 1 }{ \sqrt{1+\kappa^2} } \left( \bm Y+ \kappa \bm Z \right), \quad \bm B = \frac{ 1 }{ \sqrt{1+\kappa^{-2}} } \left( \bm Y - \kappa^{-1} \bm Z \right) \,.  
\end{align}
As discussed in Section~\ref{subsubsec:reduction-to-detection}, the reasons for the definition in \eqref{eq-split-add-Gaussian-model} are as follows:
\begin{itemize}
    \item Under $\bm Y \sim \Qb$, we have 
    \begin{align}
        \bm A = \frac{ 1 }{ \sqrt{1+\kappa^2} } \left( \bm W + \kappa \bm Z \right), \quad \bm B = \frac{ 1 }{ \sqrt{1+\kappa^{-2}} } \left( \bm W - \kappa^{-1} \bm Z \right) \,.  \label{eq-behavior-A-B-Qb}
    \end{align}
    In particular, $(\bm A,\bm B)$ are vectors with i.i.d.\ standard normal entries and $\bm B$ is \underline{independent of} $\bm A$.
    \item Under $\bm Y \sim\Pb$, we have 
    \begin{align}
        \bm A = \frac{ 1 }{ \sqrt{1+\kappa^2} } \Theta + \overline{\bm A}, \quad \bm B = \frac{ 1 }{ \sqrt{(1+\kappa^{-2})} } \Theta + \overline{\bm B} \,,  \label{eq-behavior-A-B-Pb}
    \end{align}
    where 
    \begin{align*}
        \overline{\bm A} = \frac{ 1 }{ \sqrt{1+\kappa^2} } \left( \bm W + \kappa \bm Z \right), \quad \overline{\bm B} = \frac{ 1 }{ \sqrt{1+\kappa^{-2}} } \left( \bm W - \kappa^{-1} \bm Z \right) \,.
    \end{align*}
    In particular, $(\overline{\bm A}, \overline{\bm B})$ are vectors with i.i.d.\ standard normal entries and $\overline{\bm B}$ is \underline{independent of} $\bm A$. Also we have $\bm A$ is an additive Gaussian model (with ``signal term'' $\frac{ 1 }{ \sqrt{1+\kappa^2} } \Theta$).
\end{itemize}

\begin{lemma}{\label{lem-reduce-to-detection-add-Gaussian-model}}
    Suppose an estimator $\mathcal X=\mathcal X(\bm A)$ satisfies that
    \begin{equation}{\label{eq-estimator-add-Gaussian-model}}
        \mathbb E_{\Pb}\left[ \frac{ \langle \mathcal X,\Theta \rangle }{ \| \mathcal X \| \| \Theta \| } \right] = c \mbox{ for some constant } c>0 \,.
    \end{equation}
    In addition, suppose that 
    \begin{equation}{\label{eq-concentration-signal-term}}
        \Pb\left( \frac{M}{2} \leq \|\Theta\|\leq 2M \right)=1-o(1) \mbox{ for some } M=M_N=\omega_N(1) \,.
    \end{equation}
    Then
    \begin{align}
        &\Pb\left( \langle \mathcal X(\bm A),\bm B \rangle \geq \frac{c}{8\sqrt{1+\kappa^{-2}}} \cdot M \| \mathcal X \| \right) \geq \Omega(1) \,; \label{eq-behavior-Pb-add-Gaussian-model} \\
        &\Qb\left( \langle \mathcal X(\bm A),\bm B \rangle \geq \frac{c}{8\sqrt{1+\kappa^{-2}}} \cdot M \| \mathcal X \| \right) \leq \exp(-\Omega(1)\cdot M^2) \,. \label{eq-behavior-Qb-add-Gaussian-model}
    \end{align}
\end{lemma}
\begin{proof}
    We first prove \eqref{eq-behavior-Qb-add-Gaussian-model}. Recall \eqref{eq-behavior-A-B-Qb}, under $\Qb$ we have $\bm B$ is a standard Wigner matrix independent of $\bm A$ (and thus also independent of $\mathcal X$). Thus, conditioned on $\mathcal X$ we have
    \begin{align*}
        \big\langle \mathcal X, \bm B \big\rangle \sim \mathcal N\left( 0,\| \mathcal X \|^2 \right) \,.
    \end{align*}
    Thus, from a simple Gaussian tail inequality we see that \eqref{eq-behavior-Qb} holds.
    
    Then we prove \eqref{eq-behavior-Pb-add-Gaussian-model}. Recall \eqref{eq-behavior-A-B-Pb}, we can decompose $\langle \mathcal X, \bm B \rangle$ into the following terms:
    \begin{align*}
        \big\langle \mathcal X, \bm B \big\rangle &= \frac{ 1 }{ \sqrt{(1+\kappa^{-2})} } \cdot \big\langle \mathcal X, \Theta \big\rangle + \big\langle \mathcal X, \overline{\bm B} \big\rangle \,.  
    \end{align*}
    Using \eqref{eq-estimator-add-Gaussian-model}, we see that   
    \begin{align}
        & \Pb\left( \frac{ 1 }{ \sqrt{(1+\kappa^{-2})} } \cdot \big\langle \mathcal X, \Theta \big\rangle \geq \frac{c}{4\sqrt{1+\kappa^{-2}}} \cdot M \| \mathcal X \| \right) \nonumber \\
        \ge \ & \Pb\left( \langle \mathcal X, \Theta \rangle \geq \frac{c}{2} \cdot \| \Theta \| \| \mathcal X \| \right) - \Pb\left( \| \Theta \| \leq \frac{M}{2} \right) \geq \frac{c}{2}-o(1) \,.  \label{eq-behavior-Pb-add-Gaussian-model-part-1-bound}
    \end{align}
    where the last inequality follows from \eqref{eq-estimator-add-Gaussian-model} and \eqref{eq-concentration-signal-term}. In addition, from the proof of \eqref{eq-behavior-Qb-add-Gaussian-model}, we see that
    \begin{align}
        \Pb\left( \langle \mathcal X, \overline{\bm B} \rangle \geq \frac{c}{4\sqrt{1+\kappa^{-2}}} \cdot M \| \mathcal X \|_{\Fop} \right) \leq e^{-\Theta(1) \cdot M^2 }=o(1) \,.   \label{eq-behavior-Pb-add-Gaussian-model-part-2-bound}
    \end{align}
    Combining \eqref{eq-behavior-Pb-add-Gaussian-model-part-1-bound} and \eqref{eq-behavior-Pb-add-Gaussian-model-part-2-bound}, we have that 
    \begin{align*}
        &\Pb\left( \big\langle \mathcal X, \bm B \big\rangle \geq \frac{c}{8\sqrt{1+\kappa^{-2}}} \cdot M \| \mathcal X \|_{\Fop} \right) \\
        \ge\ & \Pb\left( \frac{ 1 }{ \sqrt{(1+\kappa^{-2})} } \cdot \langle \mathcal X, \Theta \rangle \geq \frac{c}{4\sqrt{1+\kappa^{-2}}} \cdot M \| \mathcal X \| \right) \\
        &- \Pb\left( |\langle \mathcal X, \overline{\bm B} \rangle| < \frac{c}{8\sqrt{1+\kappa^{-2}}} \cdot M \| \mathcal X \| \right) \\
        \geq\ & \frac{c}{2}-o(1)-e^{-\Theta(1) \cdot M^2 } = \frac{c}{2}-o(1) \,. 
    \end{align*}
    We then see that \eqref{eq-behavior-Pb-add-Gaussian-model} holds.
\end{proof}

\subsection{Correlation preserving projection}{\label{subsec:cor-pre-proj}}

Before we specialize at the general binary observation model, we introduce a correlation preserving projection process introduced in \cite{HS17, DHSS25}. This processing step enables us only to consider the ``sufficiently nice'' estimators.

\begin{lemma}[Theorem~2.3 in \cite{HS17}]{\label{lem-proj-to-convex-set}}
    Let $c>0$ be a constant and let $\mathcal C$ be a convex set in $\mathbb R^N$ with $\mathcal Y \in \mathcal C$. Let $P$ be a vector with $\langle P,\mathcal Y \rangle\geq c\|P\|\|\mathcal Y\|$. Then, if we let $Q$ be the vector that minimize $\|Q\|$ subject to $Q\in\mathcal C$ and $\langle P,Q \rangle \geq c\|P\|\|\mathcal Y\|$, we have
    \begin{align*}
        \langle Q,\mathcal Y \rangle\geq \frac{c}{2}\cdot \|Q\|\|\mathcal Y\| \,.
    \end{align*}
    Furthermore, $Q$ satisfies $\|Q\|\geq c\|\mathcal Y\|$.
\end{lemma}
Using Lemma~\ref{lem-proj-to-convex-set}, we prove that we can project the estimator into the set of matrices with bounded entries and bounded nuclear norm, while preserving correlation.
\begin{lemma}[Lemma~A.1 in \cite{DHSS25}]{\label{lem-cor-pre-proj}}
    For any $K,\tau>0$, let 
    \begin{equation}{\label{eq-def-mathcal-K}}
        \mathcal K=\left\{ \mathcal Y \in \mathbb R^{n*n}: \| \mathcal Y \|_{\infty} \leq \tau,\ \mathcal Y+\frac{1}{K} \mathbbm{1} \mathbbm{1}^{\top} \succ \mathbb O \right\}  \,.
    \end{equation}
    Suppose that $\bm Y \in \mathbb R^{n*n}$ be a general binary observation model with signal term $\Theta$ satisfying 
    \begin{equation}{\label{eq-regularity-Theta}}
        \Pb\left( \frac{M}{2} \leq \|\Theta\|_{\Fop} \leq 2M, \Theta\in\mathcal K \right) = 1-o(1) \,.
    \end{equation}
    Given any matrix $\mathcal X=\mathcal X(\bm Y)$ such that 
    \begin{equation}{\label{eq-expectation-inner-product}}
        \mathbb E_{\Pb}\left[ \frac{ \langle \Theta,\mathcal X(\bm Y) \rangle }{ \|\Theta\|_{\Fop} \| \mathcal X(\bm Y) \|_{\Fop} } \right] \geq c \,.
    \end{equation}
    There is a polynomial-time algorithm which outputs $\widehat{\mathcal X}(\bm Y) \in \mathcal K$ and $\| \widehat{\mathcal X} \|_{\Fop} \geq \frac{c}{4} \cdot M$ such that  
    \begin{equation}{\label{inner-product-non-vanishing}}
        \Pb\left( \frac{ \langle \Theta,\widehat{\mathcal X}(\bm Y) \rangle }{ \|\Theta\|_{\Fop} \| \widehat{\mathcal X}(\bm Y) \|_{\Fop} } \geq \frac{c}{4} \right) \geq \frac{c}{2} \,.
    \end{equation}
\end{lemma}
\begin{proof}
    We apply the correlation preserving projection from Lemma~\ref{lem-proj-to-convex-set}. By definition, $\mathcal K$ is a convex set containing $\Theta$. Let $\widehat{\mathcal X}$ be the matrix that minimizes $\| \widehat{\mathcal X} \|_{\Fop}$ subject to $\widehat{\mathcal X} \in \mathcal K$ and $\langle \widehat{\mathcal X}, \mathcal X \rangle \geq \frac{c}{2} \cdot M\| \mathcal X \|$. From Cauchy-Schwartz inequality we then have $\| \widehat{\mathcal X} \|_{\Fop} \geq M$. In addition, this semidefinite program can be solved in polynomial time using ellipsoid method. Finally, since a standard Markov inequality yields that
    \begin{align*}
        \Pb\left( \frac{ \langle \Theta,\widehat{\mathcal X}(\bm Y) \rangle }{ \|\Theta\|_{\Fop} \| \widehat{\mathcal X}(\bm Y) \|_{\Fop} } \geq \frac{c}{2} \right) \geq \frac{c}{2} \,.
    \end{align*}
    From Lemma~\ref{lem-proj-to-convex-set}, we then have
    \begin{equation*}
        \Pb\left( \frac{ \langle \Theta,\widehat{\mathcal X}(\bm Y) \rangle }{ \|\Theta\|_{\Fop} \| \widehat{\mathcal X}(\bm Y) \|_{\Fop} } \geq \frac{c}{4} \right) \geq \Pb\left( \frac{ \langle \Theta,\widehat{\mathcal X}(\bm Y) \rangle }{ \|\Theta\|_{\Fop} \| \widehat{\mathcal X}(\bm Y) \|_{\Fop} } \geq \frac{c}{2} \right) \geq \frac{c}{2} \,.  \qedhere
    \end{equation*}
\end{proof}

\subsection{Binary observation model}{\label{subsec:fact-Ber-obser-model}}

Throughout this subsection, we will work on $\bm Y$ sampled according to Definition~\ref{def-Berboulli-obser-model}. We first need a basic fact on Bernoulli random variables.
\begin{lemma}{\label{lem-split-binom}}
    Fix $p>0$ and $a,b>0$. Let $(\xi,\xi')$ be a pair of Bernoulli variables such that
    \begin{equation}{\label{eq-def-splitting-binom}}
        \begin{aligned}
            &\Pb(\xi=\xi'=1)=pab \,; \\ 
            &\Pb(\xi=1,\xi'=0)=a-pab \,; \\
            &\Pb(\xi=0,\xi'=1)=b-pab \,; \\
            &\Pb(\xi=\xi'=0)=1-a-b+pab \,.
        \end{aligned}
    \end{equation}
    Then we have 
    \begin{enumerate}
        \item[(1)] For $Y \sim \mathsf{Ber}(p)$ independent of $(\xi,\xi')$, we have $(Y\xi,Y\xi')$ is a pair of independent Bernoulli variables with parameter $pa$ and $pb$, respectively.
        \item[(2)] For $Y \sim \mathsf{Ber}(q)$ independent of $(\xi,\xi')$, we have $(Y\xi,Y\xi')$ is a pair of Bernoulli variables with parameter $qa$ and $qb$, respectively; in addition, we have 
        \begin{equation}{\label{eq-cond-split-binom}}
            \mathbb E\left[ Y\xi' \mid Y\xi \right] = qb - \frac{(q-p)b}{1-qa} (Y\xi-qa) \,.
        \end{equation}
    \end{enumerate}
\end{lemma}
\begin{proof}
    From straightforward calculation, we get that for $Y \sim \mathsf{Ber}(q)$ independent of $(\xi,\xi')$, we have 
    \begin{equation}{\label{eq-law-splitting-binom}}
        \begin{aligned}
            &\Pb(Y\xi=Y\xi'=1)=pqab \,; \\ 
            &\Pb(Y\xi=1,Y\xi'=0)=qa-pqab \,; \\
            &\Pb(Y\xi=0,Y\xi'=1)=qb-pqab \,; \\
            &\Pb(Y\xi=Y\xi'=0)=1-qa-qb+pqab \,.
        \end{aligned}
    \end{equation}
    Thus, it is straightforward to verify that when $Y \sim \mathsf{Ber}(p)$ independent of $(\xi,\xi')$, we have $(Y\xi,Y\xi')$ is a pair of independent Bernoulli variables with parameter $pa$ and $pb$, respectively. In addition, for general $Y \sim \mathsf{Ber}(q)$ independent of $(\xi,\xi')$, we have $(Y\xi,Y\xi')$ is a pair of Bernoulli variables with parameter $qa$ and $qb$, and 
    \begin{align*}
        \mathbb E\left[ Y\xi' \mid Y\xi=1 \right] = \frac{pqab}{qa} = pb \,; \quad \mathbb E\left[ Y\xi' \mid Y\xi=0 \right] = \frac{ qb(1-pa) }{ 1-qa } \,.
    \end{align*}
    Thus, we have that \eqref{eq-cond-split-binom} holds.
\end{proof}
Now, given a general binary observation model $\bm Y \in \mathbb R^N$, we will introduce external randomness $\xi,\xi' \in \mathbb R^N$ such that $(\xi_i,\xi'_i)_{1 \leq i \leq N}$ are i.i.d.\ sampled from the law of \eqref{eq-def-splitting-binom} (and $(\xi,\xi')$ is independent of $\bm Y$). Define $\bm A,\bm B \in \mathbb R^N$ such that 
\begin{equation}{\label{eq-split-Ber-obser-model}}
    \bm A_i = \bm Y_i \xi_i, \quad \bm B_i = \bm Y_i \xi_i' \mbox{ for } 1 \leq i \leq N \,.
\end{equation}
The reasons for the definition in \eqref{eq-split-Ber-obser-model} are as follows:
\begin{itemize}
    \item Under $\bm Y \sim \Qb$ (i.e., when the entries of $\bm Y$ are i.i.d.\ $\mathsf{Ber}(p)$), from Lemma~\ref{lem-split-binom} we have
    \begin{align}{\label{behavior-A,B-Qb-Ber-obser-model}}
        \bm A \sim \mathsf{Ber}(pa)^{\otimes N}, \quad \bm B \sim \mathsf{Ber}(pb)^{\otimes N} \,.
    \end{align}
    In addition, $\bm B$ is \underline{independent of} $\bm A$.
    \item Given $\Theta\in\mathbb R^N$ and under $\bm Y \sim \Pb(\cdot\mid\Theta)$ (i.e., when $\{ \bm Y_i:1\leq i \leq N \}$ are independent Bernoulli variables with parameter $\{\Theta_i+p: 1 \leq i \leq N \}$), from Lemma~\ref{lem-split-binom} we have
    \begin{align}{\label{behavior-A,B-Pb-Ber-obser-model}}
        \mathbb E\left[ \bm B \mid \bm A \right] = b(p\mathbbm{1}+\Theta) - \overline{\bm A}, \mbox{ where } \overline{\bm A}_i = \frac{ b\Theta_i (\bm A_i-a(p+\Theta_i)) }{ 1-a(p+\Theta_i) } \,.
    \end{align}
    In addition, $\bm B-\mathbb E[\bm B \mid \bm A]$ is \underline{independent of} $\bm A$.
\end{itemize}
\begin{lemma}{\label{lem-reduce-to-detection-Ber-obser-model}}
    Fix $M=M_N=\omega_N(1),L=L_N=o_N(1),p=p_N=o_N(1)$ and a constant $c>0$. Consider a binary observation model $\bm Y$ with $\mathbb E[\bm Y_i\mid\Theta]=\Theta_i+p$ satisfying
    \begin{equation}{\label{eq-concentration-Theta}}
        \Pb\left( \|\Theta\|_{\infty} \leq L, \frac{M}{2} \leq \|\Theta\| \leq 2M \right)=1-o(1) \,.
    \end{equation}
    In addition, suppose an estimator $\mathcal X=\mathcal X(\bm A)$ satisfies that
    \begin{equation}{\label{eq-estimator-Ber-obser-model}}
        \|\mathcal X\| \geq \frac{cM}{2}, \quad \| \mathcal X \|_{\infty} \leq L, \quad \Pb\left( \frac{ \langle \mathcal X,\Theta \rangle }{ \| \mathcal X \| \| \Theta \| } \geq c \right) \geq c \,.
    \end{equation}
    Then we have
    \begin{align}
        &\Pb\left( \langle \mathcal X(\bm A),\bm B-pb \mathbbm{1}\mathbbm{1}^{\top} \rangle \geq \frac{c}{8} \cdot M \| \mathcal X \| \right) \geq \Omega(1) \,; \label{eq-behavior-Pb-Ber-obser-model} \\
        &\Qb\left( \langle \mathcal X(\bm A),\bm B-pb \mathbbm{1}\mathbbm{1}^{\top} \rangle \geq \frac{c}{8} \cdot M \| \mathcal X \| \right) \leq \exp\left( -\Omega(1)\cdot \frac{M^2}{p+L} \right) \,. \label{eq-behavior-Qb-Ber-obser-model}
    \end{align}
\end{lemma}
\begin{proof}
    We first prove \eqref{eq-behavior-Qb-Ber-obser-model}. Recall \eqref{behavior-A,B-Qb-Ber-obser-model}, under $\Qb$ we have $\bm B$ has i.i.d.\ standard $\mathsf{Ber}(pb)$ entries independent of $\bm A$ (and thus also independent of $\mathcal X$). Thus, conditioned on $\mathcal X$, we have
    \begin{align*}
        \langle \mathcal X,\bm B-pb \mathbbm{1}\mathbbm{1}^{\top} \rangle = \sum_{i=1}^{N} \mathcal X_i (\bm B_i-pb)
    \end{align*}
    is a sum of independent mean-zero random variables, with
    \begin{align*}
        & |\mathcal X_i(\bm B_i-pb)| \leq \|\mathcal X\|_{\infty} \leq L \,; \\
        & \sum_{i=1}^{N} \mathsf{var}(\mathcal X_i(\bm B_i-pb)) = pb(1-pb)\sum_{i=1}^{N} \mathcal X_i^2 =pb(1-pb)\|\mathcal X\|^2 \,.
    \end{align*}
    Thus, using standard Bernstein inequality (see, e.g., \cite[Theorem~2.8.4]{Vershynin18}) we have that given $\mathcal X$
    \begin{align*}
        &\Qb\left( \langle \mathcal X,\bm B-pb \mathbbm{1}\mathbbm{1}^{\top} \rangle \geq M\|\mathcal X\| \right) \\
        =\ & \Qb\left( \sum_{i=1}^{N} \mathcal X_i (\bm B_i-pb) \geq M\| \mathcal X \| \right) \\
        \leq\ & \exp\left( -\Omega(1) \cdot \frac{ M^2 \| \mathcal X \|^2 }{ pb(1-pb)\|\mathcal X\|^2+LM\|\mathcal X\| } \right) \\
        =\ & \exp\left( -\Omega(1) \cdot \frac{ M^2 }{ pb(1-pb)+LM/\|\mathcal X\| } \right) \leq \exp\left( -\Omega(1) \cdot \frac{M^2}{p+L} \right) \,,
    \end{align*}
    where the last inequality follows from \eqref{eq-concentration-Theta} and \eqref{eq-estimator-Ber-obser-model}. This yields \eqref{eq-behavior-Qb-Ber-obser-model}.

    Then we prove \eqref{eq-behavior-Pb-Ber-obser-model}. Recall \eqref{behavior-A,B-Pb-Ber-obser-model}, we can decompose $\langle \mathcal X, \bm B-pb\mathbbm{1}\mathbbm{1}^{\top} \rangle$ into the following terms:
    \begin{align*}
        \langle \mathcal X,\bm B-pb \mathbbm{1}\mathbbm{1}^{\top} \rangle = \langle \mathcal X,\bm B-\mathbb E[\bm B \mid \bm A] \rangle + b \langle \mathcal X,\Theta \rangle + \langle \mathcal X,\overline{\bm A} \rangle \,.
    \end{align*}
    Using \eqref{eq-estimator-Ber-obser-model} and \eqref{eq-concentration-Theta}, we have
    \begin{align}
        &\Pb\left( b \langle \mathcal X,\Theta \rangle \geq \frac{cb}{2} \cdot M\|\mathcal X\| \right) \nonumber \\
        \geq\ & \Pb\Big( \langle \mathcal X,\Theta \rangle \geq c \| \Theta \|\|\mathcal X\| \Big)- \Pb\left( \|\Theta\|\leq \frac{M}{2} \right) \geq c-o(1) \,.  \label{eq-behavior-Qb-Ber-obser-model-part-1}
    \end{align}
    In addition, conditioned on $\bm A$ (and thus $\mathcal X=\mathcal X(\bm A)$ is fixed) we have
    \begin{align*}
        \langle \mathcal X,\bm B-\mathbb E[\bm B \mid \bm A] \rangle = \sum_{i=1}^{N} \mathcal X_i \left( \bm B_i - \mathbb E[\bm B_i \mid \bm A] \right)
    \end{align*}
    is a sum of independent mean-zero random variables, with
    \begin{align*}
        \sum_{i=1}^{N} \mathsf{var}\left( \mathcal X_i(\bm B_i - \mathbb E[\bm B_i \mid \bm A]) \right) \leq \sum_{i=1}^{N} \mathcal X_i^2 = \| \mathcal X \|^2 \,.
    \end{align*}
    Thus, using Chebyshev's inequality we have (below we use $b,c>0$ are constants and $M=M_N=\omega_N(1)$)
    \begin{align}
        &\Pb\left( |\langle \mathcal X,\bm B-\mathbb E[\bm B \mid \bm A] \rangle| \geq \frac{cbM\|\mathcal X\|}{8} \right) \nonumber \\
        \leq\ & \frac{64}{c^2b^2M^2 \| \mathcal X \|^2} \cdot \sum_{i=1}^{N} \mathsf{var}\left( \mathcal X_i(\bm B_i - \mathbb E[\bm B_i \mid \bm A]) \right) = o(1) \,.  \label{eq-behavior-Qb-Ber-obser-model-part-2}
    \end{align}
    Finally, we have
    \begin{align}
        &\Pb\left( |\langle \mathcal X,\overline{\bm A} \rangle| \geq \frac{cbM\|\mathcal X\|}{8} \right) \leq \Pb\left( \| \overline{\bm A} \| \geq \frac{cbM}{8} \right) \nonumber \\
        \leq\ & \frac{64}{c^2b^2M^2} \cdot \mathbb E\left[ \|\overline{\bm A}\|^2 \right] \leq \frac{64}{c^2b^2M^2} \cdot \mathbb E\left[ \sum_{i=1}^N a b^2 \Theta_i^2 (p+\Theta_i) \right] \nonumber \\
        \leq\ & o(1) + \Omega(1) \cdot \frac{M^2(p+L)}{M^2} = o(1) \,,  \label{eq-behavior-Qb-Ber-obser-model-part-3}
    \end{align}
    where the first inequality follows from $|\langle \mathcal X,\overline{\bm A} \rangle| \leq \| \mathcal X \|\| \overline{\bm A} \|$, the second inequality follows from Chebyshev inequality, the third inequality follows from \eqref{behavior-A,B-Pb-Ber-obser-model}, the fourth inequality follows from \eqref{eq-concentration-Theta}, and the last equality follows from the assumption $p,L=o_N(1)$. Combining \eqref{eq-behavior-Qb-Ber-obser-model-part-1}--\eqref{eq-behavior-Qb-Ber-obser-model-part-3}, we have
    \begin{align*}
        &\Pb\left( \langle \mathcal X,\bm B-pb \mathbbm{1}\mathbbm{1}^{\top} \rangle \geq \frac{cbM\|\mathcal X\|}{4} \right) \\
        \ge\ & \Pb\left( b \langle \mathcal X,\Theta \rangle \geq \frac{cbM\|\mathcal X\|}{2} \right) \\
        &- \Pb\left( |\langle \mathcal X,\bm B-\mathbb E[\bm B \mid \bm A] \rangle| < \frac{cbM\|\mathcal X\|}{8} \right) - \Pb\left( |\langle \mathcal X,\overline{\bm A} \rangle| < \frac{cbM\|\mathcal X\|}{8} \right) \\
        \geq\ & (c-o(1))-o(1)-o(1) = c-o(1) \,. 
    \end{align*}
    We then see that \eqref{eq-behavior-Pb-add-Gaussian-model} holds.
\end{proof}

\subsection{Preliminaries on graphs}{\label{subsec:prelim-graphs}}

We first collect some results on the structural properties of graphs proved in \cite{CDGL24+}.
\begin{lemma}[Lemma~B.2 in \cite{CDGL24+}]{\label{lem-property-H-Subset-S}}
    For $H \subset S$ with $\mathsf{I}(S)=\emptyset$, we have $|\mathsf L(S) \setminus V(H)| \geq 2(\tau(H)-\tau(S))$. In particular, for $H \subset S$ such that $\mathsf{I}(S)=\emptyset$ and $\mathsf L(S) \subset V(H)$, we have $\tau(H) \leq \tau(S)$. 
\end{lemma}
\begin{lemma}[Lemma~B.3 in \cite{CDGL24+}]{\label{lem-decomposition-H-Subset-S}}
    For $H \subset S$, we can decompose $E(S)\setminus E(H)$ into $\mathtt m$ cycles ${C}_{\mathtt 1}, \ldots, {C}_{\mathtt m}$ and $\mathtt t$ paths ${P}_{\mathtt 1}, \ldots, {P}_{\mathtt t}$ for some $\mathtt m, \mathtt t\geq 0$ such that the following hold.
    \begin{enumerate}
        \item[(1)] ${C}_{\mathtt 1}, \ldots, {C}_{\mathtt m}$ are vertex-disjoint (i.e., $V(C_{\mathtt i}) \cap V(C_{\mathtt j})= \emptyset$ for all $\mathtt i \neq \mathtt j$) and $V(C_{\mathtt i}) \cap V(H)=\emptyset$ for all $1\leq\mathtt i\leq \mathtt m$.
        \item[(2)] $\operatorname{EndP}({P}_{\mathtt j}) \subset V(H) \cup (\cup_{\mathtt i=1}^{\mathtt m} V(C_{\mathtt i})) \cup (\cup_{\mathtt k=1}^{\mathtt j-1} V(P_{\mathtt k})) \cup \mathsf L(S)$ for $1 \leq \mathtt j \leq \mathtt t$.
        \item[(3)] $\big( V(P_{\mathtt j}) \setminus \operatorname{EndP}(P_{\mathtt j}) \big) \cap \big( V(H) \cup (\cup_{\mathtt i=1}^{\mathtt m} V(C_{\mathtt i})) \cup (\cup_{\mathtt k=1}^{\mathtt j-1} V(P_{\mathtt k}) ) \cup \mathsf L (S) \big) = \emptyset$ for $\mathtt 1 \leq \mathtt j \leq \mathtt t$.
        \item[(4)] $\mathtt t = |\mathsf L(S) \setminus V(H)|+\tau(S)-\tau(H)$.
    \end{enumerate}
\end{lemma}
\begin{remark}{\label{rmk-repeat-decomposition}}
    It is not hard to see that for all $H \subset S$, there are at least $(|\mathsf L(S) \setminus V(H)|+\tau(S)-\tau(H))!$ different decompositions $\{ {P}_{\mathtt 1}, \ldots, {P}_{\mathtt t} \}$ satisfying Items~(1)--(4). Indeed, in the proof (see \cite[Lemma~B.3]{CDGL24+}) one can see that $P_1 \cup \ldots \cup P_{\mathtt t}$ cannot have cycles that do not intersect $V(H) \cup \mathsf L(S) \cup (\cup V(C_{\mathtt i}))$. Thus, for each $P_{\mathtt i}$ there exists a decomposition $(\widetilde P_{\mathtt 1},\ldots,\widetilde{P}_{\mathtt t})$ that $\widetilde{P}_{\mathtt 1}$ contains $P_{\mathtt i}$, so there are at least $\mathtt t$ choices of $P_{\mathtt 1}$. By induction, we know we can construct at least $\mathtt t!$ different $\{ {P}_{\mathtt 1}, \ldots, {P}_{\mathtt t} \}$.  
\end{remark}

\begin{lemma}[Corollary~B.4 in \cite{CDGL24+}]{\label{cor-revised-decomposition-H-Subset-S}}
    For $H \subset S$, we can decompose $E(S)\setminus E(H)$ into $\mathtt m$ cycles ${C}_{\mathtt 1}, \ldots, {C}_{\mathtt m}$ and $\mathtt t$ paths ${P}_{\mathtt 1}, \ldots, {P}_{\mathtt t}$ for some $\mathtt m, \mathtt t\geq 0$ such that the following hold.
    \begin{enumerate}
        \item[(1)] ${C}_{\mathtt 1}, \ldots, {C}_{\mathtt m}$ are independent cycles in $S$.
        \item[(2)] $V(P_{\mathtt j}) \cap \big( V(H) \cup (\cup_{\mathtt i=1}^{\mathtt m} V(C_{\mathtt i})) \cup (\cup_{\mathtt k \neq \mathtt j} V(P_{\mathtt k}) ) \cup \mathsf L (S) \big) = \operatorname{EndP}(P_{\mathtt j})$ for $1 \leq \mathtt j \leq \mathtt t$.
        \item[(3)] $\mathtt t \leq 3(|\mathsf L (S) \setminus V(H)| + \tau(S)-\tau(H))$.
    \end{enumerate}
\end{lemma}
\begin{lemma}[Lemma~B.6 in \cite{CDGL24+}]{\label{lem-enu-cycle-path}}
    Given a vertex set $\mathsf A$ with $|\mathsf A| \leq D$, we have 
    \begin{align*}
        & \#\Big\{ (C_{\mathtt 1},\ldots,C_{\mathtt m};P_{\mathtt 1},\ldots,P_{\mathtt t}) : C_{\mathtt i} \text{'s are vertex disjoint cycles also vertex disjoint from } \mathsf A, \\
        & P_{\mathtt j} \text{'s are paths} ; \#\big( (\cup_{\mathtt i} V(C_{\mathtt i})) \cup (\cup_{\mathtt j} V(P_{\mathtt j})) \big) \leq 2D; \#\{ \mathtt i: |V(C_{\mathtt i})|=x \}=p_x, |E(P_{\mathtt j})|=q_{\mathtt j} \Big\} \\
        & \leq (2D)^{2\mathtt t} \prod_{x} \frac{n^{xp_x}}{p_x!} \prod_{\mathtt j=1}^{\mathtt t} n^{q_{\mathtt j}-1} \,.
    \end{align*}
\end{lemma}
\begin{lemma}[Lemma~B.7 in \cite{CDGL24+}]{\label{lem-enu-Subset-large-graph}}
    For $H \subset \mathsf{K}_n$ with $|E(H)| \leq D$, we have     
    \begin{equation}{\label{eq-enu-Subset-large-graph}}
        \begin{aligned}
            \#\Big\{ & S: \mathsf{L}(S)\subset V(H),  |E(S)|-|E(H)|=\ell+\kappa, \tau(S)-\tau(H)=\ell; \\
            & |E(S)| \leq D, \cup_{j>N} \mathcal{C}_j(S) \subset H \Big\} \leq (2D)^{4\ell} n^{\kappa} \sum_{3p_3+\ldots+Np_N \leq \kappa+l} \prod_{j=3}^{N} \frac{1}{p_j!} \,.
        \end{aligned}
    \end{equation}    
\end{lemma}

\begin{lemma}{\label{lem-enu-large-graph}}
    For $H \subset \mathsf{K}_n$, we have 
    \begin{align}
        & \#\Big\{ S: \mathsf L(S)=\emptyset, \tau(S)= t\leq D; |V(S)|=v \leq D \Big\} \nonumber \\
        \leq\ & \frac{D^{2t}n^{v}}{t!} \sum_{3p_3+\ldots+Dp_D \leq D} \prod_{j=3}^{D} \frac{1}{p_j!} \,. \label{eq-enu-large-graph}
    \end{align}   
\end{lemma}
\begin{proof}
    Take $S$ as an element in the set of \eqref{eq-enu-large-graph}. Using Lemma~\ref{lem-decomposition-H-Subset-S}, we can write (for some $\mathtt m \geq 0$)
    \[
        S \doublesetminus H = \big( \sqcup_{\mathtt i=1}^{\mathtt m} C_{\mathtt i} \big) \sqcup \big( \sqcup_{\mathtt j=1}^{t} P_{\mathtt j} \big) \,,
    \]
    where $\{ C_{\mathtt i} : 1 \leq \mathtt i \leq \mathtt m \}$ is a collection of disjoint cycles and $\{ P_{\mathtt j} : 1 \leq \mathtt j \leq t \}$ is a collection of paths satisfying (1)--(4) in Lemma~\ref{lem-decomposition-H-Subset-S}. Now let us suppose we fix $\mathcal C=\cup_{1 \leq \mathtt i \leq \mathtt m} C_{\mathtt i}$. 
    Using Remark~\ref{rmk-repeat-decomposition}, we know that given $\mathcal C$, the choices of $H$ is bounded by
    \begin{align*}
        \frac{1}{t!} \cdot \#\Big\{ P_{\mathtt j}: 1 \leq \mathtt j \leq t, P_{\mathtt j} \mbox{ satisfy Items (1)--(4)} \Big\} \,.
    \end{align*}
    Consider the mapping 
    \begin{align*}
        \Psi: \Big\{ P_{\mathtt j} : 1 \leq \mathtt j \leq t \Big\} \longrightarrow \Big\{ &(u_1,\ldots,u_N), ((\alpha_1,\beta_1),\ldots,(\alpha_t,\beta_t)): \\
        &N=v-|V(\mathcal C)|, \alpha_i,\beta_i\in V(H) \cup \{ u_1,\ldots,u_N \} \Big\} \,.
    \end{align*}
    defined as follows:
    \begin{itemize}
        \item We let $(\alpha_1,\beta_1)=\mathsf{EndP}(P_{\mathtt 1}) \subset V(H)$. Also let $(u_2,\ldots,u_{N_1})=V(P_{\mathtt 1}) \setminus \mathsf{EndP}(P_{\mathtt 1})$.
        \item For each $1 \leq \ell \leq \mathtt t$, we let $(\alpha_\ell,\beta_\ell)=\mathsf{EndP}(P_{\mathtt \ell})$, and add the vertices in $V(P_{\mathtt 1}) \setminus \mathsf{EndP}(P_{\mathtt 1})$ to $(u_1,\ldots,u_{N_{\ell-1}})$ to form $(u_1,\ldots,u_{N_{\ell}})$.
    \end{itemize}
    It is clear that $\Psi$ is a surjection. Thus, we have
    \begin{align*}
        \#\Big\{ P_{\mathtt j}: 1 \leq \mathtt j \leq t, P_{\mathtt j} \mbox{ satisfy Items (1)--(4)} \Big\} \leq n^{v-|V(\mathcal C)|} D^{2t} \,.
    \end{align*} 
    This yields that 
    \begin{align*}
        & \#\Big\{ S: \mathsf L(S)=\emptyset, \tau(S)= t\leq D; |V(S)|=v \leq D \Big\} \\
        \leq\ & \sum_{\mathcal C} \frac{ n^{v-|V(\mathcal C)|} D^{2t} }{ t! } \leq \frac{D^{2t}n^{v}}{t!} \sum_{3p_3+\ldots+Dp_D \leq D} \prod_{j=3}^{D} \frac{1}{p_j!} \,,
    \end{align*}
    leading to the desired result.
\end{proof}

The next two lemmas deal with the expectation of labeling over a chain or a path, which will be useful in Sections~\ref{subsec:bounding-low-deg-adv-SBM} and \ref{subsec:bounding-low-deg-adv-multilayer-SBM}.
\begin{lemma}{\label{lem-expectation-over-chain-original}}
    For a path $\mathcal{P}$ with $V(\mathcal P)= \{ v_0, \ldots, v_l \} $ and $\operatorname{EndP}(\mathcal P)=\{ v_0,v_l \}$, we have 
    \begin{equation}{\label{eq-expectation-over-chain}}
    \begin{aligned}
        \mathbb{E}_{\sigma \sim \nu} \Big[ \prod_{i=1}^{l} \Big( a+b\omega(\sigma_{i-1},\sigma_i) \Big) \mid \sigma_0, \sigma_l \Big] = a^l + b^l \cdot \omega(\sigma_0, \sigma_l)  \,.
    \end{aligned}
    \end{equation}
\end{lemma}
\begin{proof}
    By independence, we see that $\mathbb E_{\sigma \sim \nu} \big[ \prod_{i \in I} \omega(\sigma_{i-1},\sigma_i) \mid \sigma_0, \sigma_l \big] = 0$ if $I \subsetneq [l]$. Thus,
    \begin{align*}
        \mathbb{E}_{\sigma \sim \nu} \Big[ \prod_{i=1}^{l} \Big( a+b\omega(\sigma_{i-1},\sigma_i) \Big) \mid \sigma_0, \sigma_l \Big] = a^l + b^l \mathbb{E}_{\sigma \sim \nu}\Big[ \prod_{i=1}^{l} \omega(\sigma_{i-1},\sigma_i) \mid \sigma_0, \sigma_l \Big] \,.
    \end{align*}
    It remains to prove that
    \begin{equation}\label{eq-finalgoal1-sign-of-path-cycle-conditioned-on-endpoints}
        \mathbb{E}_{\sigma \sim \nu} \Big[ \prod_{i=1}^{l} \omega(\sigma_{i-1},\sigma_i) \mid \sigma_0, \sigma_l \Big] = \omega(\sigma_0, \sigma_l)\,.
    \end{equation}
    We shall show \eqref{eq-finalgoal1-sign-of-path-cycle-conditioned-on-endpoints} by induction. The case $l=1$ follows immediately. Now we assume that \eqref{eq-finalgoal1-sign-of-path-cycle-conditioned-on-endpoints} holds for $l$. Then we have
    \begin{align*}
        &\mathbb{E}_{\sigma \sim \nu} \Big[ \prod_{i=1}^{l+1} \omega(\sigma_{i-1},\sigma_i) \mid \sigma_0, \sigma_{l+1} \Big]\\
        = & \mathbb{E}_{\sigma \sim \nu}\Big[\omega(\sigma_l,\sigma_{l+1})\mathbb{E}_{\sigma \sim \nu} \Big[ \prod_{i=1}^{l} \omega(\sigma_{i-1},\sigma_i) \mid \sigma_0, \sigma_l,\sigma_{l+1} \Big] \mid \sigma_0,\sigma_{l+1}\Big]\\
        = & \mathbb{E}_{\sigma \sim \nu}\Big[ \omega(\sigma_l,\sigma_{l+1}) \omega(\sigma_0,\sigma_l) \mid \sigma_0,\sigma_{l+1}\Big] = \omega(\sigma_0,\sigma_{l+1})\,,
    \end{align*}
    which completes the induction procedure. 
\end{proof}

\begin{lemma}{\label{lem-leaf-cancellation}}
    For $H \subset S$ with $\mathcal L(S) \not\subset V(H)$, we have 
    \begin{equation}{\label{eq-leaf-cancellation}}
        \mathbb E_{\sigma\sim\nu}\Big[ \prod_{ (i,j) \in E(S) \setminus E(H) } \omega(\sigma_i,\sigma_j) \mid \{ \sigma_u: u \in V(H) \} \Big] = 0 \,.
    \end{equation}
\end{lemma}
\begin{proof}
    Denote $v \in \mathcal L(S) \setminus V(H)$ and $(v,w) \in E(S)$. Define $\sigma_{U}$ and $\sigma_{\setminus U}$ to be the restriction of $\sigma$ on $U$ and on $[n] \setminus U$, respectively. Also define $\nu_U$ and $\nu_{\setminus U}$ to be the restriction of $\nu$ on $U$ and on $[n] \setminus U$, respectively. Then we have (let $\mathtt V=V(H)$)
    \begin{align*}
        & \mathbb E_{\sigma\sim\nu}\Big[ \prod_{ (i,j) \in E(S) \setminus E(H) } \omega(\sigma_i,\sigma_j) \mid \{ \sigma_u: u \in V(H) \} \Big] \\
        =\ & \mathbb E_{ \sigma_{([n] \setminus \mathtt V)} \sim \nu_{([n] \setminus\mathtt V)} }\Big[ \prod_{ (i,j) \in E(S) \setminus E(H) } \omega(\sigma_i,\sigma_j) \Big] \\
        =\ & \mathbb E_{ \sigma_{([n] \setminus (\mathtt V \cup \{v\}))} \sim \nu_{([n] \setminus (\mathtt V \cup \{v\}))} }  \mathbb E_{ \sigma_{\{v\}} \sim \nu_{ \{v\} } }\Big[ \omega(\sigma_v,\sigma_w) \prod_{ (i,j) \in E(S) \setminus (E(H)\cup \{v,w\}) } \omega(\sigma_i,\sigma_j) \Big] = 0 \,. \qedhere
    \end{align*}
\end{proof}

\begin{lemma}{\label{lem-exp-over-chain}}
    Recall that for a path $P$ or a cycle $C$, we define
    \begin{align*}
        F(P) &= \mathbb E\left[ \prod_{(i,j;\ell)\in E(P)} \omega(\sigma_\ell(i),\sigma_\ell(j)) \mid \{ \sigma_\ell(i): 1 \leq \ell \leq L, i\in \mathsf{EndP}(P_{\mathtt t}) \} \right] \,; \\
        F(C) &= \mathbb E\left[ \prod_{(i,j;\ell)\in E(C)} \omega(\sigma_\ell(i),\sigma_\ell(j)) \right] \,.
    \end{align*}
    We then have
    \begin{align*}
        F(P) &= \rho^{2|\mathsf{dif}(P)|} \omega(\sigma_\ell(i),\sigma_{\ell'}(j)) \mbox{ for } \mathsf{EndP}(P)=\{ i,j \} \mbox{ with } i \in V_\ell(P), j \in V_{\ell'}(P) \,; \\
        F(C) &= \rho^{2|\mathsf{dif}(C)|} (q-1) \,.
    \end{align*}
\end{lemma}
\begin{proof}
    We will only prove the equation with respect to $F(P)$, as the equation with respect to $F(C)$ can be derived in the similar manner. Without loss of generality, we will assume that $P=P_N$ with
    \begin{equation*}
        E(P_N) =\{ (u_i,u_{i+1};\ell_i): 1 \leq i \leq N, 1 \leq \ell_i \leq L \}
    \end{equation*}
    and we will prove by induction on $N$. The case $N=1$ follows immediately. Now we assume that the equation holds for $N-1$. Then we have $F(P_{N})$ equals
    \begin{align*}
        &\mathbb{E}\left[ \prod_{i=0}^{N} \omega(\sigma_{\ell_i}(u_i),\sigma_{\ell_i}(u_{i+1})) \mid \{ \sigma_\ell(u_0), \sigma_{\ell}(u_{N+1}): 1 \leq \ell \leq L \} \right]\\
        =\ & \mathbb{E}\Bigg\{ \omega(\sigma_{\ell_{N}}(u_{N}), \sigma_{\ell_{N}}(u_{N+1}))  \\
        & \quad \left. \mathbb{E}\left[ \prod_{i=0}^{N-1} \omega(\sigma_{\ell_i}(u_i),\sigma_{\ell_i}(u_{i+1})) \mid \{ \sigma_\ell(u_0), \sigma_{\ell}(u_{N}), \sigma_{\ell}(u_{N+1}) \} \right] \mid \{ \sigma_\ell(u_0), \sigma_{\ell}(u_{N+1}) \} \right\} \\
        =\ & \rho^{2|\mathsf{dif}(P_{N-1})|} \mathbb{E}\Bigg[ \omega(\sigma_{\ell_{N}}(u_{N}), \sigma_{\ell_{N}}(u_{N+1})) \omega(\sigma_{\ell_{0}}(u_{0}), \sigma_{\ell_{N-1}}(u_{N})) \mid \{ \sigma_\ell(u_0), \sigma_{\ell}(u_{N+1}) \} \Bigg]  \\
        =\ & \rho^{2|\mathsf{dif}(P_{N-1})|+2\cdot\mathbf 1(u_{N} \in \mathsf{dif}(P_{N}))} \omega(\sigma_{\ell_{0}}(u_{0}), \sigma_{\ell_{N}}(u_{N+1})) \\
        =\ & \rho^{2|\mathsf{dif}(P_{N})|} \omega(\sigma_{\ell_{0}}(u_{0}), \sigma_{\ell_{N}}(u_{N+1})) \,,
    \end{align*}
    which completes the induction procedure.
\end{proof}

\bibliographystyle{alpha}

\small
\begin{thebibliography}{10} 

\bibitem[Abbe18]{Abbe18}
Emmanuel Abbe. 
\newblock Community detection and stochastic block models: recent developments. 
\newblock {\em Journal of Machine Learning Research}, 18(177):1--86, 2018.

\bibitem[AS15]{AS15}
Emmanuel Abbe and Colin Sandon. 
\newblock Community detection in general stochastic block models: Fundamental limits and efficient algorithms for recovery. 
\newblock In {\em Proceedings of the IEEE 56th Annual Symposium on Foundations of Computer Science (FOCS)}, pages 670--688. IEEE, 2015.

\bibitem[AS16]{AS16}
Emmanuel Abbe and Colin Sandon. 
\newblock Crossing the KS threshold in the stochastic block model with information theory. 
\newblock In {\em Proceedings of IEEE International Symposium on Information Theory (ISIT)}, pages 840--844. IEEE, 2016.

\bibitem[AS18]{AS18}
Emmanuel Abbe and Colin Sandon. 
\newblock Proof of the achievability conjectures for the general stochastic block model. 
\newblock {\em Communications on Pure and Applied Mathematics}, 71(7):1334--1406, 2018.

\bibitem[Ames15]{Ames13}
Brendan P.W. Ames. 
\newblock Guaranteed recovery of planted cliques and dense subgraphs by convex relaxation. 
\newblock {\em Journal of Optimization Theory and Applications}, 167(2):653--675, 2015.

\bibitem[ACD11]{ACD11}
Ery Arias-Castro, Emmanuel J. Candes, and Arnaud Durand. 
\newblock Detection of an anomalous cluster in a network. 
\newblock {\em Annals of Statistics}, 39(1):278--304, 2011.

\bibitem[AV14]{AV14}
Ery Arias-Castro and Nicolas Verzelen. 
\newblock Community detection in dense random networks.
\newblock {\em Annals of Statistics}, 42(3):940--969, 2014.

\bibitem[BBP05]{BBP05}
Jinho Baik, G\'erard Ben Arous, and Sandrine P\'ech\'e. 
\newblock Phase transition of the largest eigenvalue for nonnull complex sample covariance matrices. 
\newblock {\em Annals of Probability}, 33(5):1643--1697, 2005.

\bibitem[BS06]{BS06}
Jinho Baik and Jack W. Silverstein. 
\newblock Eigenvalues of large sample covariance matrices of spiked population models. 
\newblock {\em Journal of Multivariate Analysis}, 97(6):1382--1408, 2006.

\bibitem[BBK+21]{BBK+21}
Afonso S. Bandeira, Jess Banks, Dmitriy Kunisky, Christopher Moore, and Alexander S. Wein.
\newblock Spectral planting and the hardness of refuting cuts, colorability, and communities in random graphs.
\newblock In {\em Proceedings of the 34th Annual Conference on Learning Theory (COLT)}, pages 410--473. PMLR, 2021.

\bibitem[BBS17]{BBS17}
Afonso S. Bandeira, Nicolas Boumal, and Amit Singer. 
\newblock Tightness of the maximum likelihood semidefinite relaxation for angular synchronization. 
\newblock {\em Mathematical Programming}, 163(1):145--167, 2017.

\bibitem[BCLS20]{BCLS20}
Afonso S. Bandeira, Yutong Chen, Roy R. Lederman, and Amit Singer. 
\newblock Non-unique games over compact groups and orientation estimation in cryo-EM. 
\newblock {\em Inverse Problems}, 36(6):064002, 2020.

\bibitem[BEH+22]{BEH+22}
Afonso S. Bandeira, Ahmed El Alaoui, Samuel B. Hopkins, Tselil Schramm, Alexander S. Wein, and Ilias Zadik. 
\newblock The Franz-Parisi criterion and computational trade-offs in high dimensional statistics. 
\newblock In {\em Advances in Neural Information Processing Systems (NeurIPS)}, volume~35, pages 33831--33844. Curran Associates, Inc., 2022.

\bibitem[BKMR25+]{BKMR25+}
Afonso S. Bandeira, Anastasia Kireeva, Antoine Maillard, and Almut R\"{o}dder.
\newblock Randomstrasse101: Open problems of 2024.
\newblock arXiv preprint, arXiv:2504.20539.

\bibitem[BKW20]{BKW20}
Afonso S. Bandeira, Dmitriy Kunisky, and Alexander S. Wein.
\newblock Computational hardness of certifying bounds on constrained PCA problems.
\newblock In {\em Proceedings of the 11th Innovations in Theoretical Computer Science Conference (ITCS)}, pages 78:1--78:29. Schloss Dagstuhl-Leibniz-Zentrumf\"{u}r Informatik, 2020.

\bibitem[BPW18]{BPW18}
Afonso S. Bandeira, Amelia Perry, and Alexander S. Wein. 
\newblock Notes on computational-to-statistical gaps: Predictions using statistical physics.
\newblock {\em Portugaliae Mathematica}, 75(2):159--186, 2018.

\bibitem[BMNN16]{BMNN16}
Jess Banks, Cristopher Moore, Joe Neeman, and Praneeth Netrapalli. 
\newblock Information-theoretic thresholds for community detection in sparse networks. 
\newblock In {\em Proceedings of the 28th Annual Conference on Learning Theory (COLT)}, pages 383--416. PMLR, 2016.

\bibitem[BHK+19]{BHK+19}
Boaz Barak, Samuel B. Hopkins, Jonathan Kelner, Pravesh K. Kothari, Ankur Moitra, and Aaron Potechin. 
\newblock A nearly tight sum-of-squares lower bound for the planted clique problem. 
\newblock {\em SIAM Journal on Computing}, 48(2):687--735, 2019.

\bibitem[BKR26]{BKR24}
Jean Barbier, Justin Ko, and Anas A. Rahman. 
\newblock A multiscale cavity method for sublinear-rank symmetric matrix factorization. 
\newblock {\em Mathematical Statistics and Learning}, 9(1-2):1--68, 2026.

\bibitem[BGJ20]{BGJ20}
G\'erard Ben Arous, Reza Gheissari, and Aukosh Jagannath. 
\newblock Algorithmic thresholds for tensor PCA. 
\newblock {\em Annals of Probability}, 48(4):2052--2087, 2020.

\bibitem[BWZ23]{BWZ23}
G\'erard Ben Arous, Alexander S. Wein, and Ilias Zadik. 
\newblock Free energy wells and overlap gap property in sparse PCA.
\newblock {\em Communications on Pure and Applied Mathematics}, 76(10):2410--2473, 2023.

\bibitem[BN11]{BN11}
Florent Benaych-Georges and Raj Rao Nadakuditi. 
\newblock The eigenvalues and eigenvectors of finite, low rank perturbations of large random matrices. 
\newblock {\em Advances in Mathematics}, 227(1):494--521, 2011.

\bibitem[BCC+10]{BCC+10}
Aditya Bhaskara, Moses Charikar, Eden Chlamtac, Uriel Feige, and Aravindan Vijayaraghavan. 
\newblock Detecting high log-densities: An $o(n^{1/4})$ approximation for densest $k$-subgraph. 
\newblock In {\em Proceedings of the 42nd Annual ACM Symposium on Theory of Computing (STOC)}, pages 201--210. ACM, 2010.

\bibitem[BC20+]{BC20+}
Sharmodeep Bhattacharyya and Shirshendu Chatterjee. 
\newblock General community detection with optimal recovery conditions for multi-relational sparse networks with dependent layers. 
\newblock arXiv preprint, arXiv:2004.03480.

\bibitem[BJR07]{BJR07}
B\'ela Bollob\'as, Svante Janson, and Oliver Riordan. 
\newblock The phase transition in inhomogeneous random graphs. 
\newblock {\em Random Structures and Algorithms}, 31(1):3--122, 2007.

\bibitem[BLM18]{BLM15}
Charles Bordenave, Marc Lelarge, and Laurent Massouli\'e. 
\newblock Non-backtracking spectrum of random graphs: Community detection and non-regular Ramanujan graphs. 
\newblock {\em Annals of Probability}, 46(1):1--71, 2018.

\bibitem[BSAB14]{BSAB14}
Nicolas Boumal, Amit Singer, P.-A. Absil, and Vincent D. Blondel. 
\newblock Cram\'er–Rao bounds for synchronization of rotations.
\newblock {\em Information and Inference: A Journal of the IMA}, 3(1):1--39, 2014.

\bibitem[BB20]{BB20}
Matthew Brennan and Guy Bresler. 
\newblock Reducibility and statistical-computational gaps from secret leakage. 
\newblock In {\em Proceedings of the 33rd Annual Conference on Learning Theory (COLT)}, pages 648--847. PMLR, 2020.

\bibitem[BBH18]{BBH18}
Matthew Brennan, Guy Bresler, and Wasim Huleihel.
\newblock Reducibility and computational lower bounds for problems with planted sparse structure. 
\newblock In {\em Proceedings of the 31st Annual Conference on Learning Theory (COLT)}, pages 48--166. PMLR, 2018.

\bibitem[BH22]{BH22}
Guy Bresler and Brice Huang. 
\newblock The algorithmic phase transition of random $k$-SAT for low degree polynomials. 
\newblock In {\em Proceedings of the IEEE 62nd Annual Symposium on Foundations of Computer Science (FOCS)}, pages 298--309. IEEE, 2022.

\bibitem[BDT24]{BDT24}
Rares-Darius Buhai, Jingqiu Ding, and Stefan Tiegel.
\newblock Computational-statistical gaps for improper learning in sparse
linear regression.
\newblock In {\em Proceedings of the 37th Annual Conference on Learning Theory (COLT)}, pages 752--771. PMLR, 2024.

\bibitem[BHJK25]{BHJK25}
Rares-Darius Buhai, Jun-Ting Hsieh, Aayush Jain, and Pravesh K. Kothari.
\newblock The quasi-polynomial low-degree conjecture is false.
\newblock In {\em Proceedings of the IEEE 66th Annual Symposium on Foundations of Computer Science (FOCS)}, pages 2577--2590. IEEE, 2025.

\bibitem[BABP16]{BABP16}
Jo\"{e}l Bun, Romain Allez, Jean-Philippe Bouchaud, and Marc Potters. 
\newblock Rotational invariant estimator for general noisy matrices. 
\newblock {\em IEEE Transactions on Information Theory}, 62(12):7475--7490, 2016.

\bibitem[BI13]{BI13}
Cristina Butucea and Yuri I. Ingster.
\newblock Detection of a sparse submatrix of a high-dimensional noisy matrix.
\newblock {\em Bernoulli}, 19(5):2652--2688, 2013.

\bibitem[BIS15]{BIS15}
Cristina Butucea, Yuri I. Ingster, and Irina A. Suslina. 
\newblock Sharp variable selection of a sparse submatrix in a high-dimensional noisy matrix. 
\newblock {\em ESAIM: Probability and Statistics}, 19:115--134, 2015.

\bibitem[CLR17]{CLR17}
Tony T. Cai, Tengyuan Liang, and Alexander Rakhlin. 
\newblock Computational and statistical boundaries for submatrix localization in a large noisy matrix. 
\newblock {\em Annals of Statistics}, 45(4):1403--1430, 2017.

\bibitem[CDF09]{CDF09}
Mireille Capitaine, Catherine Donati-Martin, and Delphine F\'eral. 
\newblock The largest eigenvalues of finite rank deformation of large Wigner matrices: Convergence and nonuniversality of the fluctuations. 
\newblock {\em Annals of Probability}, 37(1):1--47, 2009.

\bibitem[CGGV25+]{CGGV25+}
Alexandra Carpentier, Simone Maria Giancola, Christophe Giraud, and Nicolas Verzelen.
\newblock Low-degree lower bounds via almost orthonormal bases.
\newblock arXiv preprint, arXiv:2509.09353.

\bibitem[CGV25+]{CGV25+}
Alexandra Carpentier, Christophe Giraud, and Nicolas Verzelen.
\newblock Phase transition for stochastic block model with more than $\sqrt{n}$ communities.
\newblock arXiv preprint, arXiv:2509.15822.

\bibitem[CDGL26]{CDGL24+}
Guanyi Chen, Jian Ding, Shuyang Gong, and Zhangsong Li. 
\newblock A computational transition for detecting correlated stochastic block models by low-degree polynomials. 
\newblock {\em Annals of Statistics}, 54(1):226--251, 2026.

\bibitem[CLM22]{CLM22}
Shuxiao Chen, Sifan Liu, and Zongming Ma.
\newblock Global and individualized community detection in inhomogeneous multilayer networks.
\newblock {\em Annals of Statistics}, 50(5):2664--2693, 2022.


\bibitem[CX16]{CX16}
Yudong Chen and Jiaming Xu. 
\newblock Statistical-computational tradeoffs in planted problems and submatrix localization with a growing number of clusters and submatrices. 
\newblock {\em Journal of Machine Learning Research}, 17(1):882--938, 2016.

\bibitem[CMSW25]{CMSW25}
Byron Chin, Elchanan Mossel, Youngtak Sohn, and Alexander S. Wein. 
\newblock Stochastic block models with many communities and the Kesten–Stigum bound. 
\newblock In {\em Proceedings of the 38th Annual Conference on Learning Theory (COLT)}, pages 1253--1258. PMLR, 2025.

\bibitem[Coj10]{Coj10}
Amin Coja-Oghlan. 
\newblock Graph partitioning via adaptive spectral techniques.
\newblock {\em Combinatorics, Probability and Computing}, 19(02):227--284, 2010.

\bibitem[CKPZ18]{CKPZ18}
Amin Coja-Oghlan, Florent Krzakala, Will Perkins, and Lenka Zdeborov\'a.
\newblock Information-theoretic thresholds from the cavity method. 
\newblock {\em Advances in Mathematics}, 333:694--795, 2018.

\bibitem[CK01]{CK01}
Anne Condon and Richard M. Karp. 
\newblock Algorithms for graph partitioning on the planted partition model. 
\newblock {\em Random Structures and Algorithms}, 18(2):116--140, 2001.

\bibitem[DMR11]{DMR11}
Constantinos Daskalakis, Elchanan Mossel, and Sebastien Roch. 
\newblock Evolutionary trees and the Ising model on the Bethe lattice: A proof of Steel's conjecture. 
\newblock {\em Probability Theory and Related Fields}, 149(1-2):149--189, 2011.

\bibitem[dEJL07]{dGJL07}
Alexandre d'Aspremont, Laurent El Ghaoui, Michael I. Jordan, and Gert R.G. Lanckriet. 
\newblock A direct formulation for sparse PCA using semidefinite programming. 
\newblock {\em SIAM Review}, 49(3):434--448, 2007.

\bibitem[DKMZ11]{DKMZ11}
Aurelien Decelle, Florent Krzakala, Cristopher Moore, and Lenka Zdeborov\'{a}.
\newblock Asymptotic analysis of the stochastic block model for modular networks and its algorithmic applications.
\newblock {\em Physics Review E}, 84(6):066106, 2011.

\bibitem[DAM16]{DAM16}
Yash Deshpande, Emmanuel Abbe, and Andrea Montanari. 
\newblock Asymptotic mutual information for the balanced binary stochastic block model. 
\newblock {\em Information and Inference: A Journal of the IMA}, 6(2):125--170, 2016.

\bibitem[DM14]{DM14}
Yash Deshpande and Andrea Montanari. 
\newblock Information-theoretically optimal sparse PCA. 
\newblock In {\em Proceedings of IEEE International Symposium on Information Theory (ISIT)}, pages 2197--2201. IEEE, 2014.

\bibitem[DM15]{DM15}
Yash Deshpande and Andrea Montanari. 
\newblock Finding hidden cliques of size $\sqrt{N/e}$ in nearly linear time.
\newblock {\em Foundations of Computational Mathematics}, 15(4):1069--1128, 2015.

\bibitem[DKS17]{DKS17}
Ilias Diakonikolas, Daniel M. Kane, and Alistair Stewart. 
\newblock Statistical query lower bounds for robust estimation of high-dimensional Gaussians and Gaussian mixtures. 
\newblock In {\em Proceedings of the IEEE 58th Annual Symposium on Foundations of Computer Science (FOCS)}, pages 73--84. IEEE, 2017.

\bibitem[DDL25]{DDL25}
Jian Ding, Hang Du, and Zhangsong Li. 
\newblock Low-degree hardness of detection for correlated \ER graphs.
\newblock {\em Annals of Statistics}, 53(5):1833--1856, 2025.

\bibitem[DKWB24]{DKWB24}
Yunzi Ding, Dmitriy Kunisky, Alexander S. Wein, and Afonso S. Bandeira. 
\newblock Subexponential-time algorithms for sparse PCA. 
\newblock {\em Foundations of Computational Mathematics}, 24(3):865--914, 2024.

\bibitem[DHSS25]{DHSS25}
Jingqiu Ding, Yiding Hua, Lucas Slot, and David Steurer.
\newblock Low-degree evidence for computational transition of recovery rate in stochastic block model.
\newblock In {\em Advances in Neural Information Processing Systems (NeurIPS)}, volume~39. Curran Associates, Inc., 2025.

\bibitem[DMW25]{DMW25}
Abhishek Dhawan, Cheng Mao, and Alexander S. Wein. 
\newblock Detection of dense subhypergraphs by low-degree polynomials. 
\newblock {\em Random Structures and Algorithms}, 66(1):e21279, 2025.

\bibitem[Du25+]{Du25+}
Hang Du.
\newblock Optimal recovery of correlated \ER graphs. 
\newblock {\em Annals of Applied Probability}, to appear.

\bibitem[DF89]{DF89}
Martin E. Dyer and Alan M. Frieze. 
\newblock The solution of some random NP-hard problems in polynomial expected time. 
\newblock {\em Journal of Algorithms}, 10(4):451--489, 1989.

\bibitem[EKJ20]{EKJ20}
Ahmed El Alaoui, Florent Krzakala, and Michael I. Jordan. 
\newblock Fundamental limits of detection in the spiked Wigner model. 
\newblock {\em Annals of Statistics}, 48(2):863--885, 2020.

\bibitem[FGR+17]{FGR+17}
Vitaly Feldman, Elena Grigorescu, Lev Reyzin, Santosh S. Vempala, and Ying Xiao. 
\newblock Statistical algorithms and a lower bound for detecting planted cliques.
\newblock {\em Journal of the ACM}, 64(2):1--37, 2017.

\bibitem[FP07]{FP07}
Delphine F\'eral and Sandrine P\'ech\'e. 
\newblock The largest eigenvalue of rank one deformation of large Wigner matrices. 
\newblock {\em Communications in Mathematical Physics}, 272(1):185--228, 2007.

\bibitem[Gam21]{Gamarnik21}
David Gamarnik. 
\newblock The overlap gap property: A topological barrier to optimizing over random structures. 
\newblock {\em Proceedings of the National Academy of Sciences of the United States of America}, 118(41):e2108492118, 2021.

\bibitem[GJS21]{GJS21}
David Gamarnik, Aukosh Jagannath, and Subhabrata Sen. 
\newblock The overlap gap property in principal submatrix recovery.
\newblock {\em Probability Theory and Related Fields}, 181(3-4):757--814, 2021.

\bibitem[GJW24]{GJW24}
David Gamarnik, Aukosh Jagannath, and Alexander S. Wein. 
\newblock Hardness of random optimization problems for Boolean circuits, low-degree polynomials, and Langevin dynamics. 
\newblock {\em SIAM Journal on Computing}, 53(1):1--46, 2024.


\bibitem[GHL26+]{GHL26+}
Shuyang Gong, Dong Huang, and Zhangsong Li.
\newblock Fundamental limits of community detection in contextual multi-layer stochastic block models.
\newblock arXiv preprint, arXiv:2602.08173.

\bibitem[HWX15]{HWX15}
Bruce Hajek, Yihong Wu, and Jiaming Xu. 
\newblock Computational lower bounds for community detection on random graphs. 
\newblock In {\em Proceedings of the 28th Annual Conference on Learning Theory (COLT)}, pages 899--928. PMLR, 2015.

\bibitem[HWX18]{HWX18}
Bruce Hajek, Yihong Wu, and Jiaming Xu. 
\newblock Submatrix localization via message passing.
\newblock {\em Journal of Machine Learning Research}, 18(186):1--52, 2018.

\bibitem[HLL83]{HLL83}
Paul W. Holland, Kathryn B. Laskey, and Samuel Leinhardt. 
\newblock Stochastic block models: first steps.
\newblock {\em Social Networks}, 5(2):109--137, 1983.

\bibitem[HW21]{HW21}
Justin Holmgren and Alexander S. Wein. 
\newblock Counterexamples to the low-degree conjecture. 
\newblock In {\em Proceedings of the 12th Innovations in Theoretical Computer Science Conference (ITCS)}, pages 75:1--75:9. Schloss Dagstuhl-Leibniz-Zentrumf\"{u}r Informatik, 2021.

\bibitem[HSV20]{HSV20}
Guy Holtzman, Adam Soffer, and Dan Vilenchik. 
\newblock A greedy anytime algorithm for sparse PCA. 
\newblock In {\em Proceedings of the 33rd Annual Conference on Learning Theory (COLT)}, pages 1939--1956. PMLR, 2020.

\bibitem[Hop18]{Hopkins18}
Samuel B. Hopkins. 
\newblock \emph{Statistical Inference and the Sum of Squares Method}. 
\newblock PhD thesis, Cornell University, 2018.

\bibitem[HKP+17]{HKP+17}
Samuel B. Hopkins, Pravesh K. Kothari, Aaron Potechin,  Prasad Raghavendra, Tselil Schramm, and David Steurer. 
\newblock The power of sum-of-squares for detecting hidden structures. 
\newblock In {\em Proceedings of the IEEE 58th Annual Symposium on Foundations of Computer Science (FOCS)}, pages 720--731. IEEE, 2017.

\bibitem[HS17]{HS17}
Samuel B. Hopkins and David Steurer. 
\newblock Efficient Bayesian estimation from few samples: Community detection and related problems. 
\newblock In {\em Proceedings of the IEEE 58th Annual Symposium on Foundations of Computer Science (FOCS)}, pages 379--390. IEEE, 2017.

\bibitem[Hua18]{Hua18}
Jiaoyang Huang. 
\newblock Mesoscopic perturbations of large random matrices. 
\newblock {\em Random Matrices: Theory and Applications}, 7(02):1850004, 2018.

\bibitem[HM25]{HM25}
Han Huang and Elchanan Mossel. 
\newblock Polynomial low degree hardness for broadcasting on trees. 
\newblock In {\em Proceedings of the 38th Annual Conference on Learning Theory (COLT)}, pages 2856--2857. PMLR, 2025.

\bibitem[JM13]{JM13}
Adel Javanmard and Andrea Montanari. 
\newblock State evolution for general approximate message passing algorithms, with applications to spatial coupling. 
\newblock {\em Information and Inference: A Journal of the IMA}, 2(2):115--144, 2013.

\bibitem[JMR16]{JMR16}
Adel Javanmard, Andrea Montanari, and Federico Ricci-Tersenghi. 
\newblock Phase transitions in semidefinite relaxations. 
\newblock {\em Proceedings of the National Academy of Sciences of the United States of America}, 113(16):E2218–E2223, 2016.

\bibitem[JV26+]{JV26+}
He Jia and Aravindan Vijayaraghavan.
\newblock Low-degree method fails to predict robust subspace recovery.
\newblock arXiv preprint, arXiv:2603.02594.

\bibitem[JL09]{JL09}
Iain M. Johnstone and Yu A. Lu.  
\newblock On consistency and sparsity for principal components analysis in high dimensions. 
\newblock {\em Journal of the American Statistical Association}, 104(486):682--693, 2009.

\bibitem[KS66]{KS66}
Harry Kesten and Bernt P. Stigum. 
\newblock Additional limit theorems for indecomposable multidimensional Galton-Watson processes. 
\newblock {\em Annals of Mathematical Statistics}, 37:1463--1481, 1966.

\bibitem[KBK24+]{KBK24+}
Anastasia Kireeva, Afonso S. Bandeira, and Dmitriy Kunisky.
\newblock Computational lower bounds for multi-frequency group synchronization.
\newblock {\em Applied and Computational Harmonic Analysis}, to appear.

\bibitem[KBRS11]{KBRS11}
Mladen Kolar, Sivaraman Balakrishnan, Alessandro Rinaldo, and Aarti Singh.
\newblock Minimax localization of structural information in large noisy matrices. 
\newblock In {\em Advances in Neural Information Processing Systems (NeurIPS)}, volume~24, pages 909--917. Curran Associates, Inc., 2011.

\bibitem[KMOW17]{KMOW17}
Pravesh K Kothari, Ryuhei Mori, Ryan O'Donnell, and David Witmer. 
\newblock Sum of squares lower bounds for refuting any CSP. 
\newblock In {\em Proceedings of the 49th Annual ACM Symposium on Theory of Computing (STOC)}, pages 132--145. ACM, 2017.

\bibitem[KVWX23]{KVWX23}
Pravesh K. Kothari, Santosh S. Vempala, Alexander S. Wein, and Jeff Xu.
\newblock Is planted coloring easier than planted clique?
\newblock In {\em Proceedings of the 36th Annual Conference on Learning Theory (COLT)}, pages 5343--5372. PMLR, 2023.

\bibitem[Kun21]{Kunisky21}
Dmitriy Kunisky.
\newblock Hypothesis testing with low-degree polynomials in the Morris class of exponential families. 
\newblock In {\em Proceedings of the 34th Conference on Learning Theory (COLT)}, pages 2822--2848. PMLR, 2021.

\bibitem[KMW24]{KMW24}
Dmitriy Kunisky, Cristopher Moore, and Alexander S. Wein. 
\newblock Tensor cumulants for statistical inference on invariant distributions. 
\newblock In {\em Proceedings of the IEEE 65th Annual Symposium on Foundations of Computer Science (FOCS)}, pages 1007--1026. IEEE, 2024.

\bibitem[KWB22]{KWB22}
Dmitriy Kunisky, Alexander S. Wein, and Afonso S. Bandeira. 
\newblock Notes on computational hardness of hypothesis testing: Predictions using the low-degree likelihood ratio. 
\newblock In {\em Mathematical Analysis, its Applications and Computation: ISAAC}, pages 1--50. Springer, 2022.


\bibitem[LCL20]{LCL20}
Jing Lei, Kehui Chen, and Brian Lynch. 
\newblock Consistent community detection in multi-layer network data. 
\newblock {\em Biometrika}, 107(1):61--73, 2020.

\bibitem[LZZ24]{LZZ24}
Jing Lei, Anru R. Zhang, and Zihan Zhu.
\newblock Computational and statistical thresholds in multi-layer stochastic block models.
\newblock {\em Annals of Statistics}, 52(5):2431--2455, 2024.

\bibitem[LKZ15]{LKZ15}
Thibault Lesieur, Florent Krzakala, and Lenka Zdeborov\'a. 
\newblock Phase transitions in sparse PCA. 
\newblock In {\em Proceedings of IEEE International Symposium on Information Theory (ISIT)}, pages 1635--1639. IEEE, 2015.

\bibitem[LM19]{LM19}
Marc Lelarge and L\'eo Miolane. 
\newblock Fundamental limits of symmetric low-rank matrix estimation. 
\newblock {\em Probability Theory and Related Fields}, 173(3-4):859--929, 2019.

\bibitem[Li25]{Li25}
Zhangsong Li.
\newblock Algorithmic contiguity from low-degree conjecture and applications in correlated random graphs.
\newblock In {\em Approximation, Randomization, and Combinatorial Optimization. Algorithms and Techniques (APPROX/RANDOM)}, volume~353, pages 30:1--30:18. Schloss Dagstuhl-Leibniz-Zentrumf\"{u}r Informatik, 2025.

\bibitem[LWB22]{LWB22}
Matthias L\"{o}ffler, Alexander S. Wein, and Afonso S. Bandeira. 
\newblock Computationally efficient sparse clustering.
\newblock {\em Information and Inference: A Journal of the IMA}, 11(4):1255--1286, 2022.

\bibitem[LG24]{LG24}
Yuetian Luo and Chao Gao. 
\newblock Computational lower bounds for graphon estimation via low-degree polynomials. 
\newblock {\em Annals of Statistics}, 52(5):2318--2348, 2024.

\bibitem[LZ22]{LZ22}
Yuetian Luo and Anru R. Zhang. 
\newblock Tensor clustering with planted structures: Statistical optimality and computational limits. 
\newblock {\em Annals of Statistics}, 50(1):584--613, 2022.

\bibitem[MN23]{MN23}
Zongming Ma and Sagnik Nandy.
\newblock Community detection with contextual multilayer networks.
\newblock {\em IEEE Transactions on Information Theory}, 69(5):3203--3239, 2023.

\bibitem[MW15]{MW15}
Zongming Ma and Yihong Wu. 
\newblock Computational barriers in minimax submatrix detection. 
\newblock {\em Annals of Statistics}, 43(3):1089--1116, 2015.

\bibitem[MKMZ22]{MKMZ22}
Antoine Maillard, Florent Krzakala, Marc M\'ezard, and Lenka Zdeborov\'a.
\newblock Perturbative construction of mean-field equations in extensive-rank matrix factorization and denoising. 
\newblock {\em Journal of Statistical Mechanics: Theory and Experiment}, 2022(8):083301, 2022.

\bibitem[MW25a]{MW25b}
Cheng Mao and Alexander S. Wein. 
\newblock Optimal spectral recovery of a planted vector in a subspace.
\newblock {\em Bernoulli}, 31(2):1114--1139, 2025.

\bibitem[MWZ23]{MWZ23}
Cheng Mao, Alexander S. Wein, and Shenduo Zhang. 
\newblock Detection-recovery gap for planted dense cycles. 
\newblock In {\em Proceedings of the 36th Annual Conference on Learning Theory (COLT)}, pages 2440--2481. PMLR, 2023.

\bibitem[MWXY24]{MWXY24}
Cheng Mao, Yihong Wu, Jiaming Xu, and Sophie H. Yu.
\newblock Testing network correlation efficiently via counting trees.
\newblock {\em Annals of Statistics}, 52(6):2483--2505, 2024.

\bibitem[MWXY23]{MWXY23}
Cheng Mao, Yihong Wu, Jiaming Xu, and Sophie H. Yu. 
\newblock Random graph matching at Otter's threshold via counting chandeliers.
\newblock In {\em Proceedings of the 55th Annual ACM Symposium on Theory of Computing (STOC)}, pages 1345--1356. ACM, 2023.

\bibitem[Mas14]{Mas14}
Laurent Massouli\'e. 
\newblock Community detection thresholds and the weak Ramanujan property. 
\newblock In {\em Proceedings of the 46th Annual ACM Symposium on Theory of Computing (STOC)}, pages 694--703. ACM, 2014.

\bibitem[McS01]{McS01}
Frank McSherry. 
\newblock Spectral partitioning of random graphs. 
\newblock In {\em Proceedings of the 42nd IEEE Annual Symposium on  Foundations of Computer Science (FOCS)}, pages 529--537. IEEE, 2001.

\bibitem[MW25b]{MW23+}
Ankur Moitra and Alexander S. Wein.
\newblock Precise error rates for computationally efficient testing.
\newblock {\em Annals of Statistics}, 53(2):854--878, 2025.

\bibitem[Mon15]{Mon15}
Andrea Montanari. 
\newblock Finding one community in a sparse graph. 
\newblock {\em Journal of Statistical Physics}, 161:273--299, 2015.

\bibitem[MW25c]{MW25}
Andrea Montanari and Alexander S. Wein. 
\newblock Equivalence of approximate message passing and low-degree polynomials in rank-one matrix estimation. 
\newblock {\em Probability Theory and Related Fields}, 191(1-2):181--233, 2025.

\bibitem[MNS15]{MNS15}
Elchanan Mossel, Joe Neeman, and Allan Sly. 
\newblock Reconstruction and estimation in the planted partition model. 
\newblock {\em Probability Theory and Related Fields}, 162(3-4):431--461, 2015.

\bibitem[MNS18]{MNS18}
Elchanan Mossel, Joe Neeman, and Allan Sly.
\newblock A proof of the block model threshold conjecture. 
\newblock {\em Combinatorica}, 38(3):665--708, 2018.

\bibitem[MSS25a]{MSS25a}
Elchanan Mossel, Allan Sly, and Youngtak Sohn.
\newblock Exact phase transitions for stochastic block models and reconstruction on trees.
\newblock {\em Annals of Probability}, 53(3):967--1018, 2025.

\bibitem[MSS25b]{MSS25b}
Elchanan Mossel, Allan Sly, and Youngtak Sohn.
\newblock Weak recovery, hypothesis testing, and mutual information in stochastic block models and planted factor graphs.
\newblock In {\em Proceedings of the 57th Annual ACM Symposium on Theory of Computing (STOC)}, pages 2062--2073. ACM, 2025.

\bibitem[OVBS17]{OVBS17}
Onur \"{O}zyesil, Vladislav Voroninski, Ronen Basri, and Amit Singer. 
\newblock A survey of structure from motion.
\newblock {\em Acta Numerica}, 26:305--364, 2017.

\bibitem[PWBM16+]{PWBM16+}
Amelia Perry, Alexander S. Wein, Afonso S. Bandeira, and Ankur Moitra.
\newblock Optimality and sub-optimality of PCA for spiked random matrices and synchronization.
\newblock arXiv preprint, arXiv:1609.05573.

\bibitem[PWBM18a]{PWBM18a}
Amelia Perry, Alexander S. Wein, Afonso S. Bandeira, and Ankur Moitra. 
\newblock Optimality and sub-optimality of PCA I: Spiked random matrix models. 
\newblock {\em Annals of Statistics}, 46(5):2416--2451, 2018.

\bibitem[PWBM18b]{PWBM18b}
Amelia Perry, Alexander S. Wein, Afonso S. Bandeira, and Ankur Moitra.
\newblock Message-passing algorithms for synchronization problems over compact groups.
\newblock {\em Communications on Pure and Applied Mathematics}, 71(11):2275--2322, 2018.

\bibitem[PBM24]{PBM24}
Farzad Pourkamali, Jean Barbier, and Nicolas Macris. 
\newblock Matrix inference in growing rank regimes. 
\newblock {\em IEEE Transactions on Information Theory}, 70(11):8133--8163, 2024.

\bibitem[RSS18]{RSS18}
Prasad Raghavendra, Tselil Schramm, and David Steurer. 
\newblock High dimensional estimation via sum-of-squares proofs. 
\newblock In {\em Proceedings of the International Congress of Mathematicians (ICM)}, pages 3389--3423. World Scientific, 2018.

\bibitem[RM14]{RM14}
Emile Richard and Andrea Montanari. 
\newblock A statistical model for tensor PCA. 
\newblock In {\em Advances in Neural Information Processing Systems (NeurIPS)}, volume~27, pages 2897--2905. MIT Press, 2014.

\bibitem[RS17]{RS17}
Sebastien Roch and Allan Sly. 
\newblock Phase transition in the sample complexity of likelihood-based phylogeny inference. 
\newblock {\em Probability Theory and Related Fields}, 169(1-2):3--62, 2017.

\bibitem[RCY11]{RCY11}
Karl Rohe, Sourav Chatterjee, and Bin Yu. 
\newblock Spectral clustering and the high-dimensional stochastic block model. 
\newblock {\em Annals of Statistics}, 39(4):1878--1915, 2011.

\bibitem[RCBJ19]{RCBJ19}
David M. Rosen, Luca Carlone, Afonso S. Bandeira, John J. Leonard. 
\newblock SE-Sync: A certifiably correct algorithm for synchronization over the special Euclidean group. 
\newblock {\em International Journal of Robotics Research}, 38(2-3):95--125, 2019.

\bibitem[SW22]{SW22}
Tselil Schramm and Alexander S. Wein. 
\newblock Computational barriers to estimation from low-degree polynomials. 
\newblock {\em Annals of Statistics}, 50(3):1833--1858, 2022.

\bibitem[SHSS16]{SHSS16}
Yanyao Shen, Qixing Huang, Nati Srebro, and Sujay Sanghavi. 
\newblock Normalized spectral map synchronization.
\newblock In {\em Advances in Neural Information Processing Systems (NeurIPS)}, volume~29, pages 4925--4933. Curran Associates, Inc., 2016.

\bibitem[Sin11]{Sin11}
Amit Singer. 
\newblock Angular synchronization by eigenvectors and semidefinite programming. 
\newblock {\em Applied and Computational Harmonic Analysis}, 30(1):20--36, 2011. 

\bibitem[Sin18]{Singer18}
Amit Singer.
\newblock Mathematics for cryo-electron microscopy.
\newblock In {\em Proceedings of the International Congress of Mathematicians (ICM)}, pages 3988--4000. World Scientific, 2018.

\bibitem[SS11]{SS11}
Amit Singer and Yoel Shkolnisky. 
\newblock Three-dimensional structure determination from common lines in cryo-EM by eigenvectors and semidefinite programming. 
\newblock {\em SIAM Journal on Imaging Sciences}, 4(2):543--572, 2011.

\bibitem[SN97]{SN97}
Tom A.B. Snijders and Krzysztof Nowicki. 
\newblock Estimation and prediction for stochastic block models for graphs with latent block structure. 
\newblock {\em Journal of Classification}, 14(1):75--100, 1997.

\bibitem[SW25]{SW25}
Youngtak Sohn and Alexander S. Wein.
\newblock Sharp phase transitions in estimation with low-degree polynomials.
\newblock In {\em Proceedings of the 57th Annual ACM Symposium on Theory of Computing (STOC)}, pages 891--902. ACM, 2025.

\bibitem[VMGP16]{VMGP16}
Toni Valles-Catala, Francesco A. Massucci, Roger Guimera, and Marta Sales-Pardo. 
\newblock Multilayer stochastic block models reveal the multilayer structure of complex networks. 
\newblock {\em Physical Review X}, 6(1):011036, 2016.

\bibitem[Ver18]{Vershynin18}
Roman Vershynin.
\newblock \emph{High-Dimensional Probability: An Introduction with Applications in Data Science}.
\newblock Cambridge university press, 2018.

\bibitem[VA15]{VA15}
Nicolas Verzelen and Ery Arias-Castro. 
\newblock Community detection in sparse random networks.
\newblock {\em Annals of Applied Probability}, 25(6):3465--3510, 2015.


\bibitem[Wein22]{Wein22}
Alexander S. Wein. 
\newblock Optimal low-degree hardness of maximum independent set. 
\newblock {\em Mathematical Statistics and Learning}, 4(3-4):221--251, 2022.

\bibitem[Wein23]{Wein23}
Alexander S. Wein. 
\newblock Average-case complexity of tensor decomposition for low-degree polynomials. 
\newblock In {\em Proceedings of the 55th Annual ACM Symposium on Theory of Computing (STOC)}, pages 1685--1698. ACM, 2023.

\bibitem[Wein25+]{Wein25+}
Alexander S. Wein.
\newblock Computational complexity of statistics: New insights from low-degree polynomials.
\newblock arXiv preprint, arXiv:2506.10748.

\bibitem[WEM19]{WEM19}
Alexander S. Wein, Ahmed El Alaoui, and Cristopher Moore. 
\newblock The Kikuchi hierarchy and tensor PCA. 
\newblock In {\em Proceedings of the IEEE 60th Annual Symposium on Foundations of Computer Science (FOCS)}, pages 1446--1468. IEEE, 2019.

\bibitem[YLS25]{YLS25}
Xiaodong Yang, Buyu Lin, and Subhabrata Sen.
\newblock Fundamental limits of community detection from multi-view data: Multi-layer, dynamic and partially labeled block models.
\newblock {\em Annals of Statistics}, 53(6):2728--2756, 2025.

\bibitem[YWF25]{YWF25}
Kaylee Y. Yang, Timothy L.H. Wee, and Zhou Fan.
\newblock Asymptotic mutual information in quadratic estimation problems over compact groups.
\newblock {\em Information and Inference: A Journal of the IMA}, 14(3): iaaf024, 2025. 

\bibitem[ZSWB22]{ZSWB22}
Ilias Zadik, Min Jae Song, Alexander S. Wein, and Joan Bruna. 
\newblock Lattice-based methods surpass sum-of-squares in clustering. 
\newblock In {\em Proceedings of the 35th Annual Conference on Learning Theory (COLT)}, pages 1247--1248. PMLR, 2022.

\bibitem[ZHT06]{ZHT06}
Hui Zou, Trevor Hastie, and Robert Tibshirani. 
\newblock Sparse principal component analysis. 
\newblock {\em Journal of Computational and Graphical Statistics}, 15(2):186--265, 2006. 

\end{thebibliography}

\end{document}